\documentclass[a4paper,12pt]{article}
\usepackage[latin1]{inputenc}
\usepackage{amsmath,amssymb,amsthm,fullpage}
\usepackage{graphicx}
\usepackage{float}
\usepackage{floatpag}
\usepackage{booktabs}
\usepackage{url}

\usepackage{listings}
\usepackage[usenames,dvipsnames]{color}
\DeclareMathOperator{\spn}{span}
\DeclareMathOperator{\proj}{proj}
\newcommand{\norm}[1]{\left\| #1 \right\|}
\newcommand{\ip}[2]{\left\langle #1,#2\right\rangle}
\newcommand{\comments}[1]{}

\newcommand{\vbl}[1]{\textit{\texttt{#1}}}

\newcommand{\pd}[2]{\frac{\partial #1}{\partial #2}}
\newcommand{\vect}[1]{\boldsymbol{#1}}

\newtheorem{theorem}{Theorem}[section]

\newtheorem{lemma}[theorem]{Lemma}

\theoremstyle{definition}
\newtheorem{definition}[theorem]{Definition}
\newtheorem{example}[theorem]{Example}
\newtheorem{algorithm}[theorem]{Algorithm}
\newtheorem{remark}{Remark}

\usepackage[normalem]{ulem}

\newcommand {\Z} {\mathbb{Z}}
\newcommand {\R} {\mathbb{R}}
\newcommand {\N} {\mathbb{N}}
\newcommand {\range} {\mathcal{R}}
\def\mapright#1{\smash{\mathop{\mapsto}\limits^{#1}}}
\usepackage{psfrag}
\usepackage{epsfig}

\title{Computing covariant vectors, Lyapunov vectors, Oseledets vectors, and dichotomy projectors:  a comparative numerical study}
\author{Gary Froyland, Thorsten H\"{u}ls, Gary P.\ Morriss, and Thomas M.\ Watson}

\begin{document}
\maketitle
\begin{abstract}
Covariant vectors, Lyapunov vectors, or Oseledets vectors are increasingly being used for a variety of model analyses in areas such as partial differential equations, nonautonomous differentiable dynamical systems, and random dynamical systems.  These vectors identify spatially varying directions of specific asymptotic growth rates and obey equivariance principles. In recent years new computational methods for approximating Oseledets vectors have been developed, motivated by increasing model complexity and greater demands for accuracy. In this numerical study we introduce two new approaches based on singular value decomposition and exponential dichotomies and comparatively review and improve two recent popular approaches of Ginelli \emph{et al.} \cite{Ginelli2007} and Wolfe and Samelson \cite{Wolfe2007}. We compare the performance of the four approaches via three case studies with very different dynamics in terms of symmetry, spectral separation, and dimension. We also investigate which methods perform well with limited data.
\end{abstract}
\section{Introduction}

The asymptotic behaviour of a linear ODE $\dot{x}(t)=Ax(t)$, $x(t)\in \mathbb{R}^d$ is completely determined by the spectral properties of the $d\times d$ matrix $A$.
Similarly, the long-term behaviour of a nonlinear ODE $\dot{x}(t)=f(x(t))$ in a small neighbourhood of a fixed point $x_0$, for which $f(x_0)=0$, is completely determined by the spectral properties of the linearisation of $f$ at $x_0$.
Well-known extensions of these facts can be constructed when $x_0$ is periodic via Floquet theory. However, for general time-dependent linear ODEs $\dot{x}(t)=A(t)x(t)$, the eigenvalues of $A(t)$ contain no useful information about the asymptotic behaviour as the simple example of \cite[p. 30]{Colonius2000} illustrates.  On the other hand, if the $A(t)$ are generated by a process with well-defined statistics, there is a good spectral theory for the system $\dot{x}(t)=A(t)x(t)$, and this is the content of the celebrated Oseledets Multiplicative Ergodic Theorem (MET) (\cite{Oseledec1968}, see also Arnold \cite{Arnoldbook} for a thorough treatment), which we state and explain shortly.
The ``well-defined statistics'' are often generated by some underlying (typically ergodic) dynamical system.

For clarity of exposition, we will discuss discrete-time dynamics;  it is trivial  to convert a continuous-time system to a discrete-time system by creating eg.\ time-1 maps flowing from time $t$ to time $t+1$.
Let $X$ denote our \emph{base space}, the space on which the underlying process that controls the time-dependence of the matrices $A$ occurs.
As we will place a probability measure on $X$, we formally need a $\sigma$-algebra $\mathfrak{X}$ of sets that we can measure\footnote{for example, if $X$ is a topological space, we can set $\mathfrak{X}$ to be the standard Borel $\sigma$-algebra generated by open sets on $X$.}.
We denote the underlying process on $X$ by $T:X\circlearrowleft$ and assume that $T$ is invertible.
One formally requires that $T$ is measurable\footnote{if $X$ is a topological space, and $T$ is continuous, then $T$ is measurable with respect to the standard Borel $\sigma$-algebra generated by open sets.} with respect to $\mathfrak{X}$.
The ``well-defined statistics'' are captured by a \emph{$T$-invariant probability measure} $\mu$ on $X$;  that is, $\mu=\mu\circ T^{-1}$, and we say that $T$ \emph{preserves} $\mu$.
Finally, it is common to assume that the underlying process is \emph{ergodic}, which means that any subsets $X'\in\mathfrak{X}$ of $X$ that are invariant ($T^{-1}(X')=X'$, implying that trajectories beginning in $X'$ stay in $X'$ forever in forward and backward time) have either $\mu$-measure 0 (they are trivial), or $\mu$-measure 1 (up to sets of $\mu$-measure 0 they are all of $X$).

Now we come to the matrices $A$, which are generated by a measurable matrix-valued function $A:X\to M_d(\mathbb{R})$, where $M_d(\mathbb{R})$ is the space of $d\times d$ real matrices.
We choose an initial $x\in X$ and begin iterating $T$ to produce an orbit $x,Tx,T^2x,\ldots$.
Concurrently, we multiply $\cdots A(T^2x)\cdot A(Tx)\cdot A(x)$, and we are interested in the asymptotic behaviour of this matrix product.
In particular, we are interested in (i) the growth rates
\[\lambda(x,v):=\lim_{n\to\infty} \frac{1}{n}\log\|A(T^{n-1}x)\cdots  A(Tx)\cdot A(x)v\|\]
as $v$ varies in $\mathbb{R}^d$ and (ii) the subspaces $W(x)\subset \mathbb{R}^d$ on which the various growth rates occur.  Throughout, $\norm{\cdot}$ denotes the standard Euclidean vector norm or the associated matrix operator norm $\norm{A} = \max_{\norm{v}=1}\norm{Av}$; whether $\norm{\cdot}$ is a vector or matrix norm will be clear from the context.
Surprisingly, the ``well-defined statistics'' and ergodicity ensures that these limits exist, and that there are at most $d$ different values $\lambda_1>\lambda_2>\cdots>\lambda_\ell\ge -\infty$ that $\lambda(x,v)$ can take, as $v$ varies over $\mathbb{R}^d$ and $x$ varies over $\mu$-almost all of $X$ (note we allow $\lambda_\ell=-\infty$ to include the case of non-invertible $A$).
We can also decompose $\mathbb{R}^d$ pointwise in $X$ as $\mathbb{R}^d=\bigoplus_{i=1}^\ell W_i(x)$, where for all $v\in W_i(x)\setminus\{0\}$, one has
$$\lim_{n\to \infty}\frac{1}{n} \log\|A(T^{n-1}x)\cdots A(x)v\|=\lambda_i.$$
The subspaces $W_i$ are \emph{equivariant} (or \emph{covariant}) with respect to $A$ over $T$;  that is, they satisfy
$$W_i(Tx)=A(x)W_i(x)$$ for $1\le i<\ell$.


We use the following stronger version of the MET, which guarantees an Oseledets splitting even when the matrices $A$ are non-invertible.

\begin{theorem}[\cite{Froyland2010}, Theorem 4.1]
Let $T$ be an invertible ergodic measure-preserving transformation of the probability space $(X,\mathfrak{X},\mu)$.
Let $A:X\to M_d(\mathbb{R})$ be a measurable family of matrices satisfying
$$\int \log^+\|A(x)\|\ d\mu(x)<\infty.$$
Then there exist $\lambda_1>\lambda_2>\cdots >\lambda_\ell\ge -\infty$ and dimensions $m_1,\ldots,m_\ell$ with $m_1+\cdots+m_\ell=d$, and a measurable family of subspaces $W_i(x)\subset \mathbb{R}^d$ such that for $\mu$-almost every $x\in X$, the following hold.
\begin{enumerate}
\item $\dim W_i(x)=m_i$,
\item $\mathbb{R}^d=\bigoplus_{i=1}^\ell W_i(x)$,
\item $A(x)W_i(x)\subset W_i(Tx)\mbox{ with equality if $\lambda_i>-\infty$,}$
\item For all $v\in W_i(x)\setminus\{0\}$, one has
$$\lim_{n\to \infty}\frac{1}{n} \log\|A(T^{n-1}x)\cdots A(x)v\|=\lambda_i.$$
\end{enumerate}
\end{theorem}

The range of applications of the MET to the analysis of dynamical systems is vast.
Below, we mention just of few of the settings in which the MET is used.

\begin{example}
\label{eg1.2}\quad
\begin{enumerate}
\item \textbf{Differentiable dynamics: } One of the first applications of the MET was to differentiable dynamical systems $T:X\circlearrowleft$ on smooth $d$-dimensional compact manifolds. The matrix function $A$ is the spatial derivative of $T$, denoted $DT$.  The space $\mathbb{R}^d$ is associated with the tangent space of $X$ and the equivariance condition becomes $W_i(Tx)=DT(x)\cdot W_i(x)$.  If $T$ is uniformly hyperbolic, $\bigoplus_{i:\lambda_i>0}W_i(x)=W^u(x)$, the unstable subspace at $x\in X$ and $\bigoplus_{i:\lambda_i<0}W_i(x)=W^s(x)$, the stable subspace at $x$. The spaces $W_i(x)$ provide a refinement of $W^u(x)$ and $W^s(x)$ into subspaces with different growth rates.

\item \textbf{Hard disk system:} Consider a fixed number $N$ of hard disks in a region $L_{x} \times L_{y}$ moving freely between collisions. In each collision a pair of disks change their velocities \cite{Morriss2009}. The region may be finite (hard walls) or periodic (toroidal) in either coordinate direction. The quasi-one-dimensional system studied here is a two-dimensional system with $L_{y}$ less than twice the particle diameter so that the disks remain ordered in the $x$ direction. Here $X = ([0,L_x]\times [0,L_y])^N \times \R^{2N}$ (with the appropriate equivalence classes depending on the choice of hard wall or toroidal boundary conditions) is the collection of $4N$-tuples containing all the coordinates and momenta of the $N$ particles.

The map $T:X\to X, x\mapsto \mathcal C \circ \mathcal F^{\tau(x)}(x)$ is the composition of a free-flow map $\mathcal F^{\tau(x)}$ and  a collision map $\mathcal C$. The free-flow map moves the disks in straight lines according to their momentum while none of the disks are colliding.  The time between collisions is the free-flow time $\tau(x)$ which depends on the initial condition $x\in X$.  Collisions occur when the boundary of two disks (or one disk and a wall) touch, and the collision map exchanges velocities along the direction of collision (since all disks are of equal mass).  Again, the matrix function $A$ is the spatial derivative of $T$, so that $A(x) = DT(x) = D\left(\mathcal C\circ \mathcal F^{\tau(x)}\right)(x)$.  Precise details may be found in \cite{Chung2010}.
\item \textbf{PDE:}  The Kuramoto-Sivashinski equation is a model for weakly turbulent fluids and flame fronts
\[\eta_{t} = (\eta^{2})_{x} - \eta_{xx} - \nu \eta_{xxxx},\]
where $\nu$ is a damping coefficient. Another familiar example is the complex Ginzburg-Landau equation
\[\eta_{t}=\eta-(1+i\beta) |\eta|^{2} \eta + (1+i\alpha) \eta_{xx}\]
where $\eta(x,t)$ is complex and $\alpha$ and $\beta$ are parameters. In both of these cases it is possible to approximate solutions of the partial differential equations using Fourier spectral methods (see \cite{Cvitanovic2007} for details).  For instance, in the case of the 1-dimensional Kuramoto-Sivashinski PDE we look for solutions of the form
\[\eta(x,t) = \sum_{k=-\infty}^\infty a_k(t)e^{ikx/\tilde L},\]
where $\tilde L$ is a unitless length parameter, then solve the following system of ODEs for the Fourier coefficients $a_k(t)$:
\[\dot a_k = (q_k^2 - q_k^4)a_k-i\frac{q_k}{2}\sum_{m=-\infty}^\infty a_ma_{k-m},\]
where $q_k = k/\tilde L$.  Since the $a_k$ decrease rapidly with $k$, truncating the above system of ODEs is justified.

In the setting of this review we treat $X$ as the space of Fourier coefficients $(a_1,\ldots,a_d)$ of the truncated PDE, and consider the transformation $T:X\to X$ defined by choosing some $\tau >0$ and letting $T\left((a_1,\ldots, a_d)\right) = (a_1(\tau),\ldots,a_d(\tau))$ where the $a_k(t)$ are solutions to the system of ODEs with initial conditions $a_k(0) = a_k$.  The matrix function $A$ is again the spatial derivative of $T$ so that
\[ A(x) = DT(x) = \begin{pmatrix} \pd{a_1(\tau)}{a_1} & \pd{a_1(\tau)}{a_2} & \cdots \\ \pd{a_2(\tau)}{a_1} & \pd{a_2(\tau)}{a_2} & \\ \vdots & & \ddots \end{pmatrix}.\]

\item \textbf{Nonautonomous ODEs and transfer operators:} Consider an autonomous ODE $\dot{x}(t)=f(x(t))$ on $X$ (for example, the Lorenz flow on $X=\mathbb{R}^3$), and its flow map $\xi(\tau,x)$ which flows the points $x$ forward $\tau$ time units.  We think of the coordinates $x$ as a ``generalised time'' and the ODE $\dot{x}(t)=f(x(t))$ is our base system.  We use this base ODE (the driving system) to construct a nonautonomous ODE or skew product ODE as $\dot{y}(t)=F(\xi(t,x),y(t))$. Given an initial time $t$ and a flow time $\tau$, one may construct finite-rank approximations $P^{(\tau)}_x(t)$ of the Perron-Frobenius operator $\mathcal{P}^{(\tau)}(x(t))$ that track the evolution of densities from base ``time'' $x(t)$ to $x(t+\tau)$;  see \cite{FLS10} for details.  The matrices $P^{(\tau)}_x(t)$ form a cocycle and Oseledets subspace computations enable the extraction of \emph{coherent sets} in the nonautonomous flow (see \cite{FLS10}).  Coherent sets are time-dependent analogues of almost-invariant sets for autonomous systems; see \cite{DJ99,FD03,F05}.  Finite-time constructions for coherent sets are described in \cite{FSM10}.
    In the setting of this review, $T:X\to X$ is defined as $T(x)=\xi(\tau,x)$, and $A(x)=P^{(\tau)}_x(t)$.
        \end{enumerate}
\end{example}

From now on, we denote $A(T^{n-1}x)\cdots A(x)$ as $A(x,n)$.
The proof of the classical MET \cite{Oseledec1968} proves that the matrix limit
\begin{align}
\Psi(x) = \lim_{n\to \infty} \left(\left(A(x,n)\right)^* A(x,n)\right)^{1/2n} \label{eq:limiting_matrix}
\end{align}
exists for $\mu$-almost all $x\in X$.
The matrix $\Psi(x)$ is symmetric, depends measurably on $x$, and its eigenvalues are $e^{\lambda_1(x)}>\cdots >e^{\lambda_\ell(x)}$.
The corresponding eigenspaces are denoted $U_1(x),\ldots,U_\ell(x)$ and one has
$$V_i(x):=\bigoplus_{j=i}^\ell W_j(x)=\bigoplus_{j=i}^\ell U_j(x).$$
Thus $V_i(x)$ captures growth rates from $\lambda_i(x)$ down to $\lambda_\ell(x)$;  the ``$U$''decomposition of $V_i(x)$ is orthogonal, while the ``$W$'' decomposition (the Oseledets splitting) is equivariant (or covariant).

An alternative notion of stability for non-autonomous systems is the so
called Sacker-Sell spectrum, cf.\ \cite{ss78}.
It is based on
exponential dichotomies, cf.\ \cite{co78, pa88} which we briefly
introduce for linear difference equations of the form

\begin{equation}\label{ap1}
w_{n+1} = A_n w_n,\quad n \in \Z,\quad  A_n\in M_d(\R) \text{ invertible}.
\end{equation}

In the current context we associate the sequence of matrices $\{A_n\}_{n\in \Z}$ with an invertible matrix cocycle over a single orbit, e.g. for some $x\in X$ let $A_n = A(T^nx)$.
We restrict the introduction of exponential dichotomies
to invertible systems only, and
note that a justification of our algorithm for computing dichotomy
projectors strongly depends on this assumption.
Theory defining exponential dichotomies for non-invertible matrices is contained in eg.\ \cite{as01}.
Numerical experiments indicate that Algorithms \ref{alg:dichproj1} and \ref{A2} also apply
in the non-invertible case, however, the corresponding analysis is a topic of
future research.

We denote by $\Phi$ the solution operator of \eqref{ap1},
defined as
\[
\Phi(n,m) :=
\begin{cases}
A_{n-1} \ldots A_m,\quad & \text{for }
n>m,\\
I, &\text{for } n = m,\\
A_n^{-1} \ldots  A_{m-1}^{-1},\quad&
\text{for }  n < m.
\end{cases}
\]

\begin{definition}\label{b14}
The linear difference equation \eqref{ap1}
has an \textbf{exponential dichotomy}
with data $(K,\alpha_s,\alpha_u,P_n^{s},P_n^u)$
on $J\subset \Z$, if there exist two families of projectors
$P_n^{s}$ and $P_n^u = I-P_n^{s}$ and constants
$K, \alpha_s, \alpha_u >0$, such that the following statements hold:
\begin{align}\label{ap2}
P_n^s \Phi(n,m) = \Phi(n,m)P_m^s \quad \forall n,m \in J,
\end{align}
\begin{align*}
\begin{array}{ccl}
\norm{\Phi(n,m)P_m^{s}}& \le Ke^{-\alpha_s(n-m)}
\vspace{2mm}\\
\norm{\Phi(m,n)P_n^u} & \le K e^{-\alpha_u(n-m)}
\end{array}
\quad \forall n \geq m,\ n,m\in J.
\end{align*}
\end{definition}

Consider the scaled equation
\begin{align}\label{i5}
w_{n+1} = {e^{-\lambda}} A_n w_n,\quad n \in \Z.
\end{align}

\begin{definition}
The \textbf{Sacker-Sell} or \textbf{dichotomy spectrum} is
defined as
\[
\sigma_{\text{ED}} := \{{\lambda \in \R}: \text{ \eqref{i5} has no
  exponential dichotomy on }\Z\}.
\]
The complementary
set $\R\setminus \sigma_{\text{ED}}$ is called the \textbf{resolvent
  set}.
\end{definition}

The Sacker-Sell spectrum consists of at most $d$
disjoint, closed intervals, where $d$ denotes the dimension of the
space, cf.\ \cite{ss78}, i.e.\ there exists an $\ell < d$ such that
\[
\sigma_{\text{ED}} = \bigcup_{i=1}^\ell {[\lambda_i^-,\lambda_i^+]},\quad
\text{where } {\lambda_{i+1}^+ < \lambda_i^-}\quad \text{for } i=1,\dots,\ell-1.
\]
It is well known that the Lyapunov spectrum, when it exists, is a subset of the
Sacker-Sell spectrum, see \cite{dv02}. While the Lyapunov spectrum
provides information on bounded solutions of $\Phi(n,0)$, $n\ge 0$, the
Sacker-Sell spectrum answers this question for
$\Phi(n,m)$, $n\ge m$.
These answers may be different for different initial $n$ because, in contrast to the MET setting, there is no \emph{a priori} stationarity assumption on a base dynamical system generating the matrix cocycle.
Note that for ${\lambda \in \R\setminus \sigma_{\text{ED}}}$ it
follows from \cite[Lemma 2.7]{pa88} that the inhomogeneous
equation $w_{n+1} = {e^{-\lambda}} A_n w_n + r_n$
has for every bounded sequence
$r_\mathbb Z$ a unique bounded solution on $\Z$.

Dichotomy projectors of the scaled equation \eqref{i5} are
constant in resolvent intervals $R_i := (\lambda_{i}^+,\lambda_{i-1}^-)$,
$i=1,\dots,\ell+1$, where $\lambda_0^- = \infty$ and $\lambda_{\ell+1}^+ =
-\infty$, see Figure \ref{figd2}.
We denote these families of projectors
by $(P_n^{i,s}, P_n^{i,u})$.
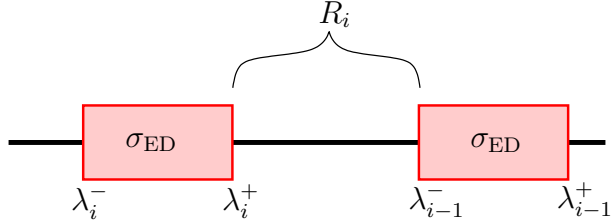
\begin{figure}[hbt]
 \begin{center}
\def\svgwidth{8cm}


\begingroup%
  \makeatletter%
  \providecommand\color[2][]{%
    \errmessage{(Inkscape) Color is used for the text in Inkscape, but the package 'color.sty' is not loaded}%
    \renewcommand\color[2][]{}%
  }%
  \providecommand\transparent[1]{%
    \errmessage{(Inkscape) Transparency is used (non-zero) for the text in Inkscape, but the package 'transparent.sty' is not loaded}%
    \renewcommand\transparent[1]{}%
  }%
  \providecommand\rotatebox[2]{#2}%
  \ifx\svgwidth\undefined%
    \setlength{\unitlength}{231bp}%
    \ifx\svgscale\undefined%
      \relax%
    \else%
      \setlength{\unitlength}{\unitlength * \real{\svgscale}}%
    \fi%
  \else%
    \setlength{\unitlength}{\svgwidth}%
  \fi%
  \global\let\svgwidth\undefined%
  \global\let\svgscale\undefined%
  \makeatother%
  \begin{picture}(1,0.37229437)%
    \put(0,0){\includegraphics[width=\unitlength]{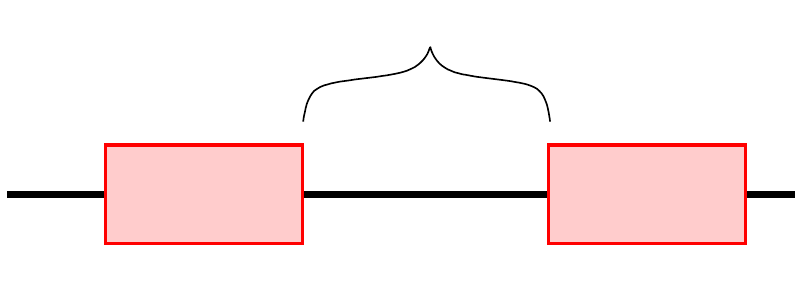}}%
    \put(0.51873885,0.32877866){\color[rgb]{0,0,0}\makebox(0,0)[lb]{\smash{$R_i$}}}%
    \put(0.20157354,0.1232276){\color[rgb]{0,0,0}\makebox(0,0)[lb]{\smash{$\sigma_\text{ED}$}}}%
    \put(0.76866788,0.12203145){\color[rgb]{0,0,0}\makebox(0,0)[lb]{\smash{$\sigma_\text{ED}$}}}%
    \put(0.11219926,0.01735393){\color[rgb]{0,0,0}\makebox(0,0)[lb]{\smash{$\lambda_i^-$}}}%
    \put(0.36057863,0.0182228){\color[rgb]{0,0,0}\makebox(0,0)[lb]{\smash{$\lambda_i^+$}}}%
    \put(0.66670524,0.01919838){\color[rgb]{0,0,0}\makebox(0,0)[lb]{\smash{$\lambda_{i-1}^-$}}}%
    \put(0.91099351,0.0208945){\color[rgb]{0,0,0}\makebox(0,0)[lb]{\smash{$\lambda_{i-1}^+$}}}%
  \end{picture}%
\endgroup%


\caption{Spectral setup.\label{figd2}}
\end{center}
\end{figure}

In analogy to the MET we obtain the family of subspaces
\begin{align*}
W_n^{i} = \range(P_n^{i,s}) \cap \range (P_n^{i+1,u}),\quad n
\in \Z,\quad i =1,\dots,\ell
\end{align*}
that decompose $\R^d$ for each $n\in\Z$
\begin{align*}
\R^d = \bigoplus_{i=1}^\ell W^i_n,
\end{align*}
and using the cocycle property
\eqref{ap2} it follows for all $i=1,\dots,\ell$ that
\begin{align*}
A_n W_n^i =  W_{n+1}^i,\quad n\in \Z.
\end{align*}
Furthermore, for each $w \in W_m^i$, there exists a constant $K = K(w)>0$
such that the following equations hold
{\begin{align*}
\norm{\Phi(n,m)w} & = K e^{\left(\lambda_i^+ + r^+_i(n-m)\right)(n-m)}, \quad \text{for } n \geq m, \quad \text{where } \limsup_{n\to\infty} r^+_i(n) = 0, \\
\norm{\Phi(n,m)w} & = K e^{\left(\lambda_i^- +r_i^-(n-m)\right)(n-m)},\quad \text{for } n<m.\quad \text{where }\limsup_{n\to\infty}r_i^-(n) = 0.
 \end{align*}}


When working with data over a finite time interval, one has access only to a finite sequence $A_0,A_1,\ldots,A_{n-1}$.
In this case, one either assumes there is an underlying ergodic process generating the sequence $A_0,A_1,\ldots,A_{n-1}$  or one considers exponential dichotomies.

An outline of the paper is as follows.
In Sections 2 and 3, we introduce two new methods for computing Oseledets vectors.
The first method is based on the proof of the generalised MET in \cite{Froyland2010} and is particularly simple to implement and fast to execute.
The second method is an adaptation of an approach to compute dichotomy projectors \cite{hu08}.
In Section 4 we review the approaches by Ginelli \emph{et al} \cite{Ginelli2007} and Wolfe and Samelson \cite{Wolfe2007}.
In Sections 2, 3, and 4 we provide MATLAB code snippets to implement the algorithms presented.
Section 5 contains numerical comparisons of the performance of the four methods on three dynamical systems.
The first case study is a dynamical systems formed via composition of a sequence of $8\times 8$ matrices constructed so that all Oseledets vectors are known at time 0;  we thus compare the accuracy of the methods \emph{exactly} in this case study.
The second case study is an eight-dimensional system generated by two hard disks in a quasi-one-dimensional box.
The third case study is a nonlinear model of time-dependent fluid flow in a cylinder;  the matrices are generated by finite-rank approximations of the corresponding time-dependent transfer operators.
The three case studies have been chosen to represent a cross-section of a variety of features of systems that either help or hinder the computation of Oseledets vectors, and we draw out the advantages and disadvantages of each of the four methods considered.

\section{An SVD-based approach\label{sec:svd}}
\label{sec:SVD}

The approach outlined in this section is simple to execute and exhibits quick convergence.  However, as the length of the sample orbit becomes too large this approach fails.

In \cite[proof of Theorem 4.1]{Froyland2010}  it is proven that the limit
\begin{align*}
  \lim_{N \to \infty}A(T^{-N}x,N)U_i(T^{-N}x)
\end{align*}
exists and is equal to the $i$th Oseledets subspace $W_i(x)$.
That is, if one computes $U_i$ in the far past and pushes forward to the present, the result is a subspace close to $W_i(x)$.
Thus, the strategy in \cite{Froyland2010} is to first estimate $U_i$ in the past and push forward.



The numerical method of approximating $W_j(x)$, $x\in X$, is implemented in the following steps:
\begin{algorithm}[To estimate $W_j(x)$]
\label{alg:svd}
\quad
\begin{enumerate}
\item \label{item:SVD_step_1}Choose $M,N >0$ and form the matrix
  \begin{equation}
  \label{PsiM}
    \Psi^{(M)}(T^{-N}x) = \left(A(T^{-N}x,M)^* A(T^{-N}x,M)\right)^{1/2M}
  \end{equation}
as an approximation of \eqref{eq:limiting_matrix} at $T^{-N}x \in X$.
\item \label{item:SVD_step_2}Compute $U_j^{(M)}(T^{-N}x)$, the $j$th orthonormal eigenspace of $\Psi^{(M)}(T^{-N}x)$ as an approximation of $U_j(T^{-N}x)$.
\item \label{item:SVD_step_3}Define $W_j^{(M,N)}(x) = A(T^{-N}x,N)U_j^{(M)}(T^{-N}x)$, approximating the Oseledets subspace $W_j(x)$.
\end{enumerate}
\end{algorithm}

Listing \ref{lst:svd_code} shows part of a MATLAB implementation of Algorithm \ref{alg:svd}.  The array \vbl{A}$ = \left[A(T^{-N}x) \ | \ A(T^{-N+1}x) \ | \cdots | \ A(T^{M-1}x)\right]$ contains the $d\times d$ matrices which generate the cocycle $A:X\times \mathbb Z^+ \to M_d(\mathbb R)$, and the matrix \vbl{Psi} is formed by multiplying the matrices contained in \vbl{A}.  Step \ref{item:SVD_step_1} of Algorithm \ref{alg:svd} is performed prior to the code in Listing \ref{lst:svd_code}, Step \ref{item:SVD_step_2} is performed in lines 1-3 and lines 4-7 perform Step \ref{item:SVD_step_3}.  The function returns \vbl{Wj} as its estimate to $W_j(x)$.

\begin{figure}[hbt]
  \begin{lstlisting}[label = lst:svd_code, caption = Sample MATLAB code of Algorithm \ref{alg:svd} to approximate $W_j(x)$]
[~,s,v] = svd(Psi);
[~,p] = sort(diag(s),'descend');
Wj = v(:,p(j))/norm(v(:,p(j)));
for h = 1:N
  Wj = A(:,(h-1)*dim+1:h*dim)*Wj;
  Wj = Wj/norm(Wj);
end
  \end{lstlisting}
\end{figure}

The values of $M$ and $N$ can be chosen with relative freedom and in our examples that follow we have chosen $M= 2N$ to compute over a time window centred on $x$, from $T^{-N}x$ to $T^Nx$.  Unfortunately, we cannot choose $M$ and $N$ arbitrarily large and expect accurate results.  If $A(T^{-N}x,M)$ is constructed via the product
\begin{equation}
A(T^{-N}x,M) = A(T^{M-N-1}x) \cdots A(T^{-N+1}x)A(T^{-N}x) \notag
\end{equation}
then with larger $M$ the numerical inaccuracies of matrix multiplication compound and this product becomes more singular and thus a poorer approximation of $A(T^{-N}x,M)$.  Because of this, $\Psi^{(M)}(T^{-N}x)$ cannot be expected to accurately approximate $\Psi(T^{-N}x)$ for large $M$.  However, even if we suppose $\Psi^{(M)}(T^{-N}x)$ accurately approximates $\Psi(T^{-N}x)$, the small, but non-zero, difference in $U^{(M)}_j(T^{-N}x)$ and $U_j(T^{-N}x)$ grows roughly as $O\left(e^{N(\lambda_1-\lambda_j)}\right)$ during the push-forward in step \ref{item:SVD_step_3} above.  For these reasons $M$ and $N$ must be chosen carefully.

\subsection{Improving the basic SVD-based approach\label{sec:svd2}}

We present a simple improvement that can overcome one of the sources of numerical instability, namely the push-forward process in step \ref{item:SVD_step_3} above.

\begin{figure}[hbt]
  \centering
  \def\svgwidth{\textwidth}
  
\begingroup%
  \makeatletter%
  \providecommand\color[2][]{%
    \errmessage{(Inkscape) Color is used for the text in Inkscape, but the package 'color.sty' is not loaded}%
    \renewcommand\color[2][]{}%
  }%
  \providecommand\transparent[1]{%
    \errmessage{(Inkscape) Transparency is used (non-zero) for the text in Inkscape, but the package 'transparent.sty' is not loaded}%
    \renewcommand\transparent[1]{}%
  }%
  \providecommand\rotatebox[2]{#2}%
  \ifx\svgwidth\undefined%
    \setlength{\unitlength}{482.59541016bp}%
    \ifx\svgscale\undefined%
      \relax%
    \else%
      \setlength{\unitlength}{\unitlength * \real{\svgscale}}%
    \fi%
  \else%
    \setlength{\unitlength}{\svgwidth}%
  \fi%
  \global\let\svgwidth\undefined%
  \global\let\svgscale\undefined%
  \makeatother%
  \begin{picture}(1,0.34774953)%
    \put(0,0){\includegraphics[width=\unitlength]{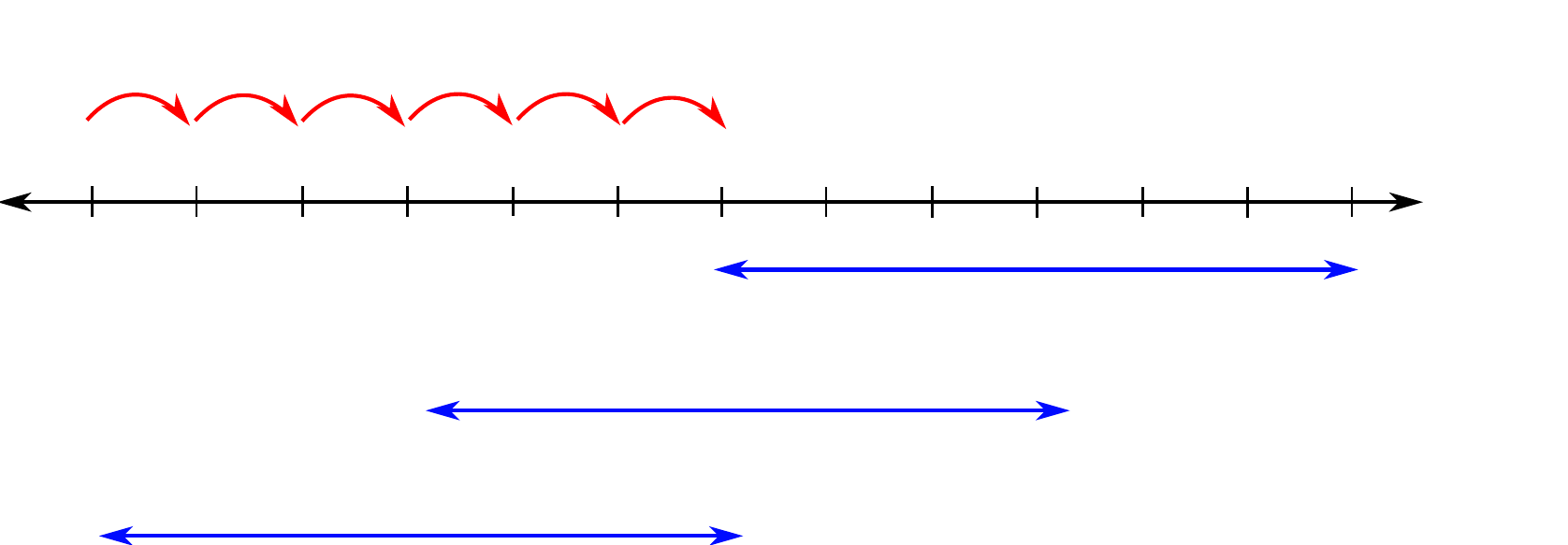}}%
    \put(0.45110571,0.23849567){\color[rgb]{0,0,0}\makebox(0,0)[lb]{\smash{$x$}}}%
    \put(0.2231491,0.23740918){\color[rgb]{0,0,0}\makebox(0,0)[lb]{\smash{$T^{-N_k}x$}}}%
    \put(0.14882441,0.23941571){\color[rgb]{0,0,0}\makebox(0,0)[lb]{\smash{$\cdots$}}}%
    \put(0.02371107,0.2390941){\color[rgb]{0,0,0}\makebox(0,0)[lb]{\smash{$T^{-N_1}x$}}}%
    \put(0.84400866,0.23795163){\color[rgb]{0,0,0}\makebox(0,0)[lb]{\smash{$T^Mx$}}}%
    \put(0.36083317,0.2371375){\color[rgb]{0,0,0}\makebox(0,0)[lb]{\smash{$\cdots$}}}%
    \put(0.40308832,0.04609009){\color[rgb]{0,0,1}\makebox(0,0)[lb]{\smash{\reflectbox{$\ddots$}}}}%
    \put(0.6467707,0.18438988){\color[rgb]{0,0,1}\makebox(0,0)[lb]{\smash{$\Psi^{(M)}(x)$}}}%
    \put(0.43269017,0.09558435){\color[rgb]{0,0,1}\makebox(0,0)[lb]{\smash{$\Psi^{(M)}(T^{-N_k}x)$}}}%
    \put(0.21624146,0.01412008){\color[rgb]{0,0,1}\makebox(0,0)[lb]{\smash{$\Psi^{(M)}(T^{-N_1}x)$}}}%
    \put(0.61312334,0.14452571){\color[rgb]{0,0,1}\makebox(0,0)[lb]{\smash{\reflectbox{$\ddots$}}}}%
    \put(-0.01157433,0.30611268){\color[rgb]{1,0,0}\makebox(0,0)[lb]{\smash{$U_j^{(M)}(T^{-N_1}x)$}}}%
    \put(0.19208461,0.30755113){\color[rgb]{1,0,0}\makebox(0,0)[lb]{\smash{$W_j^{(M,N_k)}(T^{-N_k}x)$}}}%
    \put(0.14520504,0.30610226){\color[rgb]{1,0,0}\makebox(0,0)[lb]{\smash{$\cdots$}}}%
    \put(0.43463711,0.30739317){\color[rgb]{1,0,0}\makebox(0,0)[lb]{\smash{$W_j^{(M,0)}(x)$}}}%
    \put(0.39039169,0.3059177){\color[rgb]{1,0,0}\makebox(0,0)[lb]{\smash{$\cdots$}}}%
  \end{picture}%
\endgroup%

  \caption{Schematic of the re-orthogonalisation described in Section \ref{sec:svd2}.  The black line represents the orbit centred at $x\in X$ and the points $T^{-N_k}x$ are those points at which we ensure orthogonality with the subspaces $V_j(T^{-N_k}x)^\perp$.  To do this we use the (blue) approximations $\Psi^{(M)}(T^{-N_k}x)$ to approximate $V_j(T^{-N_k}x)^\perp$ and perform the (red) push-forward and orthoganlisation steps starting with $U_j^{(M)}(T^{-N_1}x)$ and ending with $W_j^{(M,0)}(x)$ (see Algorithm \ref{alg:svd2}).}
  \label{fig:svd2_schematic}
\end{figure}
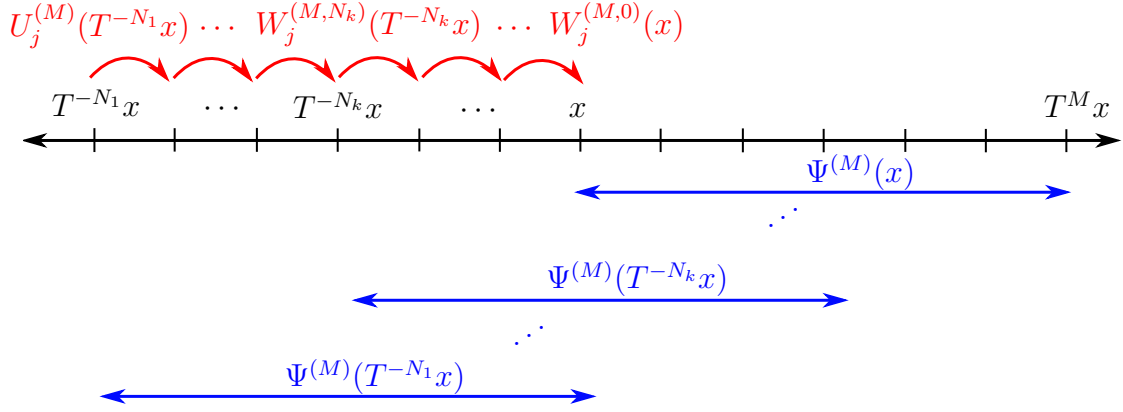

Recall that the subspace $V_j(x) = U_j(x) \oplus \cdots \oplus U_\ell(x)=\left(U_1(x)\oplus\cdots \oplus U_{j-1}(x)\right)^\perp$ is $A$-invariant and that for $v\in V_j(x) \backslash V_{j+1}(x)$ (with $V_{\ell+1}(x) = \{0\}$) we have
\begin{align}
  \lambda_j(x) = \lim_{n\to\infty}\frac{1}{n}\log\norm{A(x,n)v}. \notag
\end{align}

The subspace  $V_j(x)$ contains $W_j(x), W_{j+1}(x),\ldots, W_\ell(x)$, and so the Oseledets subspace $W_j(x)$ is necessarily perpendicular to all $U_1(x),\ldots,U_{j-1}(x)$.  To solve the numerical instability of step \ref{item:SVD_step_3} we enforce this condition periodically.

The amended algorithm is implemented as follows:
\begin{algorithm}[To estimate $W_j(x)$]
\sloppy
\label{alg:svd2}
\quad
\begin{enumerate}
\item \label{item:svd2_s1}Choose $M,N_1>N_2>\cdots > N_n=0$ and form the matrices
  \begin{align}
    \Psi^{(M)}(T^{-N_k}x) = \left(A(T^{-N_k}x,M)^* A(T^{-N_k}x,M)\right)^{1/2M},\quad k=1,\ldots,n.\notag
  \end{align}
\item \label{item:svd2_s2}Compute all the orthonormal eigenspaces $U^{(M)}_i(T^{-N_k}x)$, $i=1,\ldots,j-1$ of (\ref{PsiM}) (replacing $N$ with $N_k$ in (\ref{PsiM})) and the eigenspace $U^{(M)}_j(T^{-N_1}x)$.
\item \label{item:svd2_s3}
Let $\proj_V:\mathbb R^d \to \mathbb R^d$ be the orthogonal projection onto the subspace $V$ so that $\ker\left(\proj_V\right) = V^\perp$ and $V^{(M)}_j(x) = \left(U^{(M)}_1(x) \oplus \cdots \oplus U^{(M)}_{j-1}(x)\right)^\perp$. Define $W^{(M,N_1)}_j(T^{-N_1}x) = U_j^{(M)}(T^{-N_1}x)$, and then define iteratively by pushing forward and taking orthogonal projections:
  \begin{align}
W^{(M,N_{k+1})}_j\left(T^{-N_{k+1}}x\right) = \proj_{V_j^{(M)}\left(T^{-N_{k+1}}x\right)} \left(A\left(T^{-N_k}x, N_{k+1} - N_k\right) W^{(M,N_k)}_j\left(T^{-N_k}x\right)\right) \notag
\end{align}
\item $W^{(M,N_n)}_j(x)=W^{(M,0)}_j(x)$ is our approximation of $W_j(x)$.
\end{enumerate}
\end{algorithm}

Listing \ref{lst:svd2} shows an example implementation of Algorithm \ref{alg:svd2} in MATLAB.  Lines 1-18 are responsible for performing Steps \ref{item:svd2_s1} and \ref{item:svd2_s2}, whilst the push forward procedure of Step \ref{item:svd2_s3} is performed in lines 20-30.  Again, the matrix cocycle is stored in \vbl{A}$=\left[ A(T^{-N}x) \ \right|\allowbreak \left. \ A(T^{N-1}x) \ \right|\allowbreak \left. \cdots \right|\allowbreak \left. \ A(T^{M-1}x)\right]$ and the function returns \vbl{Wj} as its approximation of $W_j(x)$.  The variable \vbl{Nk} is a one-dimensional array containing the elements of $\{N_k\}$ and is counted by \vbl{k}.

\begin{figure}[!h]
  \begin{lstlisting}[label=lst:svd2, caption = Sample MATLAB code of Algorithm \ref{alg:svd2} to approximate $W_j(x)$]
k=0;
for n = Nk,
  Psi = eye(dim);
  for i=0:M-1,
    Psi = A(:,(n + i - 1)*dim + 1:(n + i)*dim)*Psi;
    Psi = Psi/normest(Psi);
  end
  [~,s,u] = svd(Psi);
  [~,p] = sort(diag(s),'descend');
  if n==1,
    Wj = u(:,p(j))/norm(u(:,p(j)));
  else
    for i = 1:j-1,
      k = k+1;
      U(:,i,k) = u(:,p(i))/norm(u(:,p(i)));
    end
  end
end
k=0;
for n = 1:N,
  Wj = A(:,(n-1)*dim+1:n*dim)*Wj;
  Wj = Wj/norm(Wj);
  if any(Nk == n+1),
    k = k+1;
    for i = 1:j-1,
      Wj = Wj - dot(Wj,U(:,i,k))*U(:,i,k);
      Wj = Wj/norm(Wj);
    end
  end
end
  \end{lstlisting}
\end{figure}

\begin{remark}
\label{svdbad}
Unfortunately, some numerical issues with this approach remain.  They stem primarily from the long multiplication involved in building the variable \vbl{Psi} of Listings \ref{lst:svd_code} and \ref{lst:svd2}.  This results in \vbl{Psi} becoming too singular and hence $U^{(M)}_j(T^{-n}x)$ ($j\neq 1$) poorly approximates $U_j(T^{-n}x)$.  As can be seen in Section \ref{sec:numerical}, Algorithm \ref{alg:svd2} works superbly for $W_2(x)$ as $U^{(M)}_1(T^{-n}x)$ is well approximated for large $n$.  However when estimating $W_j(x)$, $j>2$, a good estimate of $U^{(M)}_{j-1}(x)$ is required for an accurate projection $\proj_{V^{(M)}_j}$ and for large $n$ such an estimate becomes unreliable.
\end{remark}

\section{A dichotomy projector approach\label{sec:dich_proj}}

We derive an approach for the computation of a vector
 $w^j_n \in W^j_n = \range(P^{j,s}_n)\cap \range (P^{j+1,u}_n)$.
For this task, we first need a guess of ${\Lambda^\text{right} \in R_{i}}$ and of
${\Lambda^\text{left} \in R_{i+1}}$ in two neighbouring resolvent intervals that
lie close to the common spectral interval, see
Figure \ref{figd1}.
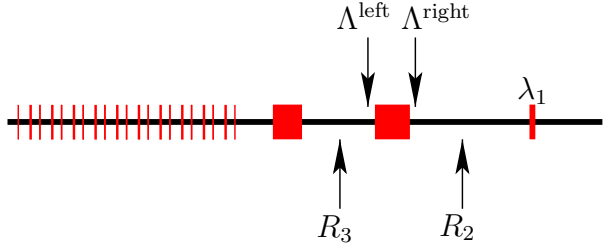
\begin{figure}[hbt]
 \begin{center}
\def\svgwidth{8cm}
\begingroup%
  \makeatletter%
  \providecommand\color[2][]{%
    \errmessage{(Inkscape) Color is used for the text in Inkscape, but the package 'color.sty' is not loaded}%
    \renewcommand\color[2][]{}%
  }%
  \providecommand\transparent[1]{%
    \errmessage{(Inkscape) Transparency is used (non-zero) for the text in Inkscape, but the package 'transparent.sty' is not loaded}%
    \renewcommand\transparent[1]{}%
  }%
  \providecommand\rotatebox[2]{#2}%
  \ifx\svgwidth\undefined%
    \setlength{\unitlength}{197bp}%
    \ifx\svgscale\undefined%
      \relax%
    \else%
      \setlength{\unitlength}{\unitlength * \real{\svgscale}}%
    \fi%
  \else%
    \setlength{\unitlength}{\svgwidth}%
  \fi%
  \global\let\svgwidth\undefined%
  \global\let\svgscale\undefined%
  \makeatother%
  \begin{picture}(1,0.41624365)%
    \put(0,0){\includegraphics[width=\unitlength]{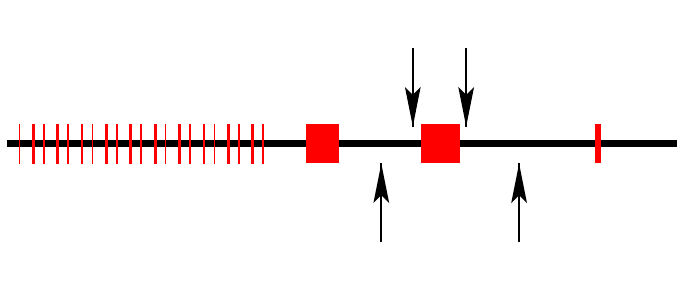}}%
    \put(0.55189557,0.36020153){\color[rgb]{0,0,0}\makebox(0,0)[lb]{\smash{$\Lambda^\text{left}$}}}%
    \put(0.65799593,0.35946376){\color[rgb]{0,0,0}\makebox(0,0)[lb]{\smash{$\Lambda^\text{right}$}}}%
    \put(0.51889174,0.01511336){\color[rgb]{0,0,0}\makebox(0,0)[lb]{\smash{$R_3$}}}%
    \put(0.72040306,0.01763223){\color[rgb]{0,0,0}\makebox(0,0)[lb]{\smash{$R_2$}}}%
    \put(0.8488665,0.2443325){\color[rgb]{0,0,0}\makebox(0,0)[lb]{\smash{$\lambda_1$}}}%
  \end{picture}%
\endgroup%
\caption{Choice of {$\Lambda^\text{left/right}$} in case $i = 2$.\label{figd1}}
\end{center}
\end{figure}

Numerical experiments indicate
that we get the best results by
choosing $\Lambda^\text{right}$ and $\Lambda^\text{left}$ close to (but outside) the
second Sacker-Sell interval.
This conclusion is supported by theoretical estimates on the approximation error for Algorithm \ref{A2}, discussed at the end of Section 3.

The following observation from  \cite{hu08, hu09}  allows the computation of
dichotomy projectors by solving
\begin{align}\label{a2}
   w_{n+1}^i = A_n w_n^i + \delta_{n,m-1} e_i,\quad n \in \Z,\quad e_i
   \ i\text{-th unit vector.}
\end{align}
With Green's function, cf.\ \cite{pa88},
the unique bounded solution $w_\Z^i$ of \eqref{a2} has the explicit
form
\begin{align}\label{a4}
w_{n}^i = G(n,m)e_i,\ n \in \Z, \quad  \text{where} \quad
G(n,m) =
\begin{cases}
\Phi(n,m) P_m^s,\quad & n \ge m,\\
-\Phi(n,m) P_m^u,\quad & n < m,
\end{cases}
\end{align}
and consequently
\begin{align*}
P^s_m =
\begin{pmatrix}
  | & & | \\
  w^1_m & \cdots & w^d_m \\
  | & & |
\end{pmatrix}.
\end{align*}
Numerically, we approximate the unique bounded solution on $\Z$ by the
least squares solution of \eqref{a2} on a sufficiently long interval.
For an error analysis of this approximation process, we refer to
\cite[Theorem 4]{hu09}.

The algorithms that we propose in this section compute a vector
$w \in W_0^j$ in analogy to $W_j(x)$ in the previous sections. For simplicity, we restrict the representation
to the case $j=2$ and assume that $W_n^1$ and  $W_n^2$ are
one-dimensional subspaces.

\sloppy In the absence of information about the dichotomy intervals, one may proceed as follows.
Given a finite sequence of matrices, one can estimate a point in the spectral interval ${[\lambda_q^-,\lambda_q^+]}, q=1,2,3$ by computing the {(logarithmic)} growth rates of one-, two-, and three-dimensional subspaces using direct multiplication;  these growth rates should approximate {$\lambda_1, \lambda_1+\lambda_2,$ and $\lambda_1+\lambda_2+\lambda_3$}, respectively.
By taking {differences} to obtain estimates ${\hat{\lambda}_q}, q=1,2,3$, (the caret indicating estimated quantities) one should obtain values in the interior of ${[\lambda_q^-,\lambda_q^+]}, q=1,2,3$.
We then estimate ${\Lambda^\text{left}=\hat{\lambda}_2-(\hat{\lambda}_2-\hat{\lambda}_3)/10\lessapprox\lambda_2^-}$ and ${\Lambda^\text{right}=\hat{\lambda}_2-(\hat{\lambda}_2-\hat{\lambda}_1)/10\gtrapprox\lambda_2^+}$.

In the first step of our first algorithm, we compute a basis of the two-dimensional space $\range(P_0^{3,u})$.
Then, in the second step, we search for the direction $w$ in this subspace
that additionally lies in $\range(P_0^{2,s})$ and assure in this way
that $w \in \range(P_0^{2,s}) \cap \range(P_0^{3,u}) = W_0^2$.

\begin{algorithm}[A Dichotomy Projector approach to estimate $W_2(x)$ by computing $W_0^2$\label{alg:dichproj1}]
\quad
\begin{enumerate}
\item Suppose $N \in \N$ and consider $n\in [-N,N]\cap\Z$. Let $A_n = A(T^nx)$.

Solve the least squares problem
\begin{align}\label{al1}
    \tilde w^i_{n+1} &= {e^{-\Lambda^\text{left}}} A_n \tilde w_n^i + \delta_{n,-1}
r^i, \quad  n = -N,\ldots,N-1, \ i=1,2 \\
& \text{such that} \ \|(\tilde w^i_{-N},\ldots,\tilde w^i_N)\|_2 \  \text{is minimised,}\notag
\end{align}
where the $r^i$ are chosen at random and $\norm{\cdot}_2$ is the $\ell^2$-norm. Define $p^i: = A_{-1}\tilde w^{i}_{-1}$, $i = 1,2$.

\item
Solve for $\tilde w_{[0,N]}$ and {$\kappa$} the least squares problem
\begin{align}\label{al2}
\tilde w_{n+1} &= {e^{-\Lambda^\text{right}}} A_n \tilde w_n,\quad n =
0,\dots,N-1,\\
\tilde w_0 + {\kappa} p^1 + p^2 &= 0,\label{al3}\\
& \text{such that} \ \norm{(\tilde w_0,\ldots, \tilde w_{N},{\kappa})}_2 \ \text{is minimised}.\notag
\end{align}
Then $\tilde w_0$ is our approximation of $w_2(x) \in W_2(x)$.
\end{enumerate}
\end{algorithm}

The unique bounded solutions on $\Z$ of these two steps satisfy
$p^1,p^2 \in \range(P_0^{3,u})$
and these vectors are generically linear independent.
Furthermore $\tilde w_0 \in \range(P_0^{2,s})$ due to \eqref{al2} and $\tilde w_0
\in \range(P_0^{3,u})$ due to \eqref{al3}. Thus $\tilde w_0 \in
\range(P_0^{3,u}) \cap \range(P_0^{2,s}) = W_0^2$.

Note that \eqref{al1} has the form

\[B \vect{\tilde w} = \vect r, \quad \text{with } B \in M_{2dN,d(2N+1)}(\R), \ \vect r \in \R^{2dN},\]
where
\[
B=
\begin{pmatrix}
-{e^{-\Lambda^\text{left}}} A_{-N} & I \\
& \ddots & \ddots \\
&& -{e^{-\Lambda^\text{left}}}A_{N} & I\\
\end{pmatrix},\ \vect{\tilde w} = \begin{pmatrix} \tilde w_{-N} \\ \vdots \\ \tilde w_{N-1}\end{pmatrix},
\]
and the $n$th entry of $\vect r$ is the vector $\delta_{n,-1}r^i \in \R^d$ for $i=1,2$.

The least squares solution can be obtained, using the
Moore-Penrose inverse:
\[
\vect{\tilde w} = B^+ \vect r,\quad \text{where } B^+ = B^T(BB^T)^{-1},
\]
cf.\ \cite{sst72}. Numerically, we find $\vect{\tilde w}$ by solving
the linear system $BB^T y = \vect r$; then $\vect{\tilde w} = B^T y$.

\begin{figure}[!h]
  \begin{lstlisting}[label=lst:ed1, caption = Sample MATLAB code for Algorithm \ref{alg:dichproj1}]
% step 1
B = zeros(2*N*dim,2*(N+1)*dim);
for i = 1:2*N
  B(dim*(i-1)+1:dim*i,dim*(i-1)+1:dim*i) 
   = -exp(-Lambda_left)*A(:,dim*(i-1)+1:dim*i);
  B(dim*(i-1)+1:dim*i,dim*(i)+1:dim*(i+1)) 
   = eye(dim);
end
R = zeros(2*dim*N,2);
R(dim*(N-1)+1:dim*N,:) = rand(dim,2);
y = (B*B')\R;
u = B'*y;
p1 = A(:,dim*(N-1)+1:dim*N)*u(dim*(N-1)+1:dim*N,1);
p2 = A(:,dim*(N-1)+1:dim*N)*u(dim*(N-1)+1:dim*N,2);
p1 = v1/norm(v1); v2 = v2/norm(v2);
% step 2
B = zeros(dim*(N+1),dim*(N+1)+1);
for i = 0:N-1
  B(dim*i+1:dim*(i+1),dim*i+1:dim*(i+1)) 
   = -exp(-Lambda_right)*A(:,dim*(i+N)+1:dim*(i+N+1));
  B(dim*i+1:dim*(i+1),dim*(i+1)+1:dim*(i+2)) 
   = eye(dim);
end
B(dim*N+1:dim*(N+1),1:dim) = eye(dim);
B(dim*N+1:dim*(N+1),dim*(N+1)+1) = p1;
R = zeros(dim*(N+1),1);
R(dim*N+1:dim*(N+1),1) = -p2;
y = (B*B')\R;
u = B'*y;
w2 = u(dim*N+1:dim*(N+1))/norm(u(dim*N+1:dim*(N+1)));
  \end{lstlisting}
\end{figure}

Note that in the unlikely case where $p^1 \in W^2_0$, Algorithm
\ref{alg:dichproj1} fails.
An alternative approach for computing vectors in $W_0^2$ that avoids this
problem is
introduced in Algorithm \ref{A2}. The main idea of this algorithm
is to take a random vector
$r$, project it first to $\range(P_0^{3,u})$ and then eliminate
components in the wrong subspaces, by projecting with $P_0^{2,s}$.

\begin{algorithm}[An alternate Dichotomy Projector approach\label{A2}]
\quad
\begin{enumerate}
\item Again, suppose $N\in \N$ and consider $n\in [-N,N]\cap\Z$ and let $A_n = A(T^nx)$ as above. Solve the least squares problem
\begin{align}\label{step1}
\tilde w_{n+1} &= {e^{-\Lambda^\text{left}}}A_n \tilde w_n +\delta_{n,-1} r,\quad n =
-N,\dots,N-1,\\
&\text{such that} \ \|(\tilde w_{-N},\ldots,\tilde w_N)\|_2 \ \text{is minimised},\notag
\end{align}
where $r$ is chosen at random, and define
$r' = A_{-1} \tilde w_{-1}$.

\item Solve the least squares problem
\begin{align}\label{step2}
  \tilde w'_{n+1} &= {e^{-\Lambda^\text{right}}} A_n \tilde w'_n + \delta_{n,-1} r',\quad n =
-N,\dots,N-1,\\
&\text{such that} \ \|(\tilde w'_{-N},\ldots,\tilde w'_N)\|_2 \ \text{is minimised}.\notag
\end{align}
Then $\tilde w'_0$ is our approximation of $w_2(x) \in W_2(x)$.
\end{enumerate}
\end{algorithm}

The solution $w_0$ on $\Z$ of these two steps satisfies
$
w_0 = P_0^{2,s} P_0^{3,u} r \in \range(P_0^{2,s}) \cap
\range(P_0^{3,u}) = W^2_0$.

\begin{figure}[hbt]
  \begin{lstlisting}[label=lst:ed2, caption = Sample MATLAB code for the second step of Algorithm \ref{A2}]
B = zeros(2*N*dim,2*(N+1)*dim);
for i = 1:2*N
  B(dim*(i-1)+1:dim*i,dim*(i-1)+1:dim*i) 
   = -exp(-Lambda_right)*A(:,dim*(i-1)+1:dim*i);
  B(dim*(i-1)+1:dim*i,dim*i+1:dim*(i+1)) = eye(dim);
end
R = zeros(2*dim*N,1);
R(dim*(N-1)+1:dim*N,:) = p1;
y = (B*B')\R;
u = B'*y;
w2 = u(dim*N+1:dim*(N+1))/norm(u(dim*N+1:dim*(N+1)));
  \end{lstlisting}
\end{figure}

\subsection{Error estimate}\label{S3.1}

We give an error estimate for the solution of Algorithm \ref{A2}
for a finite choice of $N$.
Details on deriving this
estimate are postponed to a forthcoming publication.

For ${\Lambda^\text{left}}$ and ${\Lambda^\text{right}}$ close to the boundary of the second
Sacker-Sell spectral interval, we denote the dichotomy rates of
\begin{equation}
w_{n+1} = {e^{-\Lambda^\text{left}}} A_n w_n,\quad
w_{n+1} = {e^{-\Lambda^\text{right}}} A_n w_n,\quad n \in \Z
\label{eq:difference_equations}
\end{equation}
by $(\alpha^{\ell,s}, \alpha^{\ell,u})$
and $(\alpha^{r,s}, \alpha^{r,u})$,
respectively. Let $w_0$ be the solution of Algorithm
\ref{A2} on $\Z$ and let $\tilde w_0$ be its approximation for a
finite choice of $N$.
Careful estimates show that the approximation error
in the ``wrong subspace" $\mathcal{R}(Q)$,
with $Q :=I-P_0^{2,s} P_0^{3,u}$ is given as
\begin{equation}\label{eq:error_estimate}
\|Q(w_0 - \tilde w_0)\| \le C N (e^{-\alpha^{\ell,s} N} + e^{-\alpha^{r,u}N}),
\end{equation}
where the constant $C>0$ does not depend on $N$.

\comments{Note that the dichotomy rates depend on the choice of $\Lambda^\text{left}$
and $\Lambda^\text{right}$. More precisely, $\alpha^{\ell,s}$ is zero if
$\Lambda^\text{left}$ lies at the left boundary of the resolvent interval
$R_3$ and $\Lambda^\text{left}$ is maximal at the right boundary of $R_3$.
On the other hand, $\alpha^{r,u}$ is zero at the right boundary of
$R_2$ and maximal at its left boundary.}

The exponential dichotomy rates $\alpha^{\ell,s}$ and $\alpha^{r,u}$ of the difference equations \eqref{eq:difference_equations} depend on the choice of $\Lambda^\text{left}$ and $\Lambda^\text{right}$ in the following way:  for $\Lambda^\text{left}$ in the resolvent set $R_3=[\lambda_3^+,\lambda_2^-]$ the difference equation
\[ w_{n+1} = e^{-\Lambda^\text{left}} A_n w_n, \quad n\in \Z\]
has an exponential dichotomy with stable dichotomy rate $\alpha^{\ell,s}$ for all $\alpha^{\ell,s}$ with
\[0 < \alpha^{\ell,s} < \Lambda^\text{left} - \lambda^+_3. \]
Similarly, for $\Lambda^\text{right}$ in the resolvent set $R_2 = [\lambda_2^+,\lambda_1^-]$ the difference equation
\[ w_{n+1} = e^{-\Lambda^\text{right}}A_n w_n, \quad n \in \Z\]
has an exponential dichotomy with unstable dichotomy rate $\alpha^{r,u}$ for all $\alpha^{r,u}$ with
\[ 0 < \alpha^{r,u} < \lambda_1^- - \Lambda^\text{right}. \]
Note that both of the above inequalities are strict.

Inspecting equation \eqref{eq:error_estimate}, we get the best (smallest) maximal error if we choose $\Lambda^\text{left} \in R_3$ and $\Lambda^\text{right}\in R^2$ so as to maximise $\alpha^{\ell,s}$ and $\alpha^{r,u}$. Consequently, we get the best numerical approximations, if
${\Lambda^\text{left}}$ and ${\Lambda^\text{right}}$ are chosen close to, but not equal to, the boundary of the common spectral interval ${[\lambda_2^-,\lambda_2^+]}$.
\section{The Ginelli and Wolfe schemes\label{sec:gin_wolfe}}
\subsection{The Ginelli scheme}
\label{sec:ginelli}
The Ginelli Scheme was first presented by Ginelli \emph{et al.\ } in \cite{Ginelli2007} as a method for accurately computing the covariant Lyapunov vectors of an orbit of an invertible differentiable dynamical system where the $A(x)=DT(x)$ are the Jacobian matrices of the flow or map.

Estimates of the $W_j(x)$ are found by constructing equivariant subspaces $S_j(x)=W_1(x)\oplus\cdots\oplus W_j(x)$ and filtering the invariant directions contained therein using a power method on the inverse system restricted to the subspaces $S_j(x)$.

To construct the subspaces $S_j(x)$ we utilise the notion of the stationary Lyapunov basis \cite{Ershov1998}.  Choose $j$ orthonormal vectors $s_1(T^{-n}x), s_2(T^{-n}x),\ldots, s_j(T^{-n}x)$, $n\geq 1$, such that $s_i(T^{-n}x) \notin V_{j+1}(T^{-n}x)$ for $1\le i\le j$ and construct
\begin{align}
\tilde s^{(n)}_i(x) = A(T^{-n}x,n)s_i(T^{-n}x), \quad i=1,\ldots,j\notag.
\end{align}
Using the Gram-Schmidt procedure, construct the orthonormal basis $\{s_1^{(n)}(x),\ldots, s_j^{(n)}(x)\}$ from $\{\tilde s_1^{(n)}(x), \ldots, \tilde s_j^{(n)}(x)\}$, that is,
\begin{align}
  s_1^{(n)}(x) & = \frac{1}{\norm{\tilde s_1^{(n)}(x)}}\tilde s_1^{(n)}(x), \notag\\
  s_2^{(n)}(x) & = \frac{1}{\norm{\left(\tilde s_2^{(n)}(x) - \left(\tilde s_2^{(n)}(x)\cdot s_1^{(n)}(x)\right)s_1^{(n)}(x)\right)}}\left(\tilde s_2^{(n)}(x) - \left(\tilde s_2^{(n)}(x)\cdot s_1^{(n)}(x)\right)s_1^{(n)}(x) \right),  \label{eq:QR_def} \\
& \vdots \notag
\end{align}
Then as $n\to \infty$ the basis $\{s_1^{(n)}(x) , \ldots, s_j^{(n)}(x)\}$ converges to a set of orthonormal vectors $\{s_1^{(\infty)}(x),\ldots,s_j^{(\infty)}(x)\}$ which span the $j$ fastest expanding directions of the cocycle $A$ \cite{Ershov1998}, that is, if the multiplicities $m_1 = \cdots = m_j = 1$
\begin{align}
  S_j(x):=\spn\left\{s^{(\infty)}_1(x),\ldots,s^{(\infty)}_j(x)\right\} & = W_1(x)\oplus \cdots \oplus W_j(x) \notag \\
& = V_1(x) \backslash V_{j+1}(x). \label{eq:subspaces}
\end{align}
If the Oseledets subspaces are not all one-dimensional, that is the Lyapunov spectrum is degenerate, then we choose $S_j(x)$ only for those $j$ which are the sum of the first $k$ multiplicities, i.e., $j= m_1 + \cdots +m_k$.  Then
\begin{align}
  S_j(x) & = W_1(x) \oplus \cdots \oplus W_k(x) \notag \\
& = V_1(x)\backslash V_{k+1}(x). \notag
\end{align}
In the interest of readability we assume the Oseledets subspaces are one-dimensional but note that the approach may be extended to the multi-dimensional case.

  Note that the $S_j(x)$ are equivariant by construction:
\begin{align}
  A(x,n) S_j(x) = S_j(T^nx) \notag
\end{align}
provided $j \leq m_1 + \cdots + m_{\ell-1}$ if $\lambda_\ell = -\infty$.

We describe the Ginelli approach to finding $W_2(x)$.
Suppose $\dim W_1(x) = \dim W_2(x) = 1$ and $\lambda_1 > \lambda_2 > -\infty$ and that the basis $\{s_1^{(\infty)}(x),s_2^{(\infty)}(x)\}$ is known at $x\in X$.  Note first that $\spn\left\{s_1^{(\infty)}(x)\right\} = W_1(x)$.  Let $c(x) \in \mathbb R^2$ denote the coefficients of $w_2(x)\in W_2(x)$ in the basis $\{s_1^{(\infty)}(x), s_2^{(\infty)}(x)\}$ (recall that the orthogonal projection of $w_2(x)$ onto $s_i^{(\infty)}(x)$ is zero for $i=3,4,\ldots,d$)

 then
\begin{align*}
  w_2(x) = c_1(x)s_1^{(\infty)}(x) + c_2(x) s_2^{(\infty)}(x). \notag
\end{align*}


\begin{lemma}\label{lem:QR_SLB}
Let $Q(x)$ denote the $d\times 2$ matrix whose $i$th column is $s_i^{(\infty)}(x)$.
  Then for each $n \geq 0$ there exists an upper triangular, $2 \times 2$ matrix $R(x,n)$ satisfying
  \begin{equation}
    A(x,n)Q(x) = Q(T^nx)R(x,n). \label{eq:QR}
    \end{equation}
\end{lemma}
\begin{proof}
Note that
\begin{align}
  A(x,n)Q(x) & = A(x,n) \begin{pmatrix} | & | \\ s_1^{(\infty)}(x) & s_2^{(\infty)}(x) \\ | & | \end{pmatrix} \notag \\
& = \begin{pmatrix} | & | \\ A(x,n)s_1^{(\infty)}(x) & A(x,n)s_2^{(\infty)}(x) \\ | & | \end{pmatrix} \notag \\
& = Q(T^nx) R(x,n) \notag
\end{align}
where
\begin{align}
  Q(T^nx) & =
  \begin{pmatrix}
    | & | \\
    s_1^{(\infty)}(T^nx) & s_2^{(\infty)}(T^nx) \\
    | & |
  \end{pmatrix} \notag
\end{align}
and
\begin{align}
  R(x,n) & =
  \begin{pmatrix}
    \norm{A(x,n)s_1^{(\infty)}(x)} & \ip{s_1^{(\infty)}(T^nx)}{A(x,n)s_2^{(\infty)}(x)} \\
    0 & \norm{A(x,n)s_2^{(\infty)}(x)}
  \end{pmatrix}, \label{eq:R(x,n)}
\end{align}
using the equivariance of $S_1(x) = \spn\{s_1^{(\infty)}(x)\}$ and $S_2(x) = \spn\{s_1^{(\infty)}(x),s_2^{(\infty)}(x)\}$.
\end{proof}
Thus, the QR-decomposition of Lemma \ref{lem:QR_SLB} is equivalent to the Gram-Schmidt orthonormalisation that defines the stationary Lyapunov bases.  The columns of $Q(T^nx)$ form the stationary Lyapunov basis at $T^nx$.

%

We have chosen the above notation $R(x,n)$ specifically since, defined in this way, $R$ forms a cocycle which is the restriction of $A$ to the invariant subspaces $S_j$.

\begin{lemma}
  The matrix $R(x,n)$ defined above forms a cocycle over $T$.
\end{lemma}
\begin{proof}
 Let $n,m \ge 0$ then
 \begin{align}
   A(x,n+m)Q(x) & = Q(T^{n+m}x)R(x,n+m) \label{eq:lem_R_1}
 \end{align}
by Lemma \ref{lem:QR_SLB}.  Since $A(x,n+m) = A(T^nx,m)A(x,n)$,
\begin{align}
  A(x,n+m)Q(x) & = A(T^nx,m)A(x,n)Q(x) \notag \\
  & = A(T^nx,m)Q(T^nx)R(x,n) \notag \\
  & = Q(T^{n+m}x)R(T^nx,m)R(x,n). \label{eq:lem_R_2}
\end{align}
Equating \eqref{eq:lem_R_1} and \eqref{eq:lem_R_2} gives
\begin{align}
  R(T^nx,m)R(x,n) = R(x,n+m), \notag
\end{align}
as $Q(T^{n+m}x)$ is left-invertible.
\end{proof}

Since $c(x)$ is the vector of coefficients of the second Oseledets vector of the cocycle $A$, it is the second Oseledets vector of the cocycle $R$.  To see this, recall $w_2(x) = Q(x)c(x) \in W_2(x)$ so that
\begin{align}
  \lambda_2 = \lim_{n\to \infty}\frac{1}{n}\log\norm{A(x,n)Q(x)c(x)} \notag
\end{align}
which, due to \eqref{eq:QR}, becomes
\begin{align}
  \lambda_2 &= \lim_{n\to\infty}\frac{1}{n}\log\norm{Q(T^nx)R(x,n)c(x)} \notag \\
& = \lim_{n\to\infty} \frac{1}{n}\log\norm{R(x,n)c(x)} \notag
\end{align}
since the columns of $Q(T^nx)$ are orthonormal.

We may approximate $c(x)$ numerically using a simple power method on the inverse cocycle $R^{-1}$ (which exists since $\lambda_1 > \lambda_2 > -\infty$).

The Ginelli method can be summarised by the following steps:
\begin{algorithm}[Ginelli method of approximating $w_2(x)\in W_2(x)$\label{alg:ginelli}]\quad
\begin{enumerate}
\item Choose $x\in X$ and $M>0$ and form $\{s_1^{(M)}(x),s_2^{(M)}(x)\}$ by first randomly selecting two orthonormal vectors $\{s_1(T^{-M}x),s_2(T^{-M}x)\}$ then performing the push-forward/Gram-Schmidt procedure given by \eqref{eq:QR_def}.  That is, define
\[ \tilde s_i^{(M)}(x) = A(T^{-M}x,M)s_i(T^{-M}x), \quad i = 1,2,\]
followed by setting
\begin{align*}
s_1^{(M)}(x) &= \mathcal N\left(\tilde s_1^{(M)}(x)\right), \\
s_2^{(M)}(x) & = \mathcal N \left(\tilde s_2^{(M)}(x) - \left(\tilde s_2^{(M)}(x)\cdot s_1^{(M)}(x)\right)s_1^{(M)}(x)\right),
\end{align*}
where $\mathcal N:v\mapsto v/\norm{v}$. The vectors $\{s_1^{(M)}(x),s_2^{(M)}(x)\}$ form an approximation to the stationary Lyapunov basis $\{s_1^{(\infty)}(x),s_2^{(\infty)}(x)\}$.

\item \label{item:ginelli_step_2}Choose $N>0$ and using the approximate basis $\{s_1^{(M)}(x),s_2^{(M)}(x)\}$ in \eqref{eq:R(x,n)}, form an approximation to $R(x,N)$, denoted by $R^{(M)}(x,N)$.
\item Choose $c'\in \mathbb R^2$ either at random, or by some guess at the second Oseledets vector of $R$ at $T^N(x)\in X$, in this review we found $c'=(0,1)$ to work well. Use the inverse iteration method to approximate $c(x)$, that is, define our approximation to $c(x)$ as
  \begin{align}
    c^{(M,N)}(x)&  = R^{(M)}(x,N)^{-1}c' \notag\\
    & = R^{(M)}(T^Nx,-N)c' \notag
  \end{align}
\item Then
  \begin{align*}
    w_2^{(M,N)}(x)=
    \begin{pmatrix}
      | & |\\
      s_1^{(M)}(x) & s_2^{(M)}(x)\\
      | & |
    \end{pmatrix}c^{(M,N)}(x)
  \end{align*}
is our approximation to $w_2(x)\in W_2(x)$.
\end{enumerate}
\end{algorithm}

As before, there is some freedom of choice of both $M$ and $N$ as well as of the initial orthonormal basis $\{s_1(T^{-M}x),s_2(T^{-M}x)\}$, used to approximate $S_2(x)$, and of the 2-tuple $c'$.  The larger $M$ and $N$ are chosen, the more accurate $w^{(M,N)}_2(x)$ will be, provided $s_2(T^{-M}x) \notin W_1(T^{-M}x)\cup V_3(T^{-M}x)$ and $c' \notin E_1(T^Mx)$ where $E_1$ is the Oseledets subspace of $R$ with Lyapunov exponent $\lambda_1$.

\begin{figure}[hbt]
  \begin{lstlisting}[label=lst:ginelli, caption = Sample MATLAB code for Algorithm \ref{alg:ginelli} ]
[Q,~] = qr(rand(dim,j),0);c = [zeros(1,j-1) 1];
for i = 1:N,
  QNew = A(:,(i-1)*dim+1:i*dim)*Q;
  [Q,~] = qr(QNew,0);
end
Q0 = Q;
for i = N+1:2*N+1,
  QNew = A(:,(i-1)*dim +1: i*dim)*Q;
  [Q,R] = qr(QNew,0);
  AllR = horzcat(R,AllR);
end
numOfR = size(AllR,2)/j;
for i = 1:numOfR
  R = AllR(:,(i-1)*j+1:i*j);
  cNew = R\c;
  c = cNew/norm(cNew);
end
w = Q0*c;
  \end{lstlisting}
\end{figure}

Listing \ref{lst:ginelli} shows an example implementation of Algorithm \ref{alg:ginelli} in MATLAB which approximates $w_j(x) \in W_j(x)$ using $M = N$ and $s_1(T^{-M}x),\ldots,s_j(T^{-M}x)$ are chosen at random and $c'=(\underbrace{0,\ldots,0}_{j-1 \text{ entries}},1)$.  Lines 2 through 6 construct the approximation of the stationary Lyapunov basis, $\left\{s_1^{(M)}(x),\ldots,s_j^{(M)}(x)\right\}$ which is stored as columns of the matrix $Q(x)$ represented as the variable \vbl{Q0}, while lines 7 through 11 construct the cocycle $R$ stored in \vbl{AllR} as $[ R(x) \ | \ R(Tx) \ | \cdots | \ R(T^Nx)]$.  Lines 12 through 17 perform a simple power method on $R(x,N)^{-1}$ to find the coefficient vector $c$, which represents the approximation of $w_j(x)$ in the basis $\left\{s_1^{(M)}(x),\ldots,s_j^{(M)}(x)\right\}$.  Thus, the approximation is given by $Q(x)c$.  Although Algorithm \ref{alg:ginelli} is specific to the case where $j=2$, Listing \ref{lst:ginelli} is applicable to any $j$ for which $R(x,N)^{-1}$ exists.

It can be shown that, in this case where the top Lyapunov exponent has multiplicity 1, $E_1 = \spn\{(1,0)^T\}$.

\begin{lemma}
If the first Lyapunov exponent has multiplicity 1, the dominant Oseledets subspace of the cocycle $R$ is $E_1 = \spn\{(1,0)^T\}$.
\end{lemma}
\begin{proof}
  Recall that $s_1^{(\infty)}(x) \in W_1(x)$ since $\spn\left\{s_1^{(\infty)}(x)\right\} = S_1(x) = V_1(x) \backslash V_2(x)$ (from \eqref{eq:subspaces}) and for all $s \in V_1(x)\backslash V_2(x)$
  \begin{align}
    \lambda_1 & = \lim_{n\to \infty} \frac{1}{n} \log\norm{A(x,n)s}. \notag
  \end{align}
We may write $s = Q(x)(a,0)^T$ for some $a\in \mathbb R$ then
\begin{align}
  \lambda_1 & = \lim_{n\to \infty}\frac{1}{n} \log \norm{A(x,n)Q(x)(a,0)^T} \notag \\
  & = \lim_{n\to\infty}\frac{1}{n}\log\norm{Q(T^nx)R(x,n)(a,0)^T} \notag
  \end{align}
and since $\norm{Q(T^nx)s} = \norm{s}$ (the columns of $Q$ are orthonormal)
\begin{align}
  \lambda_1 & = \lim_{n\to\infty}\frac{1}{n}\log\norm{R(x,n)(a,0)^T}. \notag
\end{align}
Since $S_1$ is $A$-invariant, $\spn\{(1,0)^T\}$ is $R$-invariant and the proof is complete.
\end{proof}

\subsubsection{Limited Data Scenario\label{sec:limited_ginelli}}
In the case where convergence isn't satisfactory because the amount of cocycle data available is too small (for any $M$ and $N$ to be small), the approximations {from Algorithm \ref{alg:ginelli}} can be improved by using better guesses at $s_1(T^{-M}x),s_2(T^{-M}x)$ and $c'$.

Note that $s_1^{(\infty)}(x)$ and $s_2^{(\infty)}(x)$ are two orthonormal vectors optimised for maximal growth over the time interval $[-\infty,0]\cap\Z$. In the situation where the values of $M$ and $N$ are limited, one can choose those two vectors {that} are optimised for growth over the shorter time interval $[-M,0]\cap\Z$. In \cite{Wolfe2007} this is achieved by computing the left singular vectors of $A(T^{-M}x,M)$. This approach works well for small $M$ but can become inaccurate for very large $M$ for the reasons in Remark \ref{svdbad}.

In practice, we have observed that a combination of Step 1 in Algorithm \ref{alg:ginelli} and the above provides the most robust method of accurately approximating $s_1^{(\infty)}(x)$ and $s_2^{(\infty)}(x)$.  As an alternative to Algorithm \ref{alg:ginelli} the following may be used: { For $M\geq M'$, compute vectors optimised for growth from $-M$ to $-M + M'$, then push-forward these vectors from $-M+M'$ to $0$.}

\begin{algorithm}[{ Improved }Algorithm \ref{alg:ginelli}]\quad
\label{alg:ginalt}
\begin{enumerate}
\item \sloppy Choose $x\in X$ and $M \geq M' > 0$.  Compute the two left singular vectors of $A(T^{-M}x,M')$ corresponding to the two largest singular values and call them $\tilde s_1(T^{-M+M'}x)$ and $\tilde s_2(T^{-M+M'}x)$.  Now define $\left\{s_1^{(M,M')}(x),s_2^{(M,M')}(x)\right\}$ as an approximation to $\left\{s_1^{(\infty)}(x),s_2^{(\infty)}(x)\right\}$ by the Gram-Schmidt orthonormalisation of $A(T^{-M+M'}x,M-M')s_1(T^{-M+M'}x)$ and $A(T^{-M+M'}x,M-M')s_2(T^{-M+M'}x)$ as in \eqref{eq:QR_def}.
\end{enumerate}
Steps 2--4 as in Algorithm \ref{alg:ginelli}.
\end{algorithm}

In {practice, one} should choose $M'$ large enough so that enough data is sampled, but not so large that $A(T^{-M}x,M')$ is too singular.

\subsection{The Wolfe scheme}
\label{sec:wolfe}
The approach followed by Wolfe \emph{et al.\ } \cite{Wolfe2007} directly computes the subspace splitting as the intersection of two sets of invariant subspaces.  The description of the {numerical construction of the subspaces $S_j(x)$} featured below differs slightly from \cite{Wolfe2007}, however, the essential features of the approach are retained.  In fact, the constructions featured here improve upon those in \cite{Wolfe2007} in terms of accuracy versus amount of cocycle data used -- in the notation of Algorithm \ref{alg:wolfe} below, if $M_1$ is made larger, $w_2^{(M_1,M_1',M_2)}(x)$ is more accurate, which is not the case in \cite{Wolfe2007} for the same reasons discussed in Remark \ref{svdbad}.

Recall the eigenspace decomposition $U_j(x)$ of the limiting matrix $\Psi(x)$ presented in Section \ref{sec:SVD} and define $V_j(x) = U_j(x)\oplus \cdots \oplus U_\ell(x)$.  Recall that $V_j(x) \supset W_j(x),\allowbreak W_{j+1}(x),\allowbreak\ldots,\allowbreak W_\ell(x)$.  Also recall from the previous section that $S_j(x) \supset W_1(x),W_2(x),\ldots, W_j(x)$.  Thus
\begin{align}
  W_j(x) = V_j(x) \cap S_j(x).\notag
\end{align}

Again, in the interest of readability we assume the Oseledets subspaces $W_j(x)$, and the eigenspaces $U_j(x)$, are one-dimensional.  As in the case of the previous section, the ideas here may be extended to the case in which the Oseledets subspaces are not one-dimensional.  Let $u_j(x)$ be the singular vector spanning $U_j(x)$ and let $s_j(x)=s^{(\infty)}_j(x)$ be the $j$th element of the stationary Lyapunov basis as in the previous section.  Then note
\begin{align}
  w_j(x) & = \sum_{i=1}^j \ip{w_j(x)}{s_i(x)} s_i(x), \notag\\
  w_j(x) & = \sum_{i=j}^d \ip{w_j(x)}{u_i(x)} u_i(x). \notag
\end{align}
Taking inner products with $u_k(x)$ and $s_k(x)$ respectively gives
\begin{align}
  \ip{w_j(x)}{u_k(x)} & = \sum_{i=1}^j\ip{w_j(x)}{s_i(x)}\ip{s_i(x)}{u_k(x)} \text{for $k\geq j$}, \label{eq:wolfe_4} \\
  \ip{w_j(x)}{s_k(x)} & = \sum_{i=j}^d\ip{w_j(x)}{u_i(x)}\ip{u_i(x)}{s_k(x)} \text{for $k\leq j$}. \label{eq:wolfe_3}
\end{align}
Substituting \eqref{eq:wolfe_4} into \eqref{eq:wolfe_3} and rearranging gives
\begin{align}
  \ip{w_j(x)}{s_k(x)} & = \sum_{i=1}^j \left(\sum_{h=j}^d \ip{s_k(x)}{u_h(x)}\ip{u_h(x)}{s_i(x)}\right)\ip{s_i(x)}{w_j(x)}. \label{eq:wolfe_6}
\end{align}
Note that $\sum_{h=1}^d\ip{s_k(x)}{u_h(x)}\ip{u_h(x)}{s_i(x)} = \delta_{ki}$ so
\begin{align}
  \sum_{h=j}^d \ip{s_k(x)}{u_h(x)}\ip{u_h(x)}{s_i(x)} = \delta_{ki} - \sum_{h=1}^{j-1}\ip{s_k(x)}{u_h(x)}\ip{u_h(x)}{s_i(x)}. \notag
\end{align}
Then \eqref{eq:wolfe_6} becomes
\begin{align}
  \ip{w_j(x)}{s_k(x)} & = \sum_{i=1}^j \delta_{ki}\ip{s_i(x)}{w_j(x)} - \sum_{i=1}^j \sum_{h=1}^{j-1}\ip{s_k(x)}{u_h(x)}\ip{u_h(x)}{s_i(x)}\ip{s_i(x)}{w_j(x)} \notag \\
  & = \ip{s_k(x)}{w_j(x)} - \sum_{i=1}^j \sum_{h=1}^{j-1}\ip{s_k(x)}{u_h(x)}\ip{u_h(x)}{s_i(x)}\ip{s_i(x)}{w_j(x)} \notag \\
0 & = \sum_{i=1}^j \sum_{h=1}^{j-1}\ip{s_k(x)}{u_h(x)}\ip{u_h(x)}{s_i(x)}\ip{s_i(x)}{w_j(x)}. \label{eq:wolfe_7}
\end{align}
Equation \eqref{eq:wolfe_7} may be considered as a $j\times j$ homogeneous linear equation by defining a matrix entry-wise as
\begin{align}
  D_{ki} = \sum_{h=1}^{j-1}\ip{s_k(x)}{u_h(x)}\ip{u_h(x)}{s_i(x)} \notag
\end{align}
and solving
\begin{align}
  Dy = 0, \notag
\end{align}
where $y_i = \ip{s_i(x)}{w_j(x)}$.
The entries of $y$ are then the coefficients of $w_j(x)$ with respect to the basis $s_1(x)\ldots,s_j(x)$.

The Wolfe approach may be implemented as follows:
\begin{algorithm}[Improved Wolfe approach to approximating $w_j(x) \in W_j(x)$\label{alg:wolfe}]
\quad
\begin{enumerate}
\item Choose $x\in X$ and $M_1 \geq M_1' > 0$ and construct $\{s_1^{(M_1,M_1')}(x),\ldots,s_j^{(M_1,M_1')}(x)\}$ as an approximation of the stationary Lyapunov basis vectors $\{s_1^{(\infty)}(x),\ldots,s_j^{(\infty)}(x)\}$ using the methods outlined in Step 1 of Algorithm \ref{alg:ginalt}, that is, compute the left singular vectors of $A(T^{-M_1}x,M_1')$ corresponding to the $j-1$ largest singular values and call them $\tilde s_i(T^{-M_1+M_1'}x)$, $i=1,\ldots,j-1$.  Then form the $s_i^{(M_1,M_1')}(x)$ by the Gram-Schmidt orthornormalisation of $A(T^{-M_1+M_1'}x,M_1-M_1')\tilde s_i(T^{-M_1 + M_1'}x)$ for $i=1,\ldots,j-1$.

\item Choose $M_2>0$ and construct the one-dimensional eigenspaces $U_1^{(M_2)}(x),\ldots,\allowbreak U_{j-1}^{(M_2)}(x)$ as approximations to the eigenspaces $U_1(x),\ldots, U_{j-1}(x)$ as in Step 1 of Algorithm \ref{alg:svd}, that is, construct
\[\Psi^{(M_2)}(x) = \left(A(x,M_2)^* A(x,M_2)\right)^{1/2M_2},\]
and let $U^{(M_2)}_i(x)$ be the $i$th orthonormal eigenspace of $\Psi^{(M_2)}(x)$. Define $u_i^{(M_2)}(x) \in U_i^{(M_2)}(x)$, $i=1,\ldots, j-1$.

\item Form the matrix $D$ as above:
  \begin{align}
    D_{ki} = \sum_{h=1}^{j-1}\ip{s_k^{(M_1,M_1')}(x)}{u_h^{(M_2)}(x)}\ip{u_h^{(M_2)}(x)}{s_i^{(M_1,M_1')}(x)}. \notag
  \end{align}
\item Solve the homogeneous linear equation $Dy=0$. Then $w_j^{(M_1,M_1',M_2)} = \sum_{i=1}^j y_i s_i(x)$ forms our approximation of $w_j(x)\in W_j(x)$.
\end{enumerate}
\end{algorithm}

This approach suffers from the same numerical stability issue of Algorithm \ref{alg:svd2} of Section \ref{sec:svd2}.  Namely, the vector spaces $U_1^{(M)}(x),\ldots,U_{j-1}^{(M)}(x)$ may only poorly approximate $U_1(x),\ldots,U_{j-1}(x)$ for $M$ too large (see the final paragraph of Section \ref{sec:svd2}).

\sloppy A recent paper \cite{parlitz2011} provides alternative descriptions of both the Ginelli \emph{et al.} and Wolfe and Samelson methods, well-suited to those familiar with the QR-decomposition based numerical method for estimating Lyapunov exponents due to Benettin \emph{et al.} \cite{benettin1978,benettin1980} and Shimada and Nagashima \cite{shimada1979}.  The discussion in \cite{parlitz2011} is restricted to invertible cocycles generated by the Jacobian matrices of a dynamical system.  Although this assumption allows stable numerical methods to be constructed, i.e., better convergence obtained for larger data sets, it means some important examples in which the matrix cocycle is non-invertible are overlooked, for example the case study of Section \ref{sec:fluidflow}. While the memory footprint of the implementations discussed in \cite{parlitz2011} is estimated, there is no {discussion of convergence rates or accuracy} with respect to the amount of cocycle data available.  Finally, while the examples featured in \cite{parlitz2011} explain the methods presented in the context of differentiable dynamical systems the case studies of Section \ref{sec:numerical} in the present paper focus on comparative performance of the methods presented, via a broad range of possible applications.

\section{Numerical comparisons of the four approaches \label{sec:numerical}}
We present three detailed case studies, comparing the four approaches for calculating Oseledets subspaces.
The first case study is a nontrivial model for which we know  the Oseledets subspaces exactly and can therefore precisely measure the accuracy of the methods.
The second case study produces a relatively low-dimensional matrix cocycle, while the third case study generates a very high-dimensional matrix cocycle;  in these case studies we use two fundamental properties of Oseledets subspaces to assess the accuracy of the four approaches.

\subsection{Case Study 1: An exact model}
In general the Oseledets subspaces cannot be found analytically which makes the task of determining the efficacy of the above approaches difficult. However, the exact model described below allows us to compare the numerical approximations with the exact solution by building a cocycle in which the subspaces are known \emph{a priori}.

We generate a system with simple Lyapunov spectrum $\lambda_1 > \lambda_2 > \cdots > \lambda_d > -\infty$.  We form a diagonal matrix
\begin{align}
  R =
  {\begin{pmatrix}
    e^{\lambda_1} & & &0  \\
    & e^{\lambda_2} &  & \\
    &  & \ddots &  \\
    0 & & & e^{\lambda_d} \\
  \end{pmatrix}}
\notag
\end{align}
and generate the cocycle by the sequence of matrices $\{A_n\}$ where
\begin{align}
  A_n & = S_n R S_{n-1}^{-1}  \notag \\
  S_n & =
  \begin{cases}
    I + \epsilon Z, & \text{ for $n\neq -1, n\in [-N,N]\cap\Z$,}\\
    I +
    \begin{pmatrix}
      0 & & & & \\
      z_2 & 0 & & \\
      & \ddots & \ddots & \\
      & & z_d & 0
    \end{pmatrix},
    & \text{ for $n=-1$.}
  \end{cases} \label{eq:toy_model}
\end{align}
The entries of $Z$ and the numbers $z_2,\ldots, z_d$ are uniformly randomly generated from the interval $[0,1]$.  By construction, the columns of $S_{n-1}$ span the Oseledets subspaces at time $n\in [-N,N]\cap\Z$.

We compare the exact result at time $n=0$ with the approximations computed by the various algorithms for $d=8$, {$\{\lambda_1,\ldots,\lambda_8\} = \{\log 8, \log 7, \log 6,\ldots,\log 1\}$} and $\epsilon = 0.1$ for varying amounts of cocycle data $\left\{A(T^{-N}x),\ldots, A(T^Nx)\right\}$.
The exact model has a well separated spectrum, is generated using invertible matrices, and is of relatively low dimension.

For this model we use the following choice of parameters to execute the algorithms:
\begin{itemize}
\item \textbf{Algorithm \ref{alg:svd2}:} $M = N$ and $\{N_k\} = \{1,6,\ldots,5k-4,\ldots,5K-4,N\}$ where $5K-4 < N \leq 5K+1$.
\item \textbf{Algorithm \ref{alg:dichproj1}:} We estimate the three largest Lyapunov exponents $\lambda_1 > \lambda_2 >\lambda_3$ and set $\Lambda^\text{right} = {\lambda_2 + 0.1(\lambda_1 - \lambda_2)}$ and $\Lambda^\text{left} = {\lambda_2 - 0.1(\lambda_2 - \lambda_3})$.
\item \textbf{Algorithm \ref{A2}:} As for Algorithm \ref{alg:dichproj1}.
\item \textbf{Algorithm \ref{alg:ginalt}:}  $M = N$, $M' = 5$, and $c'= (0,1)$.
\item \textbf{Algorithm \ref{alg:wolfe}:} Let $M_1 = N$, $M_1' = 5$ and $M_2 = N$.

\end{itemize}

\begin{figure}[hbt]
  \centering
  \def\svgwidth{\textwidth}
  
\begingroup%
  \makeatletter%
  \providecommand\color[2][]{%
    \errmessage{(Inkscape) Color is used for the text in Inkscape, but the package 'color.sty' is not loaded}%
    \renewcommand\color[2][]{}%
  }%
  \providecommand\transparent[1]{%
    \errmessage{(Inkscape) Transparency is used (non-zero) for the text in Inkscape, but the package 'transparent.sty' is not loaded}%
    \renewcommand\transparent[1]{}%
  }%
  \providecommand\rotatebox[2]{#2}%
  \ifx\svgwidth\undefined%
    \setlength{\unitlength}{617bp}%
    \ifx\svgscale\undefined%
      \relax%
    \else%
      \setlength{\unitlength}{\unitlength * \real{\svgscale}}%
    \fi%
  \else%
    \setlength{\unitlength}{\svgwidth}%
  \fi%
  \global\let\svgwidth\undefined%
  \global\let\svgscale\undefined%
  \makeatother%
  \begin{picture}(1,0.66126418)%
    \put(0,0){\includegraphics[width=\unitlength]{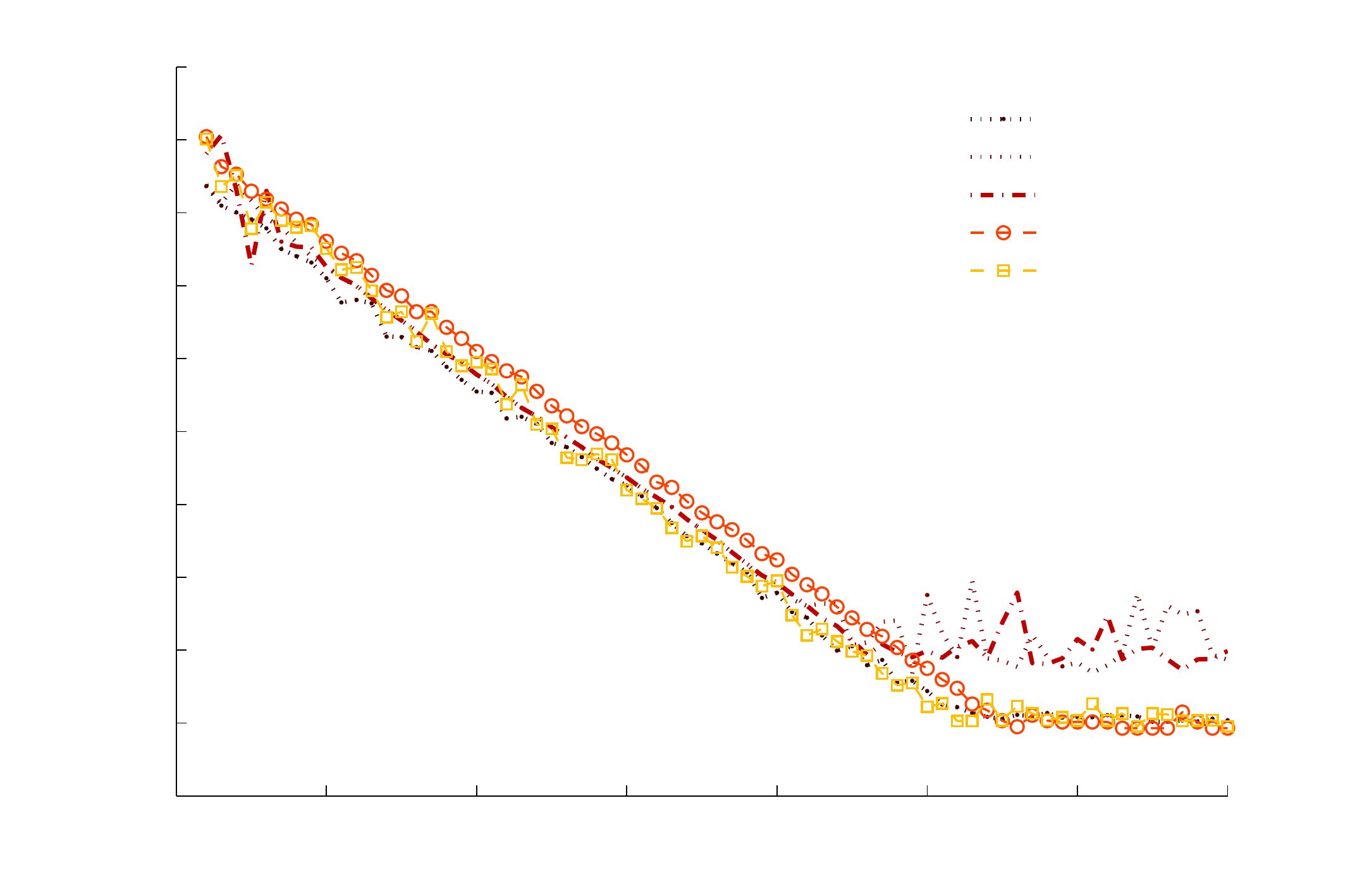}}%
    \put(0.12574263,0.05402464){\makebox(0,0)[lb]{\smash{\tiny 0}}}%
    \put(0.23203728,0.05402464){\makebox(0,0)[lb]{\smash{\tiny 50}}}%
    \put(0.33833063,0.05402464){\makebox(0,0)[lb]{\smash{\tiny 100}}}%
    \put(0.44921718,0.05402464){\makebox(0,0)[lb]{\smash{\tiny 150}}}%
    \put(0.56010211,0.05402464){\makebox(0,0)[lb]{\smash{\tiny 200}}}%
    \put(0.67098865,0.05402464){\makebox(0,0)[lb]{\smash{\tiny 250}}}%
    \put(0.7818752,0.05402464){\makebox(0,0)[lb]{\smash{\tiny 300}}}%
    \put(0.89276013,0.05402464){\makebox(0,0)[lb]{\smash{\tiny 350}}}%
    \put(0.08149627,0.06766613){\makebox(0,0)[lb]{\smash{\tiny $10^{-18}$}}}%
    \put(0.08149627,0.12155592){\makebox(0,0)[lb]{\smash{\tiny $10^{-16}$}}}%
    \put(0.08149627,0.17531118){\makebox(0,0)[lb]{\smash{\tiny $10^{-14}$}}}%
    \put(0.08149627,0.22920097){\makebox(0,0)[lb]{\smash{\tiny $10^{-12}$}}}%
    \put(0.08149627,0.28295462){\makebox(0,0)[lb]{\smash{\tiny $10^{-10}$}}}%
    \put(0.08149627,0.33670989){\makebox(0,0)[lb]{\smash{\tiny $10^{-8}$}}}%
    \put(0.08149627,0.39059968){\makebox(0,0)[lb]{\smash{\tiny $10^{-6}$}}}%
    \put(0.08149627,0.44435494){\makebox(0,0)[lb]{\smash{\tiny $10^{-4}$}}}%
    \put(0.08149627,0.49824473){\makebox(0,0)[lb]{\smash{\tiny $10^{-2}$}}}%
    \put(0.08149627,0.55199838){\makebox(0,0)[lb]{\smash{\tiny $10^0$}}}%
    \put(0.08149627,0.60575365){\makebox(0,0)[lb]{\smash{\tiny $10^2$}}}%
    \put(0.51404699,0.03471118){\makebox(0,0)[lb]{\smash{$N$}}}%
    \put(0.05226904,0.23733064){\rotatebox{90}{\makebox(0,0)[lb]{\smash{$\norm{w_2^{(N)}(x) - w_2(x)}$}}}}%
    \put(0.7701248,0.56766613){\makebox(0,0)[lb]{\smash{Algorithm \ref{alg:svd2}}}}%
    \put(0.7701248,0.53970827){\makebox(0,0)[lb]{\smash{Algorithm \ref{alg:dichproj1}}}}%
    \put(0.7701248,0.51175041){\makebox(0,0)[lb]{\smash{Algorithm \ref{A2}}}}%
    \put(0.7701248,0.48365802){\makebox(0,0)[lb]{\smash{Algorithm \ref{alg:ginalt}}}}%
    \put(0.7701248,0.45570016){\makebox(0,0)[lb]{\smash{Algorithm \ref{alg:wolfe}}}}%
  \end{picture}%
\endgroup%
  \caption{\label{fig:exact_overview}Comparing the approximations of the second Oseledets subspace $W^{(N)}_2(x)$ with the exact solution $W_2(x)$ which is known \emph{a priori}.  Each ``$N$-approximation'' is computed using cocycle data $\{A(T^{-N}x), A(T^{-N+1}x), \ldots, A(T^Nx)\}$. The comparison is simply the Euclidean norm of the separation of the two unit vectors $w_2^{(N)}(x)\in W_2^{(N)}(x)$ and $w_2(x) \in W_2(x)$.}
\end{figure}
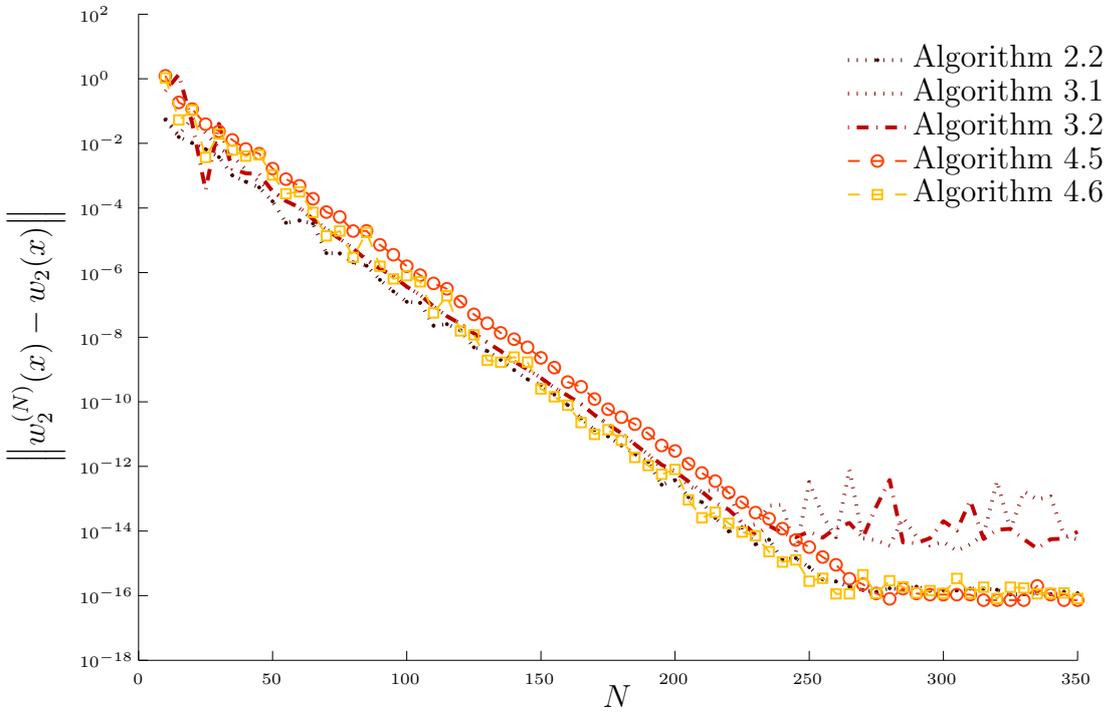

Figure \ref{fig:exact_overview} compares the approximations yielded from the various approaches outlined in Sections \ref{sec:svd}, \ref{sec:dich_proj} and \ref{sec:gin_wolfe} with the known solution of Equation \eqref{eq:toy_model}.  Each algorithm exhibits approximately exponential convergence with respect to the length of the sample cocycle up to (almost) machine accuracy of about $10^{-16}$.  Algorithm \ref{alg:ginelli} is notably erratic whereas the other algorithms converge smoothly, which suggests that in the limited data scenario (small $N$) it represents the less satisfactory choice.

Algorithm \ref{alg:svd2} is slightly more accurate than the other algorithms for large $N$, while there is a limit to the accuracy of Algorithms \ref{alg:dichproj1} and \ref{A2}.

It is worth noting that Algorithms \ref{alg:svd2} and \ref{alg:wolfe} do not perform as well when approximating Oseledets subspaces corresponding to Lyapunov exponents $\lambda_3, \ldots, \lambda_\ell$.  Whilst they reach machine accuracy with ease for $W_1$ and $W_2$, forming $A(x,n) = A(T^{n-1}x)\cdots A(x)$ via numerical matrix multiplication produces greater inaccuracies for subspaces $W_3$ through $W_\ell$ (see Remark \ref{svdbad}).  On the other hand, Algorithms \ref{alg:dichproj1}, \ref{A2} and \ref{alg:ginalt} do not suffer from the same issue because they do not need to form $A(x,n)$ but use only the generating matrices $A(x)$.  As such, they still reach machine accuracy, although a greater amount of data (larger $N$) is required.

\begin{figure}[hbt]
\def\svgwidth{\textwidth}
\begingroup%
  \makeatletter%
  \providecommand\color[2][]{%
    \errmessage{(Inkscape) Color is used for the text in Inkscape, but the package 'color.sty' is not loaded}%
    \renewcommand\color[2][]{}%
  }%
  \providecommand\transparent[1]{%
    \errmessage{(Inkscape) Transparency is used (non-zero) for the text in Inkscape, but the package 'transparent.sty' is not loaded}%
    \renewcommand\transparent[1]{}%
  }%
  \providecommand\rotatebox[2]{#2}%
  \ifx\svgwidth\undefined%
    \setlength{\unitlength}{617bp}%
    \ifx\svgscale\undefined%
      \relax%
    \else%
      \setlength{\unitlength}{\unitlength * \real{\svgscale}}%
    \fi%
  \else%
    \setlength{\unitlength}{\svgwidth}%
  \fi%
  \global\let\svgwidth\undefined%
  \global\let\svgscale\undefined%
  \makeatother%
  \begin{picture}(1,0.66126418)%
    \put(0,0){\includegraphics[width=\unitlength]{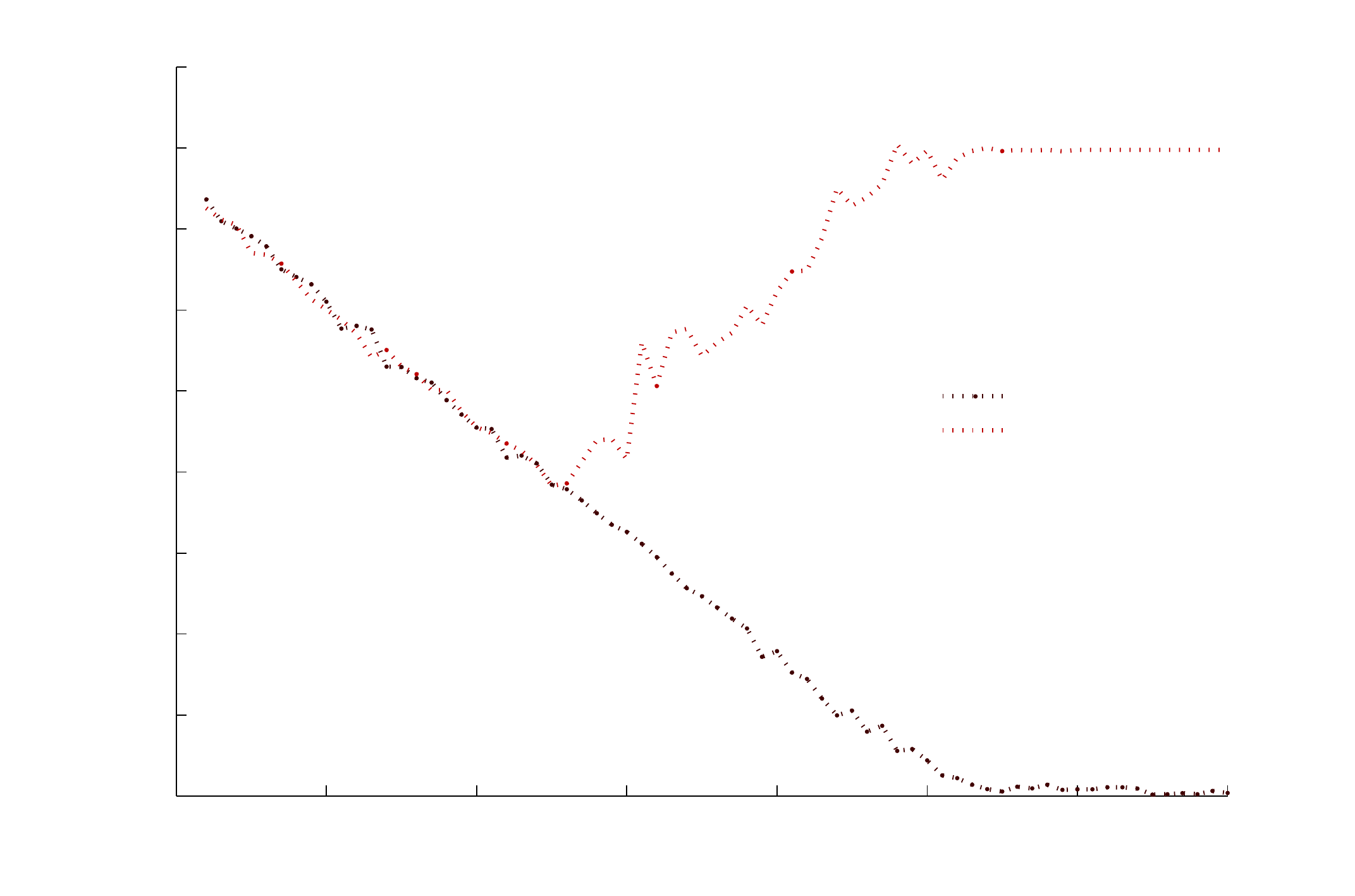}}%
    \put(0.12574263,0.05402464){\makebox(0,0)[lb]{\smash{\tiny 0}}}%
    \put(0.23203728,0.05402464){\makebox(0,0)[lb]{\smash{\tiny 50}}}%
    \put(0.33833063,0.05402464){\makebox(0,0)[lb]{\smash{\tiny 100}}}%
    \put(0.44921718,0.05402464){\makebox(0,0)[lb]{\smash{\tiny 150}}}%
    \put(0.56010211,0.05402464){\makebox(0,0)[lb]{\smash{\tiny 200}}}%
    \put(0.67098865,0.05402464){\makebox(0,0)[lb]{\smash{\tiny 250}}}%
    \put(0.7818752,0.05402464){\makebox(0,0)[lb]{\smash{\tiny 300}}}%
    \put(0.89276013,0.05402464){\makebox(0,0)[lb]{\smash{\tiny 350}}}%
    \put(0.06853031,0.06766613){\makebox(0,0)[lb]{\smash{\tiny $10^{-16}$}}}%
    \put(0.06853031,0.12749887){\makebox(0,0)[lb]{\smash{\tiny $10^{-14}$}}}%
    \put(0.06853031,0.18733063){\makebox(0,0)[lb]{\smash{\tiny $10^{-12}$}}}%
    \put(0.06853031,0.24702917){\makebox(0,0)[lb]{\smash{\tiny $10^{-10}$}}}%
    \put(0.06853031,0.30686062){\makebox(0,0)[lb]{\smash{\tiny $10^{-8}$}}}%
    \put(0.06853031,0.36669368){\makebox(0,0)[lb]{\smash{\tiny $10^{-6}$}}}%
    \put(0.06853031,0.4263906){\makebox(0,0)[lb]{\smash{\tiny $10^{-4}$}}}%
    \put(0.06853031,0.48622366){\makebox(0,0)[lb]{\smash{\tiny $10^{-2}$}}}%
    \put(0.06853031,0.54605673){\makebox(0,0)[lb]{\smash{\tiny $10^{0}$}}}%
    \put(0.06853031,0.60575365){\makebox(0,0)[lb]{\smash{\tiny $10^2$}}}%
    \put(0.51404699,0.03471118){\makebox(0,0)[lb]{\smash{$N$}}}%
    \put(0.02910546,0.22678367){\rotatebox{90}{\makebox(0,0)[lb]{\smash{$\norm{w^{(N)}_2(x) - w_2(x)}$}}}}%
    \put(0.75092681,0.3611035){\color[rgb]{0,0,0}\makebox(0,0)[lb]{\smash{Algorithm \ref{alg:svd2}}}}%
    \put(0.75113559,0.33449862){\color[rgb]{0,0,0}\makebox(0,0)[lb]{\smash{Algorithm \ref{alg:svd}}}}%
  \end{picture}%
\endgroup%
 \caption{\label{fig:exact_comp_SVD}Comparing the approximation of the second Oseledets vector with the exact solution for the two SVD based approaches, demonstrating the numerical instability which is overcome in the adapted SVD approach.}
\end{figure}
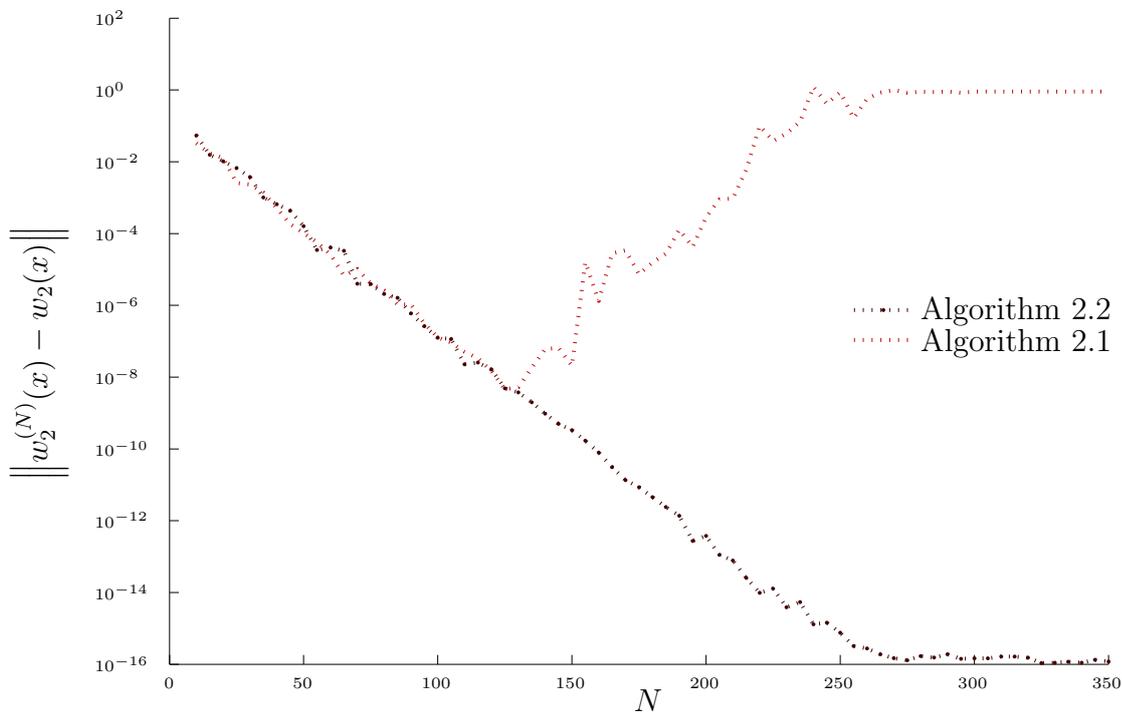

Figure \ref{fig:exact_comp_SVD} is similar to Figure \ref{fig:exact_overview} except that it compares only Algorithms \ref{alg:svd} and \ref{alg:svd2}.  In doing so, it highlights the result of one of the numerical instabilities of Algorithm \ref{alg:svd}, namely the pushing forward of $U_j^{(M)}(T^{-N}x)$ in Step \ref{item:SVD_step_3}.

\begin{figure}[hbt]
  \def\svgwidth{\textwidth}
\begingroup%
  \makeatletter%
  \providecommand\color[2][]{%
    \errmessage{(Inkscape) Color is used for the text in Inkscape, but the package 'color.sty' is not loaded}%
    \renewcommand\color[2][]{}%
  }%
  \providecommand\transparent[1]{%
    \errmessage{(Inkscape) Transparency is used (non-zero) for the text in Inkscape, but the package 'transparent.sty' is not loaded}%
    \renewcommand\transparent[1]{}%
  }%
  \providecommand\rotatebox[2]{#2}%
  \ifx\svgwidth\undefined%
    \setlength{\unitlength}{617bp}%
    \ifx\svgscale\undefined%
      \relax%
    \else%
      \setlength{\unitlength}{\unitlength * \real{\svgscale}}%
    \fi%
  \else%
    \setlength{\unitlength}{\svgwidth}%
  \fi%
  \global\let\svgwidth\undefined%
  \global\let\svgscale\undefined%
  \makeatother%
  \begin{picture}(1,0.66126418)%
    \put(0,0){\includegraphics[width=\unitlength]{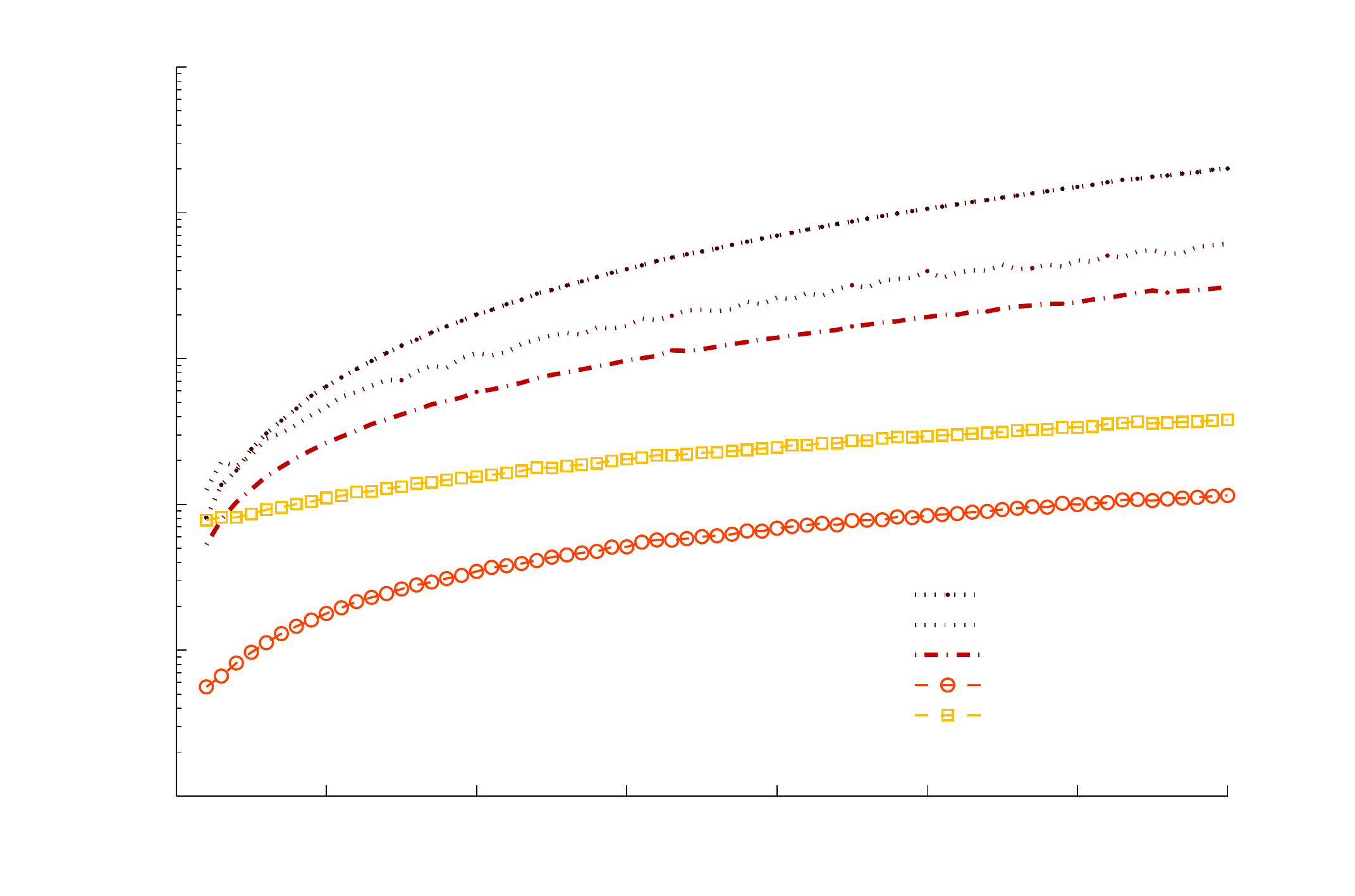}}%
    \put(0.12574263,0.05402464){\makebox(0,0)[lb]{\smash{\tiny 0}}}%
    \put(0.23203728,0.05402464){\makebox(0,0)[lb]{\smash{\tiny 50}}}%
    \put(0.33833063,0.05402464){\makebox(0,0)[lb]{\smash{\tiny 100}}}%
    \put(0.44921718,0.05402464){\makebox(0,0)[lb]{\smash{\tiny 150}}}%
    \put(0.56010211,0.05402464){\makebox(0,0)[lb]{\smash{\tiny 200}}}%
    \put(0.67098865,0.05402464){\makebox(0,0)[lb]{\smash{\tiny 250}}}%
    \put(0.7818752,0.05402464){\makebox(0,0)[lb]{\smash{\tiny 300}}}%
    \put(0.89276013,0.05402464){\makebox(0,0)[lb]{\smash{\tiny 350}}}%
    \put(0.08743922,0.06766613){\makebox(0,0)[lb]{\smash{\tiny $10^{-4}$}}}%
    \put(0.08743922,0.17531118){\makebox(0,0)[lb]{\smash{\tiny $10^{-3}$}}}%
    \put(0.08743922,0.28295462){\makebox(0,0)[lb]{\smash{\tiny $10^{-2}$}}}%
    \put(0.08743922,0.39059968){\makebox(0,0)[lb]{\smash{\tiny $10^{-1}$}}}%
    \put(0.08743922,0.49824473){\makebox(0,0)[lb]{\smash{\tiny $10^0$}}}%
    \put(0.08743922,0.60575365){\makebox(0,0)[lb]{\smash{\tiny $10^1$}}}%
    \put(0.72893031,0.21636953){\makebox(0,0)[lb]{\smash{Algorithm \ref{alg:svd2}}}}%
    \put(0.72893031,0.1942188){\makebox(0,0)[lb]{\smash{Algorithm \ref{alg:dichproj1}}}}%
    \put(0.72893031,0.17206969){\makebox(0,0)[lb]{\smash{Algorithm \ref{A2}}}}%
    \put(0.72893031,0.14978412){\makebox(0,0)[lb]{\smash{Algorithm \ref{alg:ginalt}}}}%
    \put(0.72893031,0.12763371){\makebox(0,0)[lb]{\smash{Algorithm \ref{alg:wolfe}}}}%
    \put(0.51404699,0.03471118){\makebox(0,0)[lb]{\smash{$N$}}}%
    \put(0.06080519,0.33068557){\rotatebox{90}{\makebox(0,0)[lb]{\smash{$\tau$}}}}%
  \end{picture}%
\endgroup%
  \caption{\label{fig:toymodel_times}Comparing the execution time $\tau$ of the various algorithms using MATLAB's timing functionality.  Each algorithm is executed using the cocycle data $\left\{A(T^{-N}x),\ldots, A(T^Nx)\right\}$.}
\end{figure}
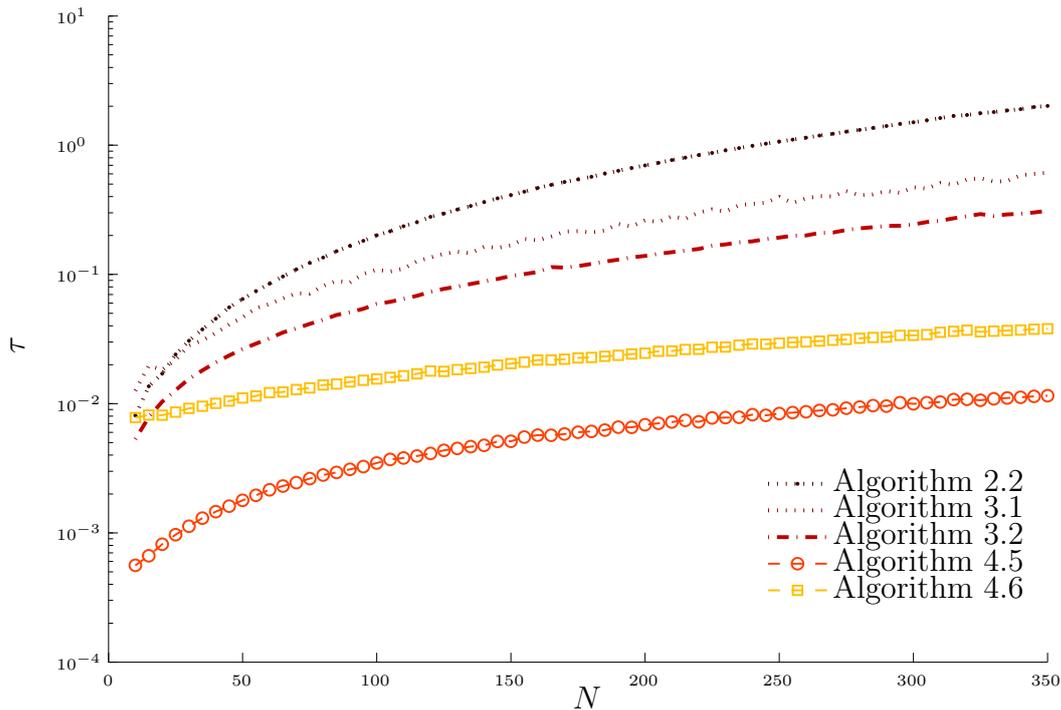

Finally, Figure \ref{fig:toymodel_times} shows the execution times of Algorithms \ref{alg:svd2}, \ref{alg:dichproj1}, \ref{A2}, {\ref{alg:ginalt}} and \ref{alg:wolfe}, which were timed using MATLAB's timing functionality.  The most time-consuming step in Algorithm {\ref{alg:ginalt}} is the SVD performed as part of the alterations from Section \ref{sec:limited_ginelli}.  Algorithm \ref{alg:wolfe} must perform two SVDs and Algorithm \ref{alg:svd2} must perform {many} more, which accounts for their longer execution times.

\subsection{Case Study 2: Particle dynamics - two disks in a quasi-one-dimensional box \label{sec:case_study_2}}
We consider the quasi-one-dimensional heat system studied extensively by Morriss \emph{et al. } \cite{Chung2010,Morriss2009,Robinson2008,Taniguchi2007,Taniguchi2005,Taniguchi2005a} which consists of two disks of diameter $\sigma=1$ in a rectangular box, $[0,L_x]\times[0,L_y]$, in which the shorter side, has length $L_y < 2\sigma$ so that the disks may not change order.  The two disks interact elastically with each other and the short walls, but periodic boundary conditions are enforced in the $y$-direction. The phase space of the system is then the set $X\subset \mathbb R^8$
\begin{align}
  X = \left(\left([0,L_x]\times[0,L_y]\right) / \sim\right)^2\times \mathbb R^2 \times \mathbb R^2 \notag
\end{align}
where $\sim$ is the equivalence class associated with periodic boundary conditions, that is, $(x_1,y_1),(x_2,y_2)\in[0,L_x]\times[0,L_y]$ have $(x_1,y_1) \sim (x_2,y_2)$ if $y_1 = y_2 \mod L_y$ and $x_1=x_2$.

The flow $\phi^\tau:X\to X$ consists of free-flight maps of time $\tau$, $\mathcal F^\tau:X\to X$, and collision maps $\mathcal C:X\to X$ so that $\phi^\tau(x) = \mathcal C \circ \mathcal F^{\tau_n}\circ\cdots\circ \mathcal F^{\tau_2}\circ \mathcal C \circ \mathcal F^{\tau_1}(x)$ where $\tau_1 + \cdots + \tau_n = \tau$. We consider a discrete-time version of the system by mapping from the instant after collision to the instant after the next collision, that is $x\mapsto \mathcal C \circ \mathcal F^{\tau(x)}(x)$ where $\tau(x)$ is the free-flight time in the continuous system.  The matrix cocycle is generated by the $8\times 8$ Jacobian matrices or the derivative of the flow evaluated instantly after each collision (i.e. $A(x) = D\left(\mathcal C \circ \mathcal F^{\tau(x)}\right)(x)$, see \cite{Chung2010} for details).  Due to a number of dynamic symmetries the system has Lyapunov exponents $\lambda_1 > \lambda_2 > 0 > -\lambda_2 > -\lambda_1$ with multiplicities $1,1,4,1$ and $1$ respectively.
This system has some symmetry, a high variation in expansion rates from
iteration to iteration, a well separated spectrum, invertible Jacobian matrices, and relatively
low dimension.
Numerical integration of an orbit consisting of 4646 collisions yielded a sequence of Jacobian matrices $\left\{A(T^{-2323}x),\ldots,A(T^{2322}x)\right\}$ which generate the cocycle $A$.

For this model we use the same choice of parameters to execute the algorithms as with the previous model:
\begin{itemize}
\item \textbf{Algorithm \ref{alg:svd2}:} $M = N$ and $\{N_k\} = \{1,6,\ldots,5k-4,\ldots,5K-4,N\}$ where $5K-4 < N \leq 5K+1$.
\item \textbf{Algorithm \ref{alg:dichproj1}:} We estimate the three largest Lyapunov exponents $\lambda_1 > \lambda_2 >\lambda_3$ and set $\Lambda^\text{right} = {\lambda_2 + 0.1(\lambda_1 - \lambda_2)}$ and $\Lambda^\text{left} = {\lambda_2 - 0.1(\lambda_2 - \lambda_3})$.
\item \textbf{Algorithm \ref{A2}:} As for Algorithm \ref{alg:dichproj1}.
\item \textbf{Algorithm \ref{alg:ginalt}:}  $M = N$, $M' = 5$, and $c'= (0,1)$.

\item \textbf{Algorithm \ref{alg:wolfe}:} Let $M_1 = N$, $M_1' = 5$ and $M_2 = N$.
\end{itemize}

\subsubsection{Criteria to assess the accuracy of estimated Oseledets spaces}

Since the Oseledets subspaces for this model are unknown, we test the approximations for two properties of Oseledets subspaces, namely their equivariance and the expansion rate, which defines the corresponding Lyapunov exponent.

\sloppy\paragraph{Equivariance:} To test for equivariance, we approximate the second Oseledets vector, $w_2^{(N)}(T^nx)$, at each time $n=0,1,\ldots,30$.  We then compute $\norm{\mathcal N\left(A(x,n)w_2^{(N)}(x)\right) - w_2^{(N)}(T^nx)}$ and plot the result, where  $v \mapright{\mathcal{N}} v/\norm{v}$.  If the approximations are equivariant this value would be zero.

\paragraph{Expansion Rate:}To test the expansion rate, each approach is used to compute the second Oseledets vector, $w_2^{(N)}(x) \in W_2^{(N)}(x)$, at time $n=0$ and we plot  $\frac{1}{m}\log\norm{A(x,m)w_2^{(N)}(x)}$ versus $m$.  If $W^{(N)}_2(x)$ is accurate, elements of $W^{(N)}_2(x)$ should grow at the correct rate: $\lambda_2$.

Whilst the Oseledets vector $w_2(x)$ must satisfy the above two properties, we must be careful when examining the results of these numerical experiments.  For instance, (i) it is possible to choose vectors that are equivariant despite not being contained in any single Oseledets subspace, and (ii) any element of $V_2(x)\backslash V_3(x) \varsupsetneq W_2(x)$ (a much larger set than $W_2(x)$) has Lyapunov exponent $\lambda_2$.

\subsubsection{Numerical Results}
Figure \ref{fig:ball_invar} shows the results of the equivariance test for the quasi-one-dimensional two disk model.  At the lower end of cocycle data length  ($N=75$) all Algorithms except \ref{A2} display reasonable equivariance, {although} Algorithm \ref{alg:dichproj1} remains equivariant for only a handful of steps. For $N=150$ and $N=225$ all approaches appear to produce close to equivariant results (note the changing scales in the vertical direction), with Algorithms \ref{alg:dichproj1} and \ref{A2} lagging behind when $N=225$.

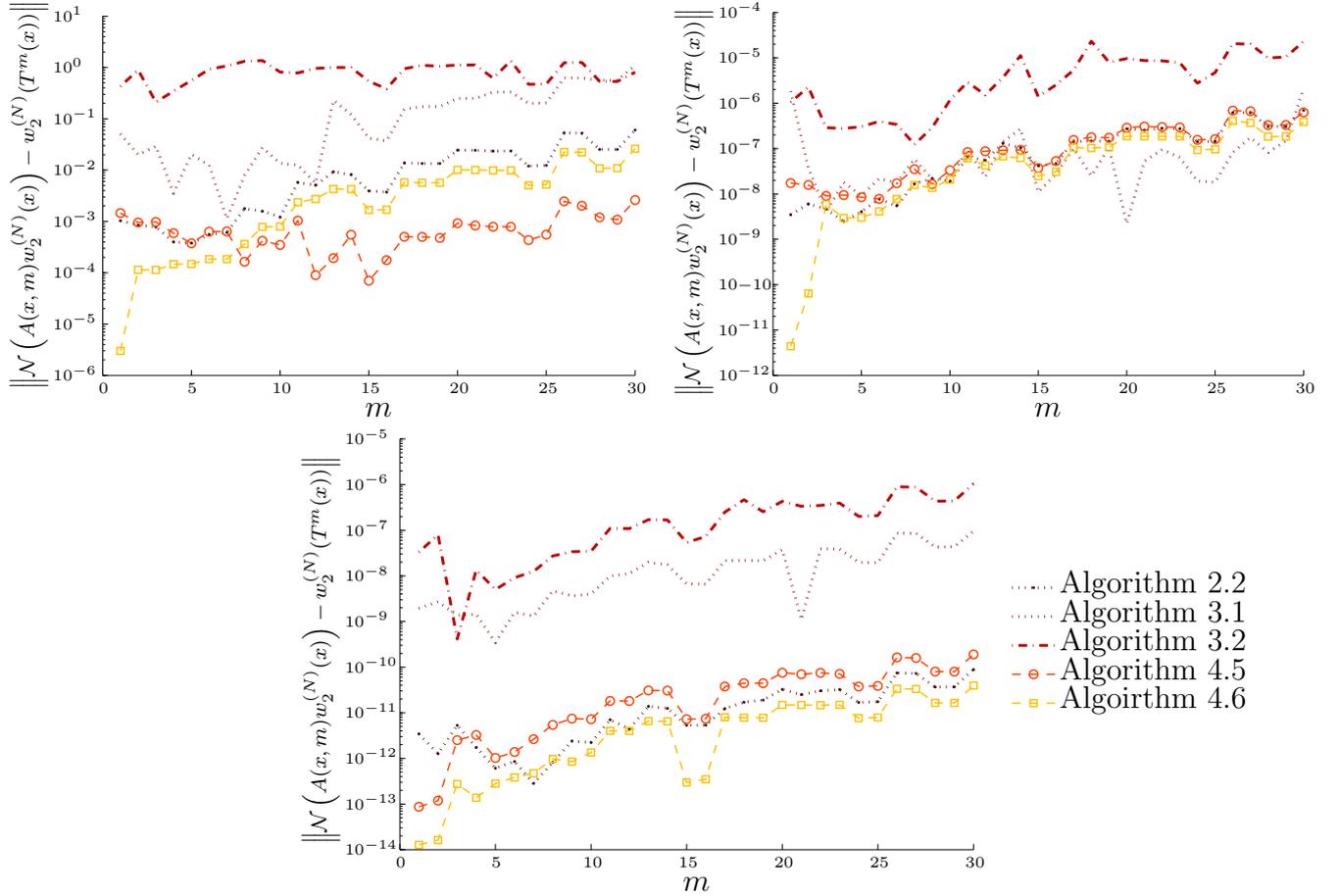
\begin{figure}[!hbt]
  \begin{center}
  \def\svgwidth{1.1\textwidth}
\begingroup%
  \makeatletter%
  \providecommand\color[2][]{%
    \errmessage{(Inkscape) Color is used for the text in Inkscape, but the package 'color.sty' is not loaded}%
    \renewcommand\color[2][]{}%
  }%
  \providecommand\transparent[1]{%
    \errmessage{(Inkscape) Transparency is used (non-zero) for the text in Inkscape, but the package 'transparent.sty' is not loaded}%
    \renewcommand\transparent[1]{}%
  }%
  \providecommand\rotatebox[2]{#2}%
  \ifx\svgwidth\undefined%
    \setlength{\unitlength}{908bp}%
    \ifx\svgscale\undefined%
      \relax%
    \else%
      \setlength{\unitlength}{\unitlength * \real{\svgscale}}%
    \fi%
  \else%
    \setlength{\unitlength}{\svgwidth}%
  \fi%
  \global\let\svgwidth\undefined%
  \global\let\svgscale\undefined%
  \makeatother%
  \begin{picture}(1,0.719163)%
    \put(0,0){\includegraphics[width=\unitlength]{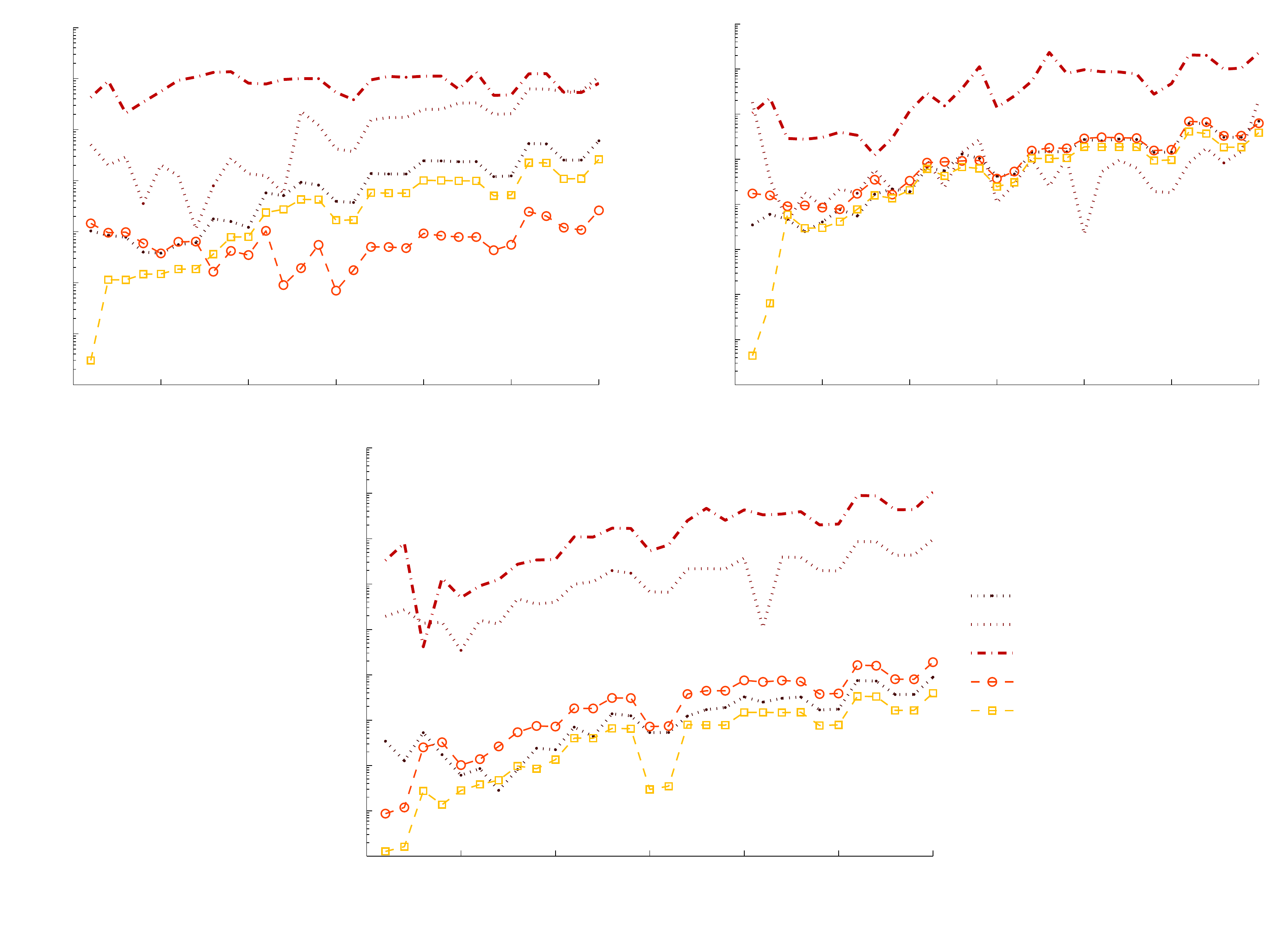}}%
    \put(0.79038216,0.25238656){\makebox(0,0)[lb]{\smash{Algorithm \ref{alg:svd2}}}}%
    \put(0.79038216,0.2300848){\makebox(0,0)[lb]{\smash{Algorithm \ref{alg:dichproj1}}}}%
    \put(0.79038216,0.20778304){\makebox(0,0)[lb]{\smash{Algorithm \ref{A2}}}}%
    \put(0.79038216,0.18548128){\makebox(0,0)[lb]{\smash{Algorithm \ref{alg:ginalt}}}}%
    \put(0.79038216,0.16317952){\makebox(0,0)[lb]{\smash{Algoirthm \ref{alg:wolfe}}}}%
    \put(0.258076,0.38852754){\makebox(0,0)[lb]{\smash{$m$}}}%
    \put(0.00446035,0.40991141){\rotatebox{90}{\makebox(0,0)[lb]{\smash{\scriptsize $\norm{\mathcal N\left(A(x,m)w_2^{(N)}(x)\right) - w_2^{(N)}(T^m(x))}$}}}}%
    \put(0.77120044,0.38852754){\makebox(0,0)[lb]{\smash{$m$}}}%
    \put(0.51428083,0.40424009){\rotatebox{90}{\makebox(0,0)[lb]{\smash{\scriptsize $\norm{\mathcal N\left(A(x,m)w_2^{(N)}(x)\right) - w_2^{(N)}(T^m(x))}$}}}}%
    \put(0.50165198,0.0222467){\makebox(0,0)[lb]{\smash{$m$}}}%
    \put(0.22821255,0.05816828){\rotatebox{90}{\makebox(0,0)[lb]{\smash{\scriptsize $\norm{\mathcal N\left(A(x,m)w_2^{(N)}(x)\right) - w_2^{(N)}(T^m(x))}$}}}}%
    \put(0.05387313,0.40693833){\makebox(0,0)[lb]{\smash{\tiny 0}}}%
    \put(0.12187996,0.40693833){\makebox(0,0)[lb]{\smash{\tiny 5}}}%
    \put(0.18685793,0.40693833){\makebox(0,0)[lb]{\smash{\tiny 10}}}%
    \put(0.25486454,0.40693833){\makebox(0,0)[lb]{\smash{\tiny 15}}}%
    \put(0.32287115,0.40693833){\makebox(0,0)[lb]{\smash{\tiny 20}}}%
    \put(0.39087775,0.40693833){\makebox(0,0)[lb]{\smash{\tiny 25}}}%
    \put(0.45897577,0.40693833){\makebox(0,0)[lb]{\smash{\tiny 30}}}%
    \put(0.01903436,0.41620814){\makebox(0,0)[lb]{\smash{\tiny $10^{-6}$}}}%
    \put(0.01903436,0.45585573){\makebox(0,0)[lb]{\smash{\tiny $10^{-5}$}}}%
    \put(0.01903436,0.4955033){\makebox(0,0)[lb]{\smash{\tiny $10^{-4}$}}}%
    \put(0.01903436,0.53515088){\makebox(0,0)[lb]{\smash{\tiny $10^{-3}$}}}%
    \put(0.01903436,0.57479846){\makebox(0,0)[lb]{\smash{\tiny $10^{-2}$}}}%
    \put(0.01903436,0.61444604){\makebox(0,0)[lb]{\smash{\tiny $10^{-1}$}}}%
    \put(0.01903436,0.65409361){\makebox(0,0)[lb]{\smash{\tiny $10^{0}$}}}%
    \put(0.01903436,0.69364978){\makebox(0,0)[lb]{\smash{\tiny $10^1$}}}%
    \put(0.56773128,0.40693833){\makebox(0,0)[lb]{\smash{\tiny 0}}}%
    \put(0.63546256,0.40693833){\makebox(0,0)[lb]{\smash{\tiny 5}}}%
    \put(0.70025661,0.40693833){\makebox(0,0)[lb]{\smash{\tiny 10}}}%
    \put(0.76798789,0.40693833){\makebox(0,0)[lb]{\smash{\tiny 15}}}%
    \put(0.83581167,0.40693833){\makebox(0,0)[lb]{\smash{\tiny 20}}}%
    \put(0.90354295,0.40693833){\makebox(0,0)[lb]{\smash{\tiny 25}}}%
    \put(0.97136564,0.40693833){\makebox(0,0)[lb]{\smash{\tiny 30}}}%
    \put(0.52885461,0.41620814){\makebox(0,0)[lb]{\smash{\tiny $10^{-12}$}}}%
    \put(0.52885463,0.45126652){\makebox(0,0)[lb]{\smash{\tiny $10^{-11}$}}}%
    \put(0.52885463,0.48632599){\makebox(0,0)[lb]{\smash{\tiny 10${-10}$}}}%
    \put(0.52885463,0.52138436){\makebox(0,0)[lb]{\smash{\tiny $10^{-9}$}}}%
    \put(0.52885463,0.55635132){\makebox(0,0)[lb]{\smash{\tiny $10^{-8}$}}}%
    \put(0.52885463,0.59140969){\makebox(0,0)[lb]{\smash{\tiny $10^{-7}$}}}%
    \put(0.52885463,0.62646916){\makebox(0,0)[lb]{\smash{\tiny $10^{-6}$}}}%
    \put(0.52885463,0.66152753){\makebox(0,0)[lb]{\smash{\tiny $10^{-5}$}}}%
    \put(0.52885463,0.69649449){\makebox(0,0)[lb]{\smash{\tiny $10^{-4}$}}}%
    \put(0.281663,0.04065727){\makebox(0,0)[lb]{\smash{\tiny 0}}}%
    \put(0.35490088,0.04065727){\makebox(0,0)[lb]{\smash{\tiny 5}}}%
    \put(0.42520154,0.04065727){\makebox(0,0)[lb]{\smash{\tiny 10}}}%
    \put(0.49843943,0.04065727){\makebox(0,0)[lb]{\smash{\tiny 15}}}%
    \put(0.57176982,0.04065727){\makebox(0,0)[lb]{\smash{\tiny 20}}}%
    \put(0.64500771,0.04065727){\makebox(0,0)[lb]{\smash{\tiny 25}}}%
    \put(0.718337,0.04065727){\makebox(0,0)[lb]{\smash{\tiny 30}}}%
    \put(0.24278632,0.04992643){\makebox(0,0)[lb]{\smash{\tiny $10^{-14}$}}}%
    \put(0.24278634,0.08516872){\makebox(0,0)[lb]{\smash{\tiny $10^{-13}$}}}%
    \put(0.24278634,0.12041079){\makebox(0,0)[lb]{\smash{\tiny $10^{-12}$}}}%
    \put(0.24278634,0.15565308){\makebox(0,0)[lb]{\smash{\tiny $10^{-11}$}}}%
    \put(0.24278634,0.19089537){\makebox(0,0)[lb]{\smash{\tiny $10^{-10}$}}}%
    \put(0.24278634,0.22613767){\makebox(0,0)[lb]{\smash{\tiny $10^{-9}$}}}%
    \put(0.24278634,0.26137996){\makebox(0,0)[lb]{\smash{\tiny $10^{-8}$}}}%
    \put(0.24278634,0.29662225){\makebox(0,0)[lb]{\smash{\tiny $10^{-7}$}}}%
    \put(0.24278634,0.33186454){\makebox(0,0)[lb]{\smash{\tiny $10^{-6}$}}}%
    \put(0.24278634,0.36719934){\makebox(0,0)[lb]{\smash{\tiny $10^{-5}$}}}%
  \end{picture}%
\endgroup%
  \caption{\label{fig:ball_invar} The equivariance test for the various algorithms on the quasi-one-dimensional two disk model. Each approach is used to approximate the second Oseledets vector, $w_2^{(N)}(T^nx)\in W_2^{(N)}(T^nx)$ {using cocycle data $\{A(T^{-N}x),\ldots, A(T^Nx)\}$,} at each time $n=0,1,\ldots,30$. We then compute $\norm{\mathcal{N}A(x,n)w_2^{(N)}(x) - w_2^{(N)}(T^nx)}$ and plot the result.  Note the different scales on each vertical axis.  The plots shown are for $N=75$ (top left), $N = 150$ (top right) and $N = 225$ (bottom).}
\end{center}
\end{figure}

Figure \ref{fig:ball_rates} shows the results of the expansion rate test for the quasi-one-dimensional two disk model for various amounts of cocycle data $\left\{A(T^{-N}x),\ldots, A(T^{N}x)\right\}$.  As expected, when there is a limited amount of data available ($N$ small) the approximations either expand at the higher rate of $\lambda_1$ or only expand at the rate of $\lambda_2$ for a brief time before the error grows too large.  As $N$ is increased, the approximations expand at $\lambda_2$ for longer periods, suggesting that they more accurately represent $w_2(x)$.

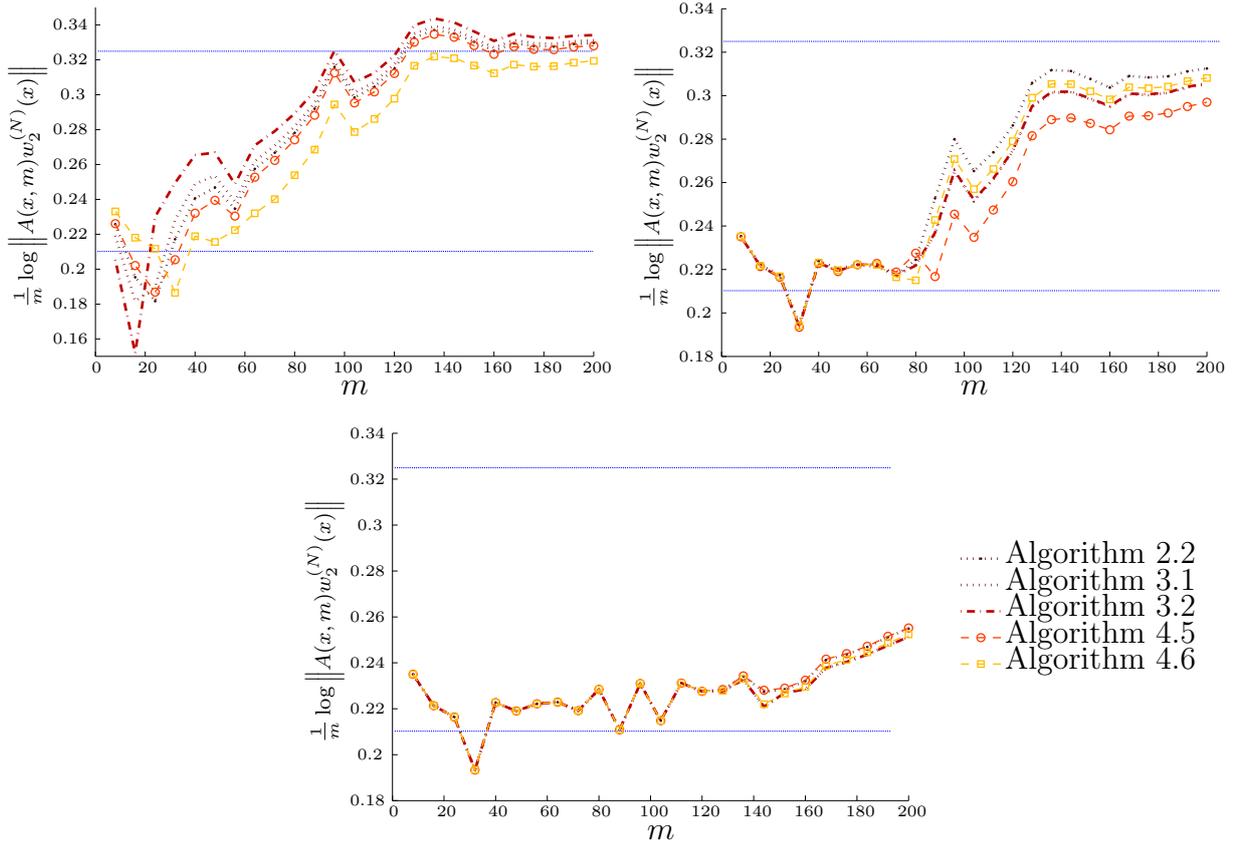
\begin{figure}[!hbt]
  \begin{center}
  \def\svgwidth{\textwidth}
\begingroup%
  \makeatletter%
  \providecommand\color[2][]{%
    \errmessage{(Inkscape) Color is used for the text in Inkscape, but the package 'color.sty' is not loaded}%
    \renewcommand\color[2][]{}%
  }%
  \providecommand\transparent[1]{%
    \errmessage{(Inkscape) Transparency is used (non-zero) for the text in Inkscape, but the package 'transparent.sty' is not loaded}%
    \renewcommand\transparent[1]{}%
  }%
  \providecommand\rotatebox[2]{#2}%
  \ifx\svgwidth\undefined%
    \setlength{\unitlength}{908bp}%
    \ifx\svgscale\undefined%
      \relax%
    \else%
      \setlength{\unitlength}{\unitlength * \real{\svgscale}}%
    \fi%
  \else%
    \setlength{\unitlength}{\svgwidth}%
  \fi%
  \global\let\svgwidth\undefined%
  \global\let\svgscale\undefined%
  \makeatother%
  \begin{picture}(1,0.719163)%
    \put(0,0){\includegraphics[width=\unitlength]{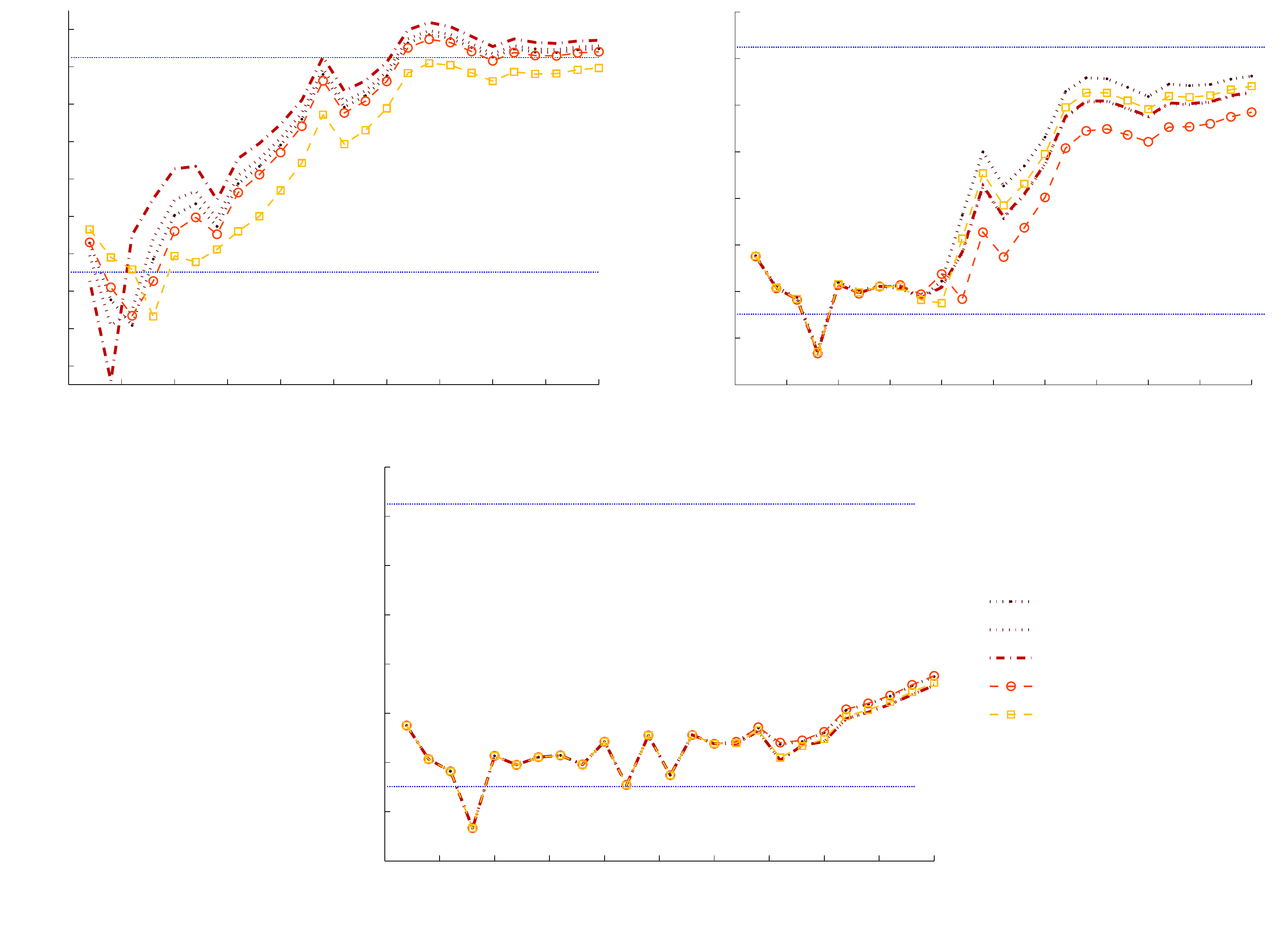}}%
    \put(0.80497467,0.24798128){\makebox(0,0)[lb]{\smash{Algorithm \ref{alg:svd2}}}}%
    \put(0.80497467,0.22595485){\makebox(0,0)[lb]{\smash{Algorithm \ref{alg:dichproj1}}}}%
    \put(0.80497467,0.20401982){\makebox(0,0)[lb]{\smash{Algorithm \ref{A2}}}}%
    \put(0.80497467,0.1820848){\makebox(0,0)[lb]{\smash{Algorithm \ref{alg:ginalt}}}}%
    \put(0.80497467,0.16005837){\makebox(0,0)[lb]{\smash{Algorithm \ref{alg:wolfe}}}}%
    \put(0.76844714,0.38852754){\makebox(0,0)[lb]{\smash{$m$}}}%
    \put(0.51925441,0.46002203){\rotatebox{90}{\makebox(0,0)[lb]{\smash{\scriptsize $\frac{1}{m}\log\norm{A(x,m)w_2^{(N)}(x)}$}}}}%
    \put(0.25624121,0.38852754){\makebox(0,0)[lb]{\smash{$m$}}}%
    \put(0.00172555,0.46057269){\rotatebox{90}{\makebox(0,0)[lb]{\smash{\scriptsize $\frac{1}{m}\log\norm{A(x,m)w_2^{(N)}(x)}$}}}}%
    \put(0.50917732,0.01857577){\makebox(0,0)[lb]{\smash{$m$}}}%
    \put(0.24722796,0.09823789){\rotatebox{90}{\makebox(0,0)[lb]{\smash{\scriptsize $\frac{1}{m}\log\norm{A(x,m)w_2^{(N)}(x)}$}}}}%
    \put(0.29570485,0.0369859){\makebox(0,0)[lb]{\smash{\tiny 0}}}%
    \put(0.33526101,0.0369859){\makebox(0,0)[lb]{\smash{\tiny 20}}}%
    \put(0.37793722,0.0369859){\makebox(0,0)[lb]{\smash{\tiny 40}}}%
    \put(0.42061344,0.0369859){\makebox(0,0)[lb]{\smash{\tiny 60}}}%
    \put(0.46328965,0.0369859){\makebox(0,0)[lb]{\smash{\tiny 80}}}%
    \put(0.50284471,0.0369859){\makebox(0,0)[lb]{\smash{\tiny 100}}}%
    \put(0.54542952,0.0369859){\makebox(0,0)[lb]{\smash{\tiny 120}}}%
    \put(0.58810573,0.0369859){\makebox(0,0)[lb]{\smash{\tiny 140}}}%
    \put(0.63078194,0.0369859){\makebox(0,0)[lb]{\smash{\tiny 160}}}%
    \put(0.67345815,0.0369859){\makebox(0,0)[lb]{\smash{\tiny 180}}}%
    \put(0.71613436,0.0369859){\makebox(0,0)[lb]{\smash{\tiny 200}}}%
    \put(0.26356498,0.04625551){\makebox(0,0)[lb]{\smash{\tiny 0.18}}}%
    \put(0.26971366,0.08452643){\makebox(0,0)[lb]{\smash{\tiny 0.2}}}%
    \put(0.26356498,0.12279736){\makebox(0,0)[lb]{\smash{\tiny 0.22}}}%
    \put(0.26356498,0.16106828){\makebox(0,0)[lb]{\smash{\tiny 0.24}}}%
    \put(0.26356498,0.1992478){\makebox(0,0)[lb]{\smash{\tiny 0.26}}}%
    \put(0.26356498,0.23751872){\makebox(0,0)[lb]{\smash{\tiny 0.28}}}%
    \put(0.26971366,0.27578965){\makebox(0,0)[lb]{\smash{\tiny 0.3}}}%
    \put(0.26356498,0.31406057){\makebox(0,0)[lb]{\smash{\tiny 0.32}}}%
    \put(0.26356498,0.35223899){\makebox(0,0)[lb]{\smash{\tiny 0.34}}}%
    \put(0.05020176,0.40693833){\makebox(0,0)[lb]{\smash{\tiny 0}}}%
    \put(0.08828943,0.40693833){\makebox(0,0)[lb]{\smash{\tiny 20}}}%
    \put(0.1294967,0.40693833){\makebox(0,0)[lb]{\smash{\tiny 40}}}%
    \put(0.17070485,0.40693833){\makebox(0,0)[lb]{\smash{\tiny 60}}}%
    \put(0.21182048,0.40693833){\makebox(0,0)[lb]{\smash{\tiny 80}}}%
    \put(0.24990859,0.40693833){\makebox(0,0)[lb]{\smash{\tiny 100}}}%
    \put(0.29111564,0.40693833){\makebox(0,0)[lb]{\smash{\tiny 120}}}%
    \put(0.33223238,0.40693833){\makebox(0,0)[lb]{\smash{\tiny 140}}}%
    \put(0.37343943,0.40693833){\makebox(0,0)[lb]{\smash{\tiny 160}}}%
    \put(0.41464758,0.40693833){\makebox(0,0)[lb]{\smash{\tiny 180}}}%
    \put(0.45585573,0.40693833){\makebox(0,0)[lb]{\smash{\tiny 200}}}%
    \put(0.01806167,0.43080066){\makebox(0,0)[lb]{\smash{\tiny 0.16}}}%
    \put(0.01806167,0.45980176){\makebox(0,0)[lb]{\smash{\tiny 0.18}}}%
    \put(0.02421057,0.48889537){\makebox(0,0)[lb]{\smash{\tiny 0.2}}}%
    \put(0.01806167,0.51798899){\makebox(0,0)[lb]{\smash{\tiny 0.22}}}%
    \put(0.01806167,0.54699009){\makebox(0,0)[lb]{\smash{\tiny 0.24}}}%
    \put(0.01806167,0.5760837){\makebox(0,0)[lb]{\smash{\tiny 0.26}}}%
    \put(0.01806167,0.6050848){\makebox(0,0)[lb]{\smash{\tiny 0.28}}}%
    \put(0.02421057,0.63417841){\makebox(0,0)[lb]{\smash{\tiny 0.3}}}%
    \put(0.01806167,0.66327093){\makebox(0,0)[lb]{\smash{\tiny 0.32}}}%
    \put(0.01806167,0.69227313){\makebox(0,0)[lb]{\smash{\tiny 0.34}}}%
    \put(0.56773128,0.40693833){\makebox(0,0)[lb]{\smash{\tiny 0}}}%
    \put(0.60480947,0.40693833){\makebox(0,0)[lb]{\smash{\tiny 20}}}%
    \put(0.6449152,0.40693833){\makebox(0,0)[lb]{\smash{\tiny 40}}}%
    \put(0.68502203,0.40693833){\makebox(0,0)[lb]{\smash{\tiny 60}}}%
    \put(0.72512885,0.40693833){\makebox(0,0)[lb]{\smash{\tiny 80}}}%
    \put(0.76211454,0.40693833){\makebox(0,0)[lb]{\smash{\tiny 100}}}%
    \put(0.80222137,0.40693833){\makebox(0,0)[lb]{\smash{\tiny 120}}}%
    \put(0.84232709,0.40693833){\makebox(0,0)[lb]{\smash{\tiny 140}}}%
    \put(0.88243392,0.40693833){\makebox(0,0)[lb]{\smash{\tiny 160}}}%
    \put(0.92254075,0.40693833){\makebox(0,0)[lb]{\smash{\tiny 180}}}%
    \put(0.96264758,0.40693833){\makebox(0,0)[lb]{\smash{\tiny 200}}}%
    \put(0.53559141,0.41620815){\makebox(0,0)[lb]{\smash{\tiny 0.18}}}%
    \put(0.54174009,0.45246035){\makebox(0,0)[lb]{\smash{\tiny 0.2}}}%
    \put(0.53559141,0.48862004){\makebox(0,0)[lb]{\smash{\tiny 0.22}}}%
    \put(0.53559141,0.52487225){\makebox(0,0)[lb]{\smash{\tiny 0.24}}}%
    \put(0.53559141,0.56103194){\makebox(0,0)[lb]{\smash{\tiny 0.26}}}%
    \put(0.53559141,0.59719163){\makebox(0,0)[lb]{\smash{\tiny 0.28}}}%
    \put(0.54174009,0.63344383){\makebox(0,0)[lb]{\smash{\tiny 0.3}}}%
    \put(0.53559141,0.66960352){\makebox(0,0)[lb]{\smash{\tiny 0.32}}}%
    \put(0.53559141,0.70576432){\makebox(0,0)[lb]{\smash{\tiny 0.34}}}%
  \end{picture}%
\endgroup%
\end{center}
\caption{\label{fig:ball_rates} The expansion rate test for the various approaches on the quasi-one-dimensional two disk model. The second Oseledets vector, $w_2^{(N)}(x)\in W_2^{(N)}(x)$, is approximated { using cocycle data $\{A(T^{-N}x),\ldots,A(T^Nx)\}$} and we plot $\frac{1}{m}\log\norm{A(x,m)w_2^{(N)}(x)}$ versus $m$.  If the approximation is accurate this quantity should tend to the value of $\lambda_2 \approx 0.210$, otherwise it would tend to the value of $\lambda_1 \approx 0.325$ both of which are shown in blue.  The plots shown are for $N=25$ (top left), $N=75$ (top right) and $N=150$ (bottom).}
\end{figure}

Most algorithms perform similarly regarding expansion rate.  { Note that the amount of cocycle data (size of $N$) needed to perform well in the Expansion Rate test is less than that needed to perform well in the Equivariance test - this demonstrates the importance of good performance in both tests in order to assess whether or not the algorithms are performing well.}

\subsection{Case Study 3: Time-dependent fluid flow in a cylinder;  a transfer operator description\label{sec:fluidflow}}

An important emerging application for Oseledets subspaces is the detection of strange eigenmodes, persistent patterns, and coherent sets for aperiodic time-dependent fluid flows.
In the periodic setting  strange eigenmodes have been found as eigenfunctions of a Perron-Frobenius operator via classical Floquet theory;  \cite{pikovsky_popovych_03,liu_haller_04,popovych_pikovsky_eckhardt_07}.
However, in the aperiodic time-dependent setting, Floquet theory cannot be applied.
An extension to aperiodically driven flows was derived in \cite{FLS10}, based on the new multiplicative ergodic theory of \cite{Froyland2010}.
Discrete approximations of a Perron-Frobenius cocycle representing the aperiodic flow are constructed and in this aperiodic setting the leading sub-dominant Oseledets subspaces play the role of the leading sub-dominant eigenfunctions in the periodic forcing case.

We review the four methods of approximating Oseledets subspaces with the aperiodically driven cylinder flow from \cite{FLS10}.
The flow domain is $Y=[0,2\pi]\times[0,\pi]$, $t\in\mathbb{R}^+$ and the flow is defined by the following forced ODE:
\begin{equation}\label{eq::travellingwavemixing}
\begin{split}
\dot{x} &=  c-\tilde{A}(\tilde{z}(t))\sin(x-\nu\tilde{z}(t))\cos(y)+\varepsilon G(g(x,y,\tilde{z}(t)))\sin(\tilde{z}(t)/2)\qquad\mod{2\pi}\\
\dot{y} &=  \tilde{A}(\tilde{z}(t))\cos(x-\nu \tilde{z}(t))\sin(y).\\
\end{split}
\end{equation}
Here, $\tilde{z}(t)=6.6685z_1(t)$, where $z_1(t)$ is generated by the standard Lorenz flow, $\tilde{A}(\tilde{z}(t))=1+0.125\sin(\sqrt{5}\tilde{z}(t))$,
$G(\psi):=1/{(\psi^2+1)}^2$ and the parameter function
$\psi=g(x,y,\tilde{z}(t)):=\sin(x-\nu\tilde{z}(t))\sin(y)+y/2-\pi/4$ vanishes at
the level set of the streamfunction of the unperturbed ($\varepsilon=0$) flow at
instantaneous time $t=0$, i.e., $s(x,y,0)=\pi/4$, which divides the
phase space in half.

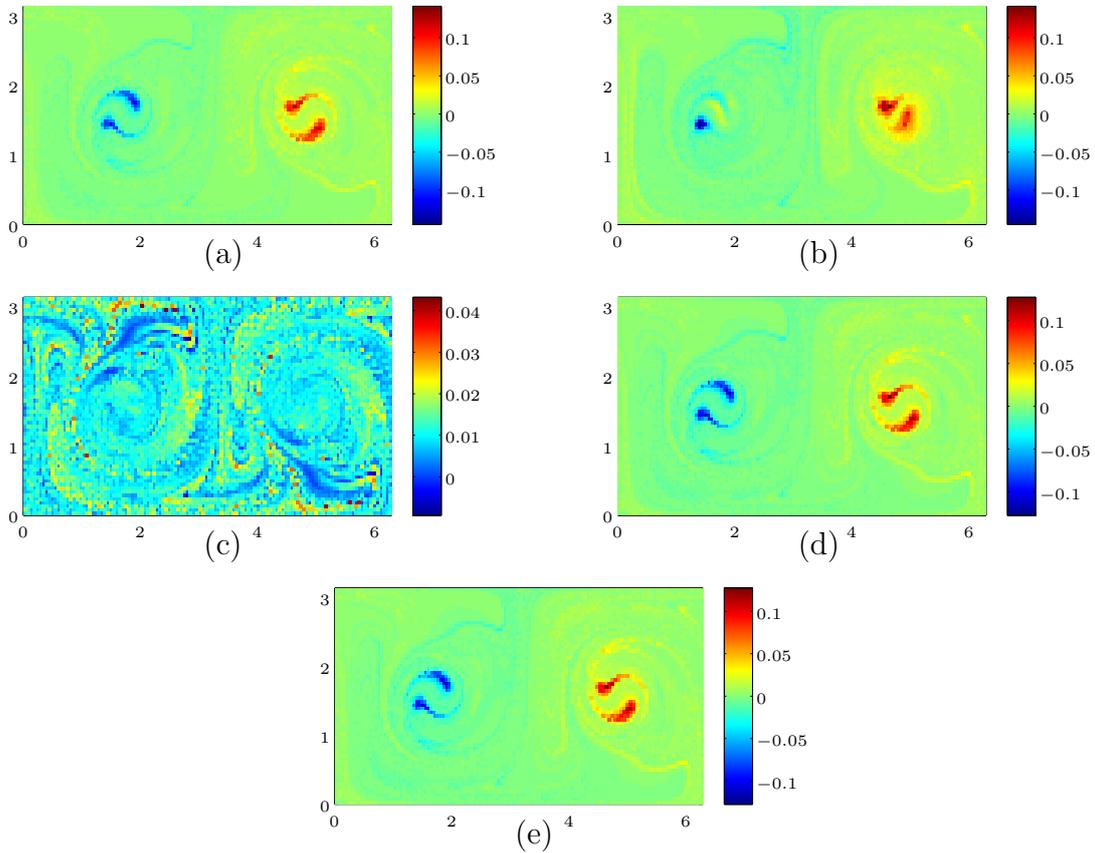
\begin{figure}[!hbt]
  \centering
  \def\svgwidth{\textwidth}
\begingroup%
  \makeatletter%
  \providecommand\color[2][]{%
    \errmessage{(Inkscape) Color is used for the text in Inkscape, but the package 'color.sty' is not loaded}%
    \renewcommand\color[2][]{}%
  }%
  \providecommand\transparent[1]{%
    \errmessage{(Inkscape) Transparency is used (non-zero) for the text in Inkscape, but the package 'transparent.sty' is not loaded}%
    \renewcommand\transparent[1]{}%
  }%
  \providecommand\rotatebox[2]{#2}%
  \ifx\svgwidth\undefined%
    \setlength{\unitlength}{855bp}%
    \ifx\svgscale\undefined%
      \relax%
    \else%
      \setlength{\unitlength}{\unitlength * \real{\svgscale}}%
    \fi%
  \else%
    \setlength{\unitlength}{\svgwidth}%
  \fi%
  \global\let\svgwidth\undefined%
  \global\let\svgscale\undefined%
  \makeatother%
  \begin{picture}(1,0.75906433)%
    \put(0,0){\includegraphics[width=\unitlength]{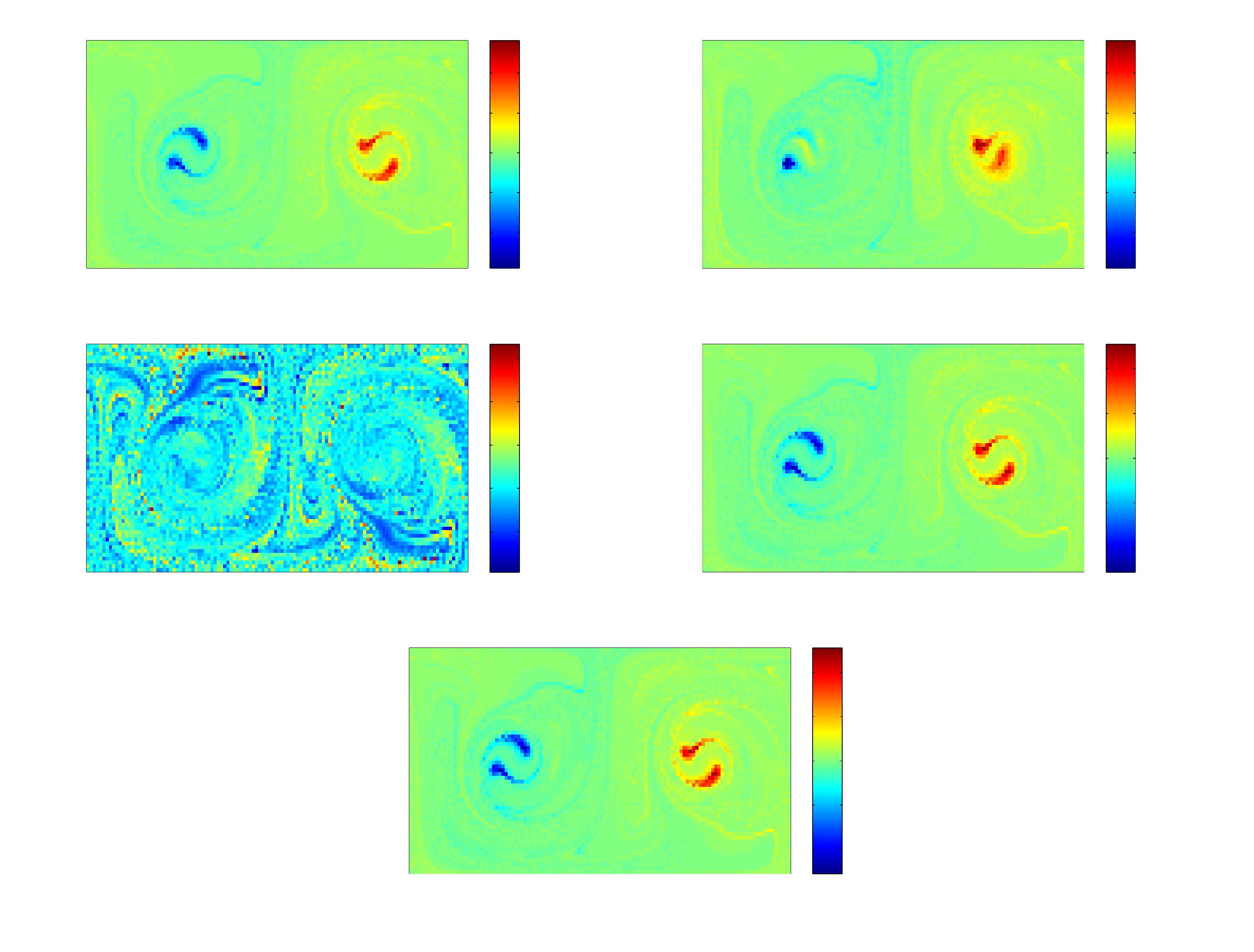}}%
    \put(0.06452257,0.52660819){\makebox(0,0)[lb]{\smash{\tiny 0}}}%
    \put(0.16130643,0.52660819){\makebox(0,0)[lb]{\smash{\tiny 2}}}%
    \put(0.25809006,0.52660819){\makebox(0,0)[lb]{\smash{\tiny 4}}}%
    \put(0.35497076,0.52660819){\makebox(0,0)[lb]{\smash{\tiny 6}}}%
    \put(0.05653006,0.53908421){\makebox(0,0)[lb]{\smash{\tiny 0}}}%
    \put(0.05653006,0.59697895){\makebox(0,0)[lb]{\smash{\tiny 1}}}%
    \put(0.05653006,0.65477661){\makebox(0,0)[lb]{\smash{\tiny 2}}}%
    \put(0.05653006,0.7125731){\makebox(0,0)[lb]{\smash{\tiny 3}}}%
    \put(0.21773918,0.51267135){\makebox(0,0)[lb]{\smash{(a) }}}%
    \put(0.41754386,0.56764211){\makebox(0,0)[lb]{\smash{\tiny $-0.1$}}}%
    \put(0.41754386,0.59931813){\makebox(0,0)[lb]{\smash{\tiny $-0.05$}}}%
    \put(0.41754386,0.6310924){\makebox(0,0)[lb]{\smash{\tiny 0}}}%
    \put(0.41754386,0.6628655){\makebox(0,0)[lb]{\smash{\tiny 0.05}}}%
    \put(0.41754386,0.69463977){\makebox(0,0)[lb]{\smash{\tiny 0.1}}}%
    \put(0.55584795,0.52660819){\makebox(0,0)[lb]{\smash{\tiny 0}}}%
    \put(0.65263158,0.52660819){\makebox(0,0)[lb]{\smash{\tiny 2}}}%
    \put(0.7494152,0.52660819){\makebox(0,0)[lb]{\smash{\tiny 4}}}%
    \put(0.84629591,0.52660819){\makebox(0,0)[lb]{\smash{\tiny 6}}}%
    \put(0.54785614,0.53908421){\makebox(0,0)[lb]{\smash{\tiny 0}}}%
    \put(0.54785614,0.59697895){\makebox(0,0)[lb]{\smash{\tiny 1}}}%
    \put(0.54785614,0.65477661){\makebox(0,0)[lb]{\smash{\tiny 2}}}%
    \put(0.54785614,0.7125731){\makebox(0,0)[lb]{\smash{\tiny 3}}}%
    \put(0.70906433,0.51267135){\makebox(0,0)[lb]{\smash{(b) }}}%
    \put(0.90886901,0.56764211){\makebox(0,0)[lb]{\smash{\tiny $-0.1$}}}%
    \put(0.90886901,0.59931813){\makebox(0,0)[lb]{\smash{\tiny $-0.05$}}}%
    \put(0.90886901,0.6310924){\makebox(0,0)[lb]{\smash{\tiny 0}}}%
    \put(0.90886901,0.6628655){\makebox(0,0)[lb]{\smash{\tiny 0.05}}}%
    \put(0.90886901,0.69463977){\makebox(0,0)[lb]{\smash{\tiny 0.1}}}%
    \put(0.06452257,0.2843076){\makebox(0,0)[lb]{\smash{\tiny 0}}}%
    \put(0.16130643,0.2843076){\makebox(0,0)[lb]{\smash{\tiny 2}}}%
    \put(0.25809006,0.2843076){\makebox(0,0)[lb]{\smash{\tiny 4}}}%
    \put(0.35497076,0.2843076){\makebox(0,0)[lb]{\smash{\tiny 6}}}%
    \put(0.05653006,0.29678363){\makebox(0,0)[lb]{\smash{\tiny 0}}}%
    \put(0.05653006,0.35467836){\makebox(0,0)[lb]{\smash{\tiny 1}}}%
    \put(0.05653006,0.41247602){\makebox(0,0)[lb]{\smash{\tiny 2}}}%
    \put(0.05653006,0.47027251){\makebox(0,0)[lb]{\smash{\tiny 3}}}%
    \put(0.21803158,0.27037076){\makebox(0,0)[lb]{\smash{(c) }}}%
    \put(0.41754386,0.32894737){\makebox(0,0)[lb]{\smash{\tiny 0}}}%
    \put(0.41754386,0.36354737){\makebox(0,0)[lb]{\smash{\tiny 0.01}}}%
    \put(0.41754386,0.39805029){\makebox(0,0)[lb]{\smash{\tiny 0.02}}}%
    \put(0.41754386,0.43265146){\makebox(0,0)[lb]{\smash{\tiny 0.03}}}%
    \put(0.41754386,0.46725146){\makebox(0,0)[lb]{\smash{\tiny 0.04}}}%
    \put(0.55584795,0.2843076){\makebox(0,0)[lb]{\smash{\tiny 0}}}%
    \put(0.65263158,0.2843076){\makebox(0,0)[lb]{\smash{\tiny 2}}}%
    \put(0.7494152,0.2843076){\makebox(0,0)[lb]{\smash{\tiny 4}}}%
    \put(0.84629591,0.2843076){\makebox(0,0)[lb]{\smash{\tiny 6}}}%
    \put(0.54785614,0.29678363){\makebox(0,0)[lb]{\smash{\tiny 0}}}%
    \put(0.54785614,0.35467836){\makebox(0,0)[lb]{\smash{\tiny 1}}}%
    \put(0.54785614,0.41247602){\makebox(0,0)[lb]{\smash{\tiny 2}}}%
    \put(0.54785614,0.47027251){\makebox(0,0)[lb]{\smash{\tiny 3}}}%
    \put(0.70906433,0.27037076){\makebox(0,0)[lb]{\smash{(d) }}}%
    \put(0.90886901,0.31637427){\makebox(0,0)[lb]{\smash{\tiny $-0.1$}}}%
    \put(0.90886901,0.35204678){\makebox(0,0)[lb]{\smash{\tiny $-0.05$}}}%
    \put(0.90886901,0.38762222){\makebox(0,0)[lb]{\smash{\tiny 0}}}%
    \put(0.90886901,0.42319649){\makebox(0,0)[lb]{\smash{\tiny 0.05}}}%
    \put(0.90886901,0.45886901){\makebox(0,0)[lb]{\smash{\tiny 0.1}}}%
    \put(0.32192982,0.04376234){\makebox(0,0)[lb]{\smash{\tiny 0}}}%
    \put(0.41871345,0.04376234){\makebox(0,0)[lb]{\smash{\tiny 2}}}%
    \put(0.51549708,0.04376234){\makebox(0,0)[lb]{\smash{\tiny 4}}}%
    \put(0.6122807,0.04376234){\makebox(0,0)[lb]{\smash{\tiny 6}}}%
    \put(0.31393801,0.05623766){\makebox(0,0)[lb]{\smash{\tiny 0}}}%
    \put(0.31393801,0.1135476){\makebox(0,0)[lb]{\smash{\tiny 1}}}%
    \put(0.31393801,0.17085731){\makebox(0,0)[lb]{\smash{\tiny 2}}}%
    \put(0.31393801,0.22816725){\makebox(0,0)[lb]{\smash{\tiny 3}}}%
    \put(0.4751462,0.02982456){\makebox(0,0)[lb]{\smash{(e) }}}%
    \put(0.6748538,0.07563368){\makebox(0,0)[lb]{\smash{\tiny $-0.1$}}}%
    \put(0.6748538,0.11091602){\makebox(0,0)[lb]{\smash{\tiny $-0.05$}}}%
    \put(0.6748538,0.14619883){\makebox(0,0)[lb]{\smash{\tiny 0}}}%
    \put(0.6748538,0.18157895){\makebox(0,0)[lb]{\smash{\tiny 0.05}}}%
    \put(0.6748538,0.21686199){\makebox(0,0)[lb]{\smash{\tiny 0.1}}}%
  \end{picture}%
\endgroup%
  \caption{The second Oseledets subspace as determined by (a) Algorithm \ref{alg:svd2}, (b) Algorithm \ref{alg:dichproj1}, (c) Algorithm \ref{A2}, (d) Algorithm \ref{alg:ginelli} and (e) Algorithm \ref{alg:wolfe}.}
  \label{fig:fluid_flow_1}
\end{figure}

We set
$\varepsilon=1$ as this value is sufficiently large to ensure no KAM
tori remain in the jet regime, but sufficiently small to maintain
islands originating from the nested periodic orbits around the
elliptic points of the unperturbed system.

We construct the discretised Perron-Frobenius matrices $P_x^{(\tau)}(t)=:A(x)$ as described in Section 3 of \cite{FLS10}, and briefly recapped in Example \ref{eg1.2}, using a uniform grid of $120\times 60$ boxes, $\tau=8$ and $-32\le t \le 32$. 
In total, we generate 8 such matrices of dimension $7200\times 7200$.
Thus, in this case study we have a limited amount of data, no symmetry, high dimension, and the matrices are non-invertible and sparse.

In order to obtain reasonable results we executed Algorithms \ref{alg:svd2}, \ref{alg:dichproj1}, \ref{A2}, \ref{alg:ginelli} and \ref{alg:wolfe} with the following parameters:

\begin{itemize}
\item \textbf{Algorithm \ref{alg:svd2}:} $M=N=4$ and $\{N_k\} = \{2,4\}$.
\item \textbf{Algorithm \ref{alg:dichproj1}:} We estimate the three largest Lyapunov exponents $\lambda_1 > \lambda_2 >\lambda_3$ and set $\Lambda^\text{right} = {\lambda_2 + 0.1(\lambda_1 - \lambda_2)}$ and $\Lambda^\text{left} = {\lambda_2 - 0.1(\lambda_2 - \lambda_3})$.
\item \textbf{Algorithm \ref{A2}:} As for Algorithm \ref{alg:dichproj1}.
\item \textbf{Algorithm \ref{alg:ginalt}:} $M' = M = 4$ (so that only an SVD is used, and no push-forward step), $N=4$ and $c' = (0,1)$.
\item \textbf{Algorithm \ref{alg:wolfe}:} { $M_1=M_1'=4$ and $M_2 = 4$.}
\end{itemize}

\begin{figure}[p]
  \floatpagestyle{empty}
  \centering
  \def\svgwidth{0.9\textwidth}
\begingroup%
  \makeatletter%
  \providecommand\color[2][]{%
    \errmessage{(Inkscape) Color is used for the text in Inkscape, but the package 'color.sty' is not loaded}%
    \renewcommand\color[2][]{}%
  }%
  \providecommand\transparent[1]{%
    \errmessage{(Inkscape) Transparency is used (non-zero) for the text in Inkscape, but the package 'transparent.sty' is not loaded}%
    \renewcommand\transparent[1]{}%
  }%
  \providecommand\rotatebox[2]{#2}%
  \ifx\svgwidth\undefined%
    \setlength{\unitlength}{665.03886719bp}%
    \ifx\svgscale\undefined%
      \relax%
    \else%
      \setlength{\unitlength}{\unitlength * \real{\svgscale}}%
    \fi%
  \else%
    \setlength{\unitlength}{\svgwidth}%
  \fi%
  \global\let\svgwidth\undefined%
  \global\let\svgscale\undefined%
  \makeatother%
  \begin{picture}(1,1.55621588)%
    \put(0,0){\includegraphics[width=\unitlength]{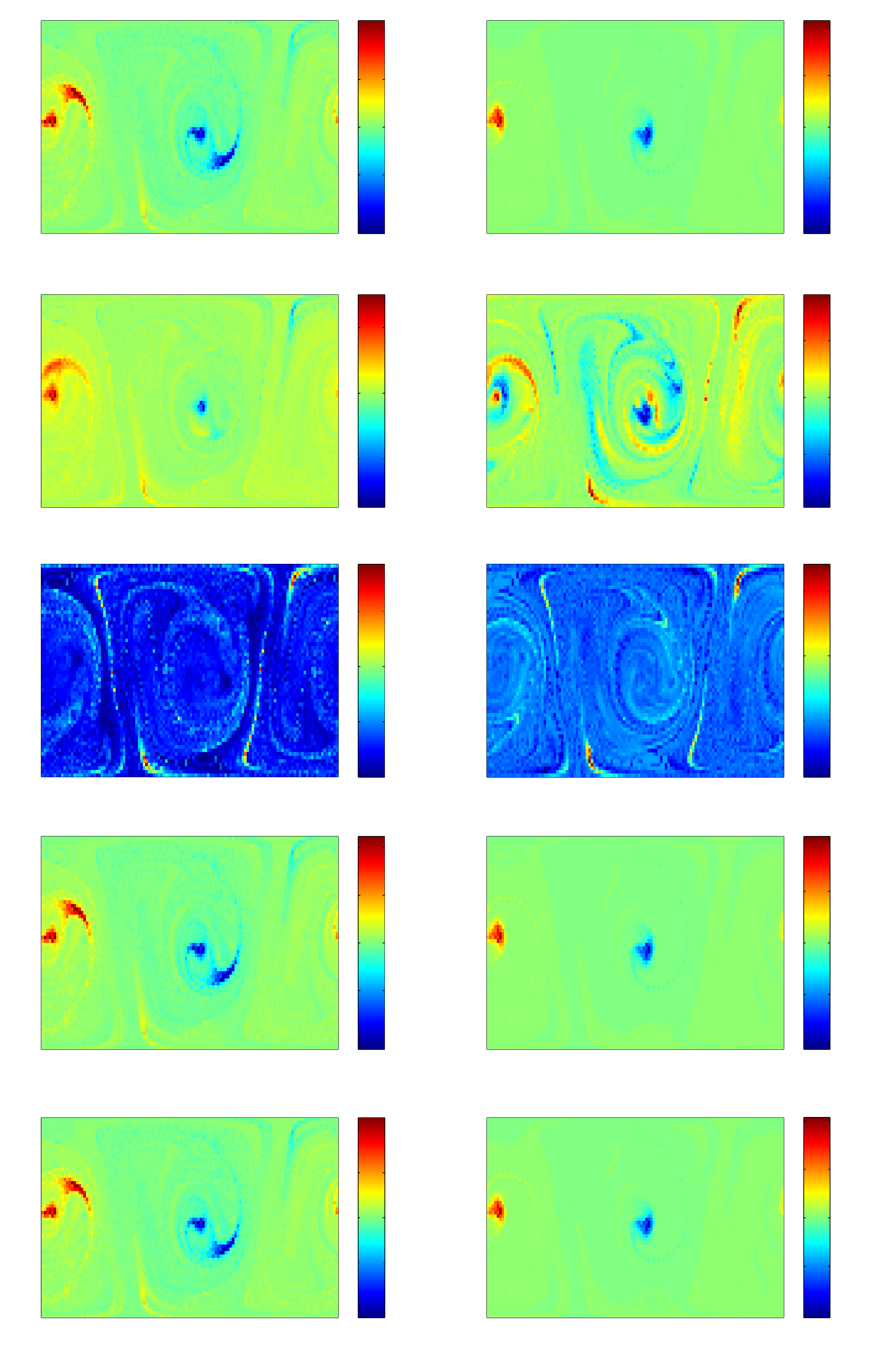}}%
    \put(0.04110015,1.26763605){\makebox(0,0)[lb]{\smash{\tiny 0}}}%
    \put(0.14823675,1.26763605){\makebox(0,0)[lb]{\smash{\tiny 2}}}%
    \put(0.25549935,1.26763605){\makebox(0,0)[lb]{\smash{\tiny 4}}}%
    \put(0.36263595,1.26763605){\makebox(0,0)[lb]{\smash{\tiny 6}}}%
    \put(0.03082527,1.28367541){\makebox(0,0)[lb]{\smash{\tiny 0}}}%
    \put(0.03082527,1.36061347){\makebox(0,0)[lb]{\smash{\tiny 1}}}%
    \put(0.03082527,1.43742583){\makebox(0,0)[lb]{\smash{\tiny 2}}}%
    \put(0.03082527,1.51423788){\makebox(0,0)[lb]{\smash{\tiny 3}}}%
    \put(0.19672986,1.2497177){\makebox(0,0)[lb]{\smash{(1a) }}}%
    \put(0.44057575,1.2959556){\makebox(0,0)[lb]{\smash{\tiny $-0.1$}}}%
    \put(0.44057575,1.35021288){\makebox(0,0)[lb]{\smash{\tiny $-0.05$}}}%
    \put(0.44057575,1.40434505){\makebox(0,0)[lb]{\smash{\tiny 0}}}%
    \put(0.44057575,1.45847723){\makebox(0,0)[lb]{\smash{\tiny 0.05}}}%
    \put(0.44057575,1.51273421){\makebox(0,0)[lb]{\smash{\tiny 0.1}}}%
    \put(0.54633348,1.26763605){\makebox(0,0)[lb]{\smash{\tiny 0}}}%
    \put(0.65347008,1.26763605){\makebox(0,0)[lb]{\smash{\tiny 2}}}%
    \put(0.76073298,1.26763605){\makebox(0,0)[lb]{\smash{\tiny 4}}}%
    \put(0.86786958,1.26763605){\makebox(0,0)[lb]{\smash{\tiny 6}}}%
    \put(0.53605889,1.28367541){\makebox(0,0)[lb]{\smash{\tiny 0}}}%
    \put(0.53605889,1.36061347){\makebox(0,0)[lb]{\smash{\tiny 1}}}%
    \put(0.53605889,1.43742583){\makebox(0,0)[lb]{\smash{\tiny 2}}}%
    \put(0.53605889,1.51423788){\makebox(0,0)[lb]{\smash{\tiny 3}}}%
    \put(0.70196348,1.2497177){\makebox(0,0)[lb]{\smash{(1b)
 }}}%
    \put(0.94580938,1.2876854){\makebox(0,0)[lb]{\smash{\tiny $-0.2$}}}%
    \put(0.94580938,1.34595268){\makebox(0,0)[lb]{\smash{\tiny $-0.1$}}}%
    \put(0.94580938,1.40434505){\makebox(0,0)[lb]{\smash{\tiny 0}}}%
    \put(0.94580938,1.46273713){\makebox(0,0)[lb]{\smash{\tiny 0.1}}}%
    \put(0.94580938,1.5210044){\makebox(0,0)[lb]{\smash{\tiny 0.2}}}%
    \put(0.04115858,0.95691727){\makebox(0,0)[lb]{\smash{\tiny 0}}}%
    \put(0.14829517,0.95691727){\makebox(0,0)[lb]{\smash{\tiny 2}}}%
    \put(0.25555778,0.95691727){\makebox(0,0)[lb]{\smash{\tiny 4}}}%
    \put(0.36269437,0.95691727){\makebox(0,0)[lb]{\smash{\tiny 6}}}%
    \put(0.03088369,0.97295664){\makebox(0,0)[lb]{\smash{\tiny 0}}}%
    \put(0.03088369,1.0498947){\makebox(0,0)[lb]{\smash{\tiny 1}}}%
    \put(0.03088369,1.12670705){\makebox(0,0)[lb]{\smash{\tiny 2}}}%
    \put(0.03088369,1.2035191){\makebox(0,0)[lb]{\smash{\tiny 3}}}%
    \put(0.1929043,0.93899892){\makebox(0,0)[lb]{\smash{(2a)}}}%
    \put(0.44063418,1.02771554){\makebox(0,0)[lb]{\smash{\tiny $-0.1$}}}%
    \put(0.44063418,1.10227239){\makebox(0,0)[lb]{\smash{\tiny 0}}}%
    \put(0.44063418,1.17682893){\makebox(0,0)[lb]{\smash{\tiny 0.1}}}%
    \put(0.54639191,0.95691727){\makebox(0,0)[lb]{\smash{\tiny 0}}}%
    \put(0.6535285,0.95691727){\makebox(0,0)[lb]{\smash{\tiny 2}}}%
    \put(0.7607914,0.95691727){\makebox(0,0)[lb]{\smash{\tiny 4}}}%
    \put(0.867928,0.95691727){\makebox(0,0)[lb]{\smash{\tiny 6}}}%
    \put(0.53611732,0.97295664){\makebox(0,0)[lb]{\smash{\tiny 0}}}%
    \put(0.53611732,1.0498947){\makebox(0,0)[lb]{\smash{\tiny 1}}}%
    \put(0.53611732,1.12670705){\makebox(0,0)[lb]{\smash{\tiny 2}}}%
    \put(0.53611732,1.2035191){\makebox(0,0)[lb]{\smash{\tiny 3}}}%
    \put(0.70415261,0.93899892){\makebox(0,0)[lb]{\smash{(2b)}}}%
    \put(0.94586781,1.03297839){\makebox(0,0)[lb]{\smash{\tiny $-0.05$}}}%
    \put(0.94586781,1.09751026){\makebox(0,0)[lb]{\smash{\tiny 0}}}%
    \put(0.94586781,1.16204333){\makebox(0,0)[lb]{\smash{\tiny 0.05}}}%
    \put(0.04110015,0.65095021){\makebox(0,0)[lb]{\smash{\tiny 0}}}%
    \put(0.14823675,0.65095021){\makebox(0,0)[lb]{\smash{\tiny 2}}}%
    \put(0.25549935,0.65095021){\makebox(0,0)[lb]{\smash{\tiny 4}}}%
    \put(0.36263595,0.65095021){\makebox(0,0)[lb]{\smash{\tiny 6}}}%
    \put(0.03082527,0.66698957){\makebox(0,0)[lb]{\smash{\tiny 0}}}%
    \put(0.03082527,0.74392763){\makebox(0,0)[lb]{\smash{\tiny 1}}}%
    \put(0.03082527,0.82073999){\makebox(0,0)[lb]{\smash{\tiny 2}}}%
    \put(0.03082527,0.89755204){\makebox(0,0)[lb]{\smash{\tiny 3}}}%
    \put(0.19886056,0.63303186){\makebox(0,0)[lb]{\smash{(3a)}}}%
    \put(0.44057575,0.66686447){\makebox(0,0)[lb]{\smash{\tiny 0}}}%
    \put(0.44057575,0.72964275){\makebox(0,0)[lb]{\smash{\tiny 0.02}}}%
    \put(0.44057575,0.79229503){\makebox(0,0)[lb]{\smash{\tiny 0.04}}}%
    \put(0.44057575,0.85494851){\makebox(0,0)[lb]{\smash{\tiny 0.06}}}%
    \put(0.54633348,0.6507){\makebox(0,0)[lb]{\smash{\tiny 0}}}%
    \put(0.65347008,0.6507){\makebox(0,0)[lb]{\smash{\tiny 2}}}%
    \put(0.76073298,0.6507){\makebox(0,0)[lb]{\smash{\tiny 4}}}%
    \put(0.86786958,0.6507){\makebox(0,0)[lb]{\smash{\tiny 6}}}%
    \put(0.53605889,0.66673876){\makebox(0,0)[lb]{\smash{\tiny 0}}}%
    \put(0.53605889,0.74367682){\makebox(0,0)[lb]{\smash{\tiny 1}}}%
    \put(0.53605889,0.82061518){\makebox(0,0)[lb]{\smash{\tiny 2}}}%
    \put(0.53605889,0.89755204){\makebox(0,0)[lb]{\smash{\tiny 3}}}%
    \put(0.70822928,0.63278105){\makebox(0,0)[lb]{\smash{(3b)}}}%
    \put(0.94580938,0.73152234){\makebox(0,0)[lb]{\smash{\tiny 0}}}%
    \put(0.94580938,0.80445071){\makebox(0,0)[lb]{\smash{\tiny 0.05}}}%
    \put(0.94580938,0.87737878){\makebox(0,0)[lb]{\smash{\tiny 0.1}}}%
    \put(0.04110015,0.34208864){\makebox(0,0)[lb]{\smash{\tiny 0}}}%
    \put(0.14823675,0.34208864){\makebox(0,0)[lb]{\smash{\tiny 2}}}%
    \put(0.25549935,0.34208864){\makebox(0,0)[lb]{\smash{\tiny 4}}}%
    \put(0.36263595,0.34208864){\makebox(0,0)[lb]{\smash{\tiny 6}}}%
    \put(0.03082527,0.35812801){\makebox(0,0)[lb]{\smash{\tiny 0}}}%
    \put(0.03082527,0.43506606){\makebox(0,0)[lb]{\smash{\tiny 1}}}%
    \put(0.03082527,0.51187842){\makebox(0,0)[lb]{\smash{\tiny 2}}}%
    \put(0.03082527,0.58869047){\makebox(0,0)[lb]{\smash{\tiny 3}}}%
    \put(0.19886056,0.32417029){\makebox(0,0)[lb]{\smash{(4a)}}}%
    \put(0.44057575,0.37040819){\makebox(0,0)[lb]{\smash{\tiny $-0.1$}}}%
    \put(0.44057575,0.42466547){\makebox(0,0)[lb]{\smash{\tiny $-0.05$}}}%
    \put(0.44057575,0.47879764){\makebox(0,0)[lb]{\smash{\tiny 0}}}%
    \put(0.44057575,0.53292982){\makebox(0,0)[lb]{\smash{\tiny 0.05}}}%
    \put(0.44057575,0.5871868){\makebox(0,0)[lb]{\smash{\tiny 0.1}}}%
    \put(0.54633348,0.34208864){\makebox(0,0)[lb]{\smash{\tiny 0}}}%
    \put(0.65347008,0.34208864){\makebox(0,0)[lb]{\smash{\tiny 2}}}%
    \put(0.76073298,0.34208864){\makebox(0,0)[lb]{\smash{\tiny 4}}}%
    \put(0.86786958,0.34208864){\makebox(0,0)[lb]{\smash{\tiny 6}}}%
    \put(0.53605889,0.35812801){\makebox(0,0)[lb]{\smash{\tiny 0}}}%
    \put(0.53605889,0.43506606){\makebox(0,0)[lb]{\smash{\tiny 1}}}%
    \put(0.53605889,0.51187842){\makebox(0,0)[lb]{\smash{\tiny 2}}}%
    \put(0.53605889,0.58869047){\makebox(0,0)[lb]{\smash{\tiny 3}}}%
    \put(0.7044701,0.32417029){\makebox(0,0)[lb]{\smash{(4b)}}}%
    \put(0.94580938,0.362138){\makebox(0,0)[lb]{\smash{\tiny $-0.2$}}}%
    \put(0.94580938,0.42040527){\makebox(0,0)[lb]{\smash{\tiny $-0.1$}}}%
    \put(0.94580938,0.47879764){\makebox(0,0)[lb]{\smash{\tiny 0}}}%
    \put(0.94580938,0.53718972){\makebox(0,0)[lb]{\smash{\tiny 0.1}}}%
    \put(0.94580938,0.59545699){\makebox(0,0)[lb]{\smash{\tiny 0.2}}}%
    \put(0.0411586,0.03744143){\makebox(0,0)[lb]{\smash{\tiny 0}}}%
    \put(0.14829519,0.03744143){\makebox(0,0)[lb]{\smash{\tiny 2}}}%
    \put(0.2555578,0.03744143){\makebox(0,0)[lb]{\smash{\tiny 4}}}%
    \put(0.36269439,0.03744143){\makebox(0,0)[lb]{\smash{\tiny 6}}}%
    \put(0.03088371,0.05348079){\makebox(0,0)[lb]{\smash{\tiny 0}}}%
    \put(0.03088371,0.12590723){\makebox(0,0)[lb]{\smash{\tiny 1}}}%
    \put(0.03088371,0.19820947){\makebox(0,0)[lb]{\smash{\tiny 2}}}%
    \put(0.03088371,0.27051051){\makebox(0,0)[lb]{\smash{\tiny 3}}}%
    \put(0.19979564,0.01952247){\makebox(0,0)[lb]{\smash{(5a)}}}%
    \put(0.4406342,0.06513424){\makebox(0,0)[lb]{\smash{\tiny $-0.1$}}}%
    \put(0.4406342,0.11613337){\makebox(0,0)[lb]{\smash{\tiny $-0.05$}}}%
    \put(0.4406342,0.16700829){\makebox(0,0)[lb]{\smash{\tiny 0}}}%
    \put(0.4406342,0.21800681){\makebox(0,0)[lb]{\smash{\tiny 0.05}}}%
    \put(0.4406342,0.26900684){\makebox(0,0)[lb]{\smash{\tiny 0.1}}}%
    \put(0.54639193,0.03744143){\makebox(0,0)[lb]{\smash{\tiny 0}}}%
    \put(0.65352852,0.03744143){\makebox(0,0)[lb]{\smash{\tiny 2}}}%
    \put(0.76079142,0.03744143){\makebox(0,0)[lb]{\smash{\tiny 4}}}%
    \put(0.86792802,0.03744143){\makebox(0,0)[lb]{\smash{\tiny 6}}}%
    \put(0.53611734,0.05348079){\makebox(0,0)[lb]{\smash{\tiny 0}}}%
    \put(0.53611734,0.12590723){\makebox(0,0)[lb]{\smash{\tiny 1}}}%
    \put(0.53611734,0.19820947){\makebox(0,0)[lb]{\smash{\tiny 2}}}%
    \put(0.53611734,0.27051051){\makebox(0,0)[lb]{\smash{\tiny 3}}}%
    \put(0.70916437,0.01952247){\makebox(0,0)[lb]{\smash{(5b)}}}%
    \put(0.94586783,0.05723997){\makebox(0,0)[lb]{\smash{\tiny $-0.2$}}}%
    \put(0.94586783,0.11212398){\makebox(0,0)[lb]{\smash{\tiny $-0.1$}}}%
    \put(0.94586783,0.16700829){\makebox(0,0)[lb]{\smash{\tiny 0}}}%
    \put(0.94586783,0.2220171){\makebox(0,0)[lb]{\smash{\tiny 0.1}}}%
    \put(0.94586783,0.27690111){\makebox(0,0)[lb]{\smash{\tiny 0.2}}}%
  \end{picture}%
\endgroup%
  \caption{\label{fig:fluid_flow_2} Comparing the approximations of the second Oseledets vector $w_2(Tx)$ at time $t=8$ with the push-forward of the approximations at time $t=0$.  Those labelled (a) are the push-forwards $A(x,1)w_2^{(4)}(x)$ whilst those labelled (b) are independently computed approximations $w_2^{(4)}(Tx)$ of $w_2(Tx)$.  The algorithms used are as follows: (1) Algorithm \ref{alg:svd2}, (2) Algorithm \ref{alg:dichproj1}, (3) Algorithm \ref{A2}, (4) Algorithm \ref{alg:ginelli} and (5) Algorithm \ref{alg:wolfe}.}
\end{figure}
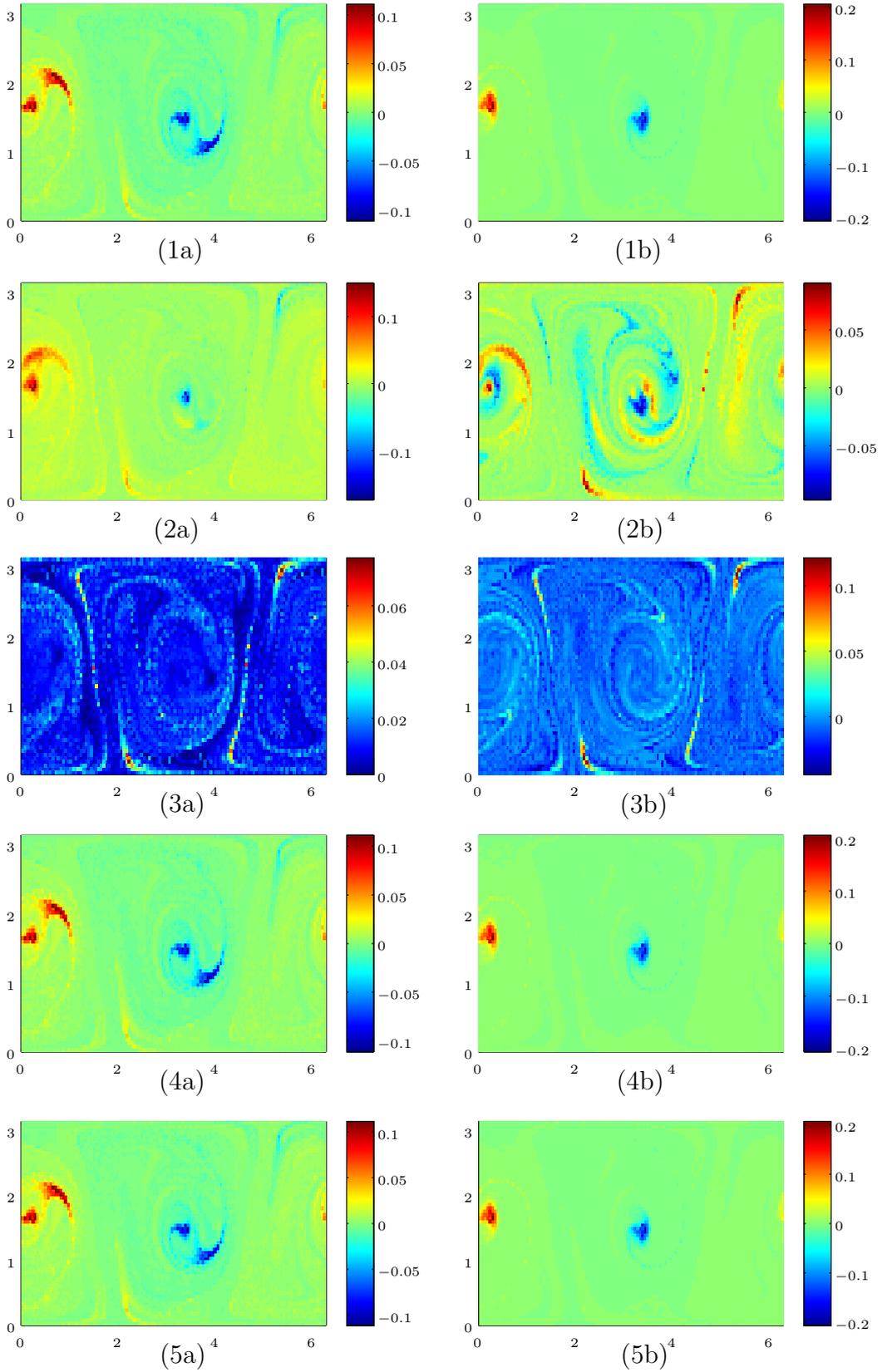

The results of these numerical experiments are shown in Figures \ref{fig:fluid_flow_1} and \ref{fig:fluid_flow_2}.  Recall that in this setting, the cocycle $A(x,n)$ is a cocycle of discretised Perron-Frobenius operators acting on piecewise constant functions defined on $Y$;  we identify these piecewise constant functions (with 7200 pieces) with vectors in $\mathbb{R}^{7200}$.   Figure \ref{fig:fluid_flow_1} first shows the approximations of the second Oseledets vector $w_2(x)$ at time $t=0$.  In this setting the Oseledets vectors locate \emph{coherent structures}:  Figure \ref{fig:fluid_flow_2} compares the push-forward of the approximations in Figure \ref{fig:fluid_flow_1} with independently computed approximations of $w_2(Tx)$ - the second Oseledets vector at time $t=8$.

In this study the data sample is insufficiently long for Algorithm \ref{A2} to work effectively, but the other algorithms produce similar results.  A visual inspection of Figure \ref{fig:fluid_flow_2} shows that the highlighted structures are approximately equivariant/coherent.

\section{Conclusion}

We introduced two new methods for computing Oseledets subspaces:  one based on singular value decompositions and the other based on dichotomy projectors.
We also reviewed recent methods by Ginelli \emph{et al.} \cite{Ginelli2007} and Wolfe and Samelson \cite{Wolfe2007}, and presented an improvement to both of these approaches that intelligently selected initial bases when only short time series were available to compute with.
Finally, we carried out a comparative numerical investigation involving all four methods.

Generally speaking, we found that Algorithms \ref{alg:svd2}, \ref{alg:ginalt}, and \ref{alg:wolfe} outperformed the dichotomy projector methods (Algorithms \ref{alg:dichproj1} and \ref{A2}) when limited to moderate amounts of data were available, however, the dichotomy projector methods performed very well when long time series of matrices were available.
The Ginelli approach (in particular the improved Algorithm \ref{alg:ginalt}) also worked very well with long time series.

The improvements made to Algorithm \ref{alg:svd} in Section \ref{sec:svd2} (namely the orthogonalisation step in Algorithm \ref{alg:svd2}) produced an algorithm that could take advantage of longer matrix sequences and return very accurate results. { Of course, for each Algorithm} one must choose the associated parameters sensibly to ensure good results.

When only a short to moderate time series was available, we found mixed results in terms of the best algorithm.
The improved SVD approach (Algorithm \ref{alg:svd2}) was best for low to moderate length time series in the exact Toy model, while the improved Ginelli (Algorithm \ref{alg:ginalt}) and improved Wolfe {(Algorithm \ref{alg:wolfe})} were marginally best in terms of equivariance and expansion rate, respectively for the 2-disk model.
Each of these three algorithms produced similar results in the fluid-flow system.

Choosing appropriate parameters for a particular application can be difficult.  In the present review, good values were chosen by {educated} experimentation. On the other hand, the dichotomy projector methods, Algorithms \ref{alg:dichproj1} and \ref{A2}, use parameters ($\Lambda^\text{right}$ and $\Lambda^\text{left}$) which can be chosen in a deterministic manner - by estimating Lyapunov exponents, which is a {reasonably robust} numerical procedure.  Furthermore, {a} rigorous error approximation exists for Algorithm \ref{A2}, a feature currently lacking for Algorithms \ref{alg:svd2}, \ref{alg:ginalt} and \ref{alg:wolfe}.

The memory footprint of each approach scales quite differently with dimension.  In Section \ref{sec:fluidflow}, Algorithms \ref{alg:svd2} and \ref{alg:wolfe} { could take advantage of} the sparseness of the $d\times d$ generating matrices of the cocycle. { However, } since $A(x,n)$ is formed by matrix multiplication, for large $n$ the matrix $A(x,n)$ becomes dense and may require memory of the order of $d^2$ floating point numbers.  The dichotomy projector Algorithms \ref{alg:dichproj1} and \ref{A2}, need to form an $Nd \times (N+1)d$ matrix, but with sparse generating matrices, this requires memory much less than of the order of $d^2$ floating point numbers. Algorithm \ref{alg:ginalt} has the most conservative memory footprint, but depends on its initialisation parameter $M'$ and the Oselelets subspace number $j$.  If $M'$ is large, then $A(T^{-M}x,M')$ in Step 1 can become dense and require ${ \mathcal O (d^2)}$ floating point numbers.  On the other hand, the stationary Lyapunov basis requires $jd$ floating point numbers to be stored, so if $j\approx d$ this can becomes comparable to $d^2$.

Section \ref{sec:fluidflow} involves non-invertible generating matrices and { apart from} Algorithm \ref{A2}, each approach succeeded in producing a { reasonable} solution, showing that the Algorithms { can perform well in the non-invertible setting. Continuing with the non-invertible situation, if one wishes to approximate} Oseledets subspaces corresponding to negative numbers with very large magnitudes ($\lambda_j \approx -\infty$){, then} Algorithms \ref{alg:svd2} and \ref{alg:wolfe} { may struggle as} rapidly contracting directions (relative to the dominant direction corresponding to $\lambda_1$) are quickly squashed during the matrix multiplication used to approximate $A(x,n)$ leading to inaccurate numerical representation of $A(x,n)$.

The dichotomy projector approaches of Algorithms \ref{alg:dichproj1} and \ref{A2} are able to compute Oseledets subspaces corresponding to smaller, sub-dominant Lyapunov exponents $\lambda_3,\lambda_4, \ldots$ provided larger amounts of cocycle data is available.  However, if $\lambda_j \approx -\infty$, we are forced to choose $\Lambda^\text{right}$ or $\Lambda^\text{left} \approx -\infty$ which means either problem \eqref{al1} or \eqref{al2} (in Algorithm \ref{alg:dichproj1} which also feature in Algorithm \ref{A2}) are ill-conditioned and fail.

The same problem manifests itself in Algorithm \ref{alg:ginalt}, even though it is able to compute Oseledets subspaces corresponding to smaller, sub-dominant Lyapunov exponents.  The sum of the {logarithm of} the diagonal entries of the $j\times j$ generating matrices of the cocycle $R(x,n)$ average to the {logarithmic} expansion rate of the $j$-parallelepiped formed at $x$ by the stationary Lyapunov vectors $s_1^{(\infty)}(x),\ldots,s_j^{(\infty)}(x)$ as it is pushed-forward.  { Thus}, the { logarithm of the} $i$th diagonal entry of the generating matrices of $R(x,n)$ { has a time average of $\lambda_i$} \cite{benettin1980} {and} if $\lambda_j \approx -\infty$, $R(x,n)$ will feature diagonal entries close to, or equal to zero and $R(x,n)^{-1}$ won't exist.

In summary, Algorithms \ref{alg:svd2} and \ref{alg:wolfe} are best suited to situations with limited cocycle data when one of the most dominant Oseledets subspaces is desired.  Algorithm \ref{alg:ginalt} can be applied to both limited and { high} data situations by choosing $M'$ appropriately, and can compute most Oseledets subspaces provided their Lyapunov exponents are well-conditioned.  If ample data is available and information regarding the system is lacking (making the choice of parameters for the other approaches difficult), the approaches of Algorithms \ref{alg:dichproj1} and \ref{A2} may be preferred for their {relatively} deterministic parameter selection.

\bibliographystyle{plain}

\end{document}